\documentclass[a4paper,12pt,twoside]{article}
\usepackage{amsmath}
\usepackage{amsfonts}
\usepackage{amssymb}
\usepackage{mathrsfs}
\usepackage[pdfborder={0 0 0}]{hyperref}
\usepackage[utf8]{inputenc}
\usepackage{amssymb}          
\usepackage{latexsym}         
\usepackage{amsxtra}          
\usepackage[utf8]{inputenc}
\usepackage{eucal}            
\usepackage{bbm}              
\usepackage{longtable}        
\usepackage{exscale}          
\usepackage{theorem}          
\usepackage{graphicx}         
\usepackage[sort]{cite}       
\usepackage{mathrsfs}         
\usepackage{array}            
\usepackage{stmaryrd}         
\usepackage{paralist}         
\usepackage{color}
\usepackage{wasysym}          
\usepackage{mathtools}        
\usepackage{fancyhdr}         
\usepackage{ifthen}           

\usepackage{oldgerm}
\usepackage{color}

\newcommand{\Ibb}[1]{ {\rm I\ifmmode\mkern -3.6mu\else\kern -.2em\fi#1}}
\newcommand{\ibb}[1]{\leavevmode\hbox{\kern.3em\vrule
     height 1.2ex depth -.3ex width .2pt\kern-.3em\rm#1}}
\newcommand{\Cl}{\mathbb{C}} 
\newcommand{\Rl}{\mathbb{R}} 
\newcommand{\Nl}{\mathbb{N}} 
\newcommand{\Zl}{\mathbb{Z}}

\definecolor{lightgray}{rgb}{0.8,0.8,0.8}

\newcommand{\IntSymbol}{\underline{\Symbol}}

\newcommand{\te}{\theta}

\newcommand{\la}{\lambda}

\newcommand{\eps}{\varepsilon}

\newcommand{\A}{\mathcal{A}}

\newcommand{\Hil}{\mathcal{H}}

\newcommand{\E}{\mathcal{E}}

\newcommand{\F}{\mathcal{F}}

\newcommand{\Ss}{\mathscr{S}}   

\def\bm{\boldsymbol{m}}

\def\brho{\boldsymbol{\rho}}

\newcommand{\qn}{{\rm q}}
\newcommand{\rn}{{\rm r}}
\newcommand{\GL}{\text{GL}}

\newcommand{\supp}{\text{supp}\,}

\renewcommand{\mathbb}[1]{\mathbbm{#1}}


\newcommand{\refitem}[1] {~\textit{\ref{#1}.)}}

\numberwithin{equation}{section}

\allowdisplaybreaks

\newcounter{comment}

\theoremheaderfont{\normalfont\bfseries}
\theorembodyfont{\itshape}
\newtheorem{lemma}{Lemma}[section]
\newtheorem{proposition}[lemma]{Proposition}
\newtheorem{theorem}[lemma]{Theorem}
\newtheorem{corollary}[lemma]{Corollary}
\newtheorem{definition}[lemma]{Definition}

\theorembodyfont{\rmfamily}

\newtheorem{remark}[lemma]{Remark}

\makeatletter
\newcommand\qedsymbol{\hbox{$\boxempty$}}
\newcommand\qed{\relax\ifmmode\boxempty\else
  {\unskip\nobreak\hfil\penalty50\hskip1em\null\nobreak\hfil\qedsymbol
  \parfillskip=\z@\finalhyphendemerits=0\endgraf}\fi}

\newcommand\subqedsymbol{\hbox{$\triangledown$}}
\newcommand\subqed{\relax\ifmmode\triangledown\else
  {\unskip\nobreak\hfil\penalty50\hskip1em\null\nobreak\hfil\subqedsymbol
  \parfillskip=\z@\finalhyphendemerits=0\endgraf}\fi}
\makeatother

\newenvironment{proof}[1][{}]{\par\noindent Proof{#1}. }{\qed\medskip}

\newenvironment{remarklist}{\begin{compactenum}[\itshape i.)]}{\end{compactenum}}

\newenvironment{lemmalist}{\begin{compactenum}[\itshape i.)]}{\end{compactenum}}
\newenvironment{theoremlist}{\begin{compactenum}[\itshape i.)]}{\end{compactenum}}
\newenvironment{propositionlist}{\begin{compactenum}[\itshape i.)]}{\end{compactenum}}
\newenvironment{definitionlist}{\begin{compactenum}[\itshape i.)]}{\end{compactenum}}
\newenvironment{corollarylist}{\begin{compactenum}[\itshape i.)]}{\end{compactenum}}

\newcommand{\D}          {\operatorname{\mathrm{d}}}

\newcommand{\RE}         {\mathsf{Re}}
\newcommand{\IM}         {\mathsf{Im}}
\newcommand{\Unit}       {\mathbb{1}}

\newcommand{\at}[1]      {\big|_{#1}}

\newcommand{\argument}   {\,\cdot\,}

\newcommand{\id}         {\operatorname{\mathsf{id}}}

\newcommand{\Trans}      {{\mathrm{\scriptscriptstyle{T}}}}

\newcommand{\SP}[1]      {\left\langle{#1}\right\rangle}

\DeclareMathSymbol\dAlembert  {\mathord}{AMSa}{"03}

\newcommand{\Fun}[1][{k}] {\mathscr{C}^{#1}}
\newcommand{\Cinfty}     {\Fun[\infty]}

\newcommand{\group}[1]   {\mathrm{#1}}

\newcommand{\norm}[1]    {\left\|{#1}\right\|}
\newcommand{\supnorm}[1] {\left\|{#1}\right\|_\infty}

\newcommand{\seminorm}[1][{}] {\operatorname{\mathrm{p}}_{#1}}
\newcommand{\halbnorm}[1]   {\operatorname{\mathrm{#1}}}

\newcommand{\Stetig}         {\mathscr{C}}
\newcommand{\Ball}           {\mathrm{B}}

\newcommand{\Schwartz}       {\mathscr{S}}

\newcommand{\order}[1]       {\boldsymbol{#1}}
\newcommand{\type}[1]        {\boldsymbol{#1}}
\newcommand{\Symbol}         {\mathrm{S}}
\newcommand{\symnorm}[1]     {\norm{#1}}

\newlength{\dinwidth}
\newlength{\dinmargin}
\numberwithin{equation}{section}

\title{Strict deformation quantization of locally convex algebras and modules}
\author{Gandalf Lechner\footnote{University of Vienna, Department of
    Physics, Boltzmanngasse 5, A-1090 Vienna, Austria, e-mail:
    gandalf.lechner@univie.ac.at}\; and \addtocounter{footnote}{5}
  Stefan Waldmann\footnote{
    University of Freiburg, Physikalisches Institut, Hermann Herder Str. 3, 
    D-71904 Freiburg, Germany, e-mail: stefan.waldmann@physik.uni-freiburg.de}
}
\date{September 2011}

\begin{document}
\maketitle

\begin{abstract}
    In this work various symbol spaces with values in a sequentially
    complete locally convex vector space are introduced and discussed.
    They are used to define vector-valued oscillatory integrals
    which allow to extend Rieffel's strict deformation quantization
    to the framework of sequentially complete locally convex
    algebras and modules with separately continuous products and module
    structures, making use of polynomially bounded actions of $\mathbb{R}^n$.
    Several well-known integral formulas for star products are
    shown to fit into this general setting, and a new class of examples
    involving compactly supported $\mathbb{R}^n$-actions on $\mathbb{R}^n$
    is constructed.
\end{abstract}

\tableofcontents

\section{Introduction}

Deformation quantization as introduced in \cite{bayen.et.al:1978a}
comes in several different flavours: in formal deformation
quantization one deforms the commutative pointwise product of the
Poisson algebra of smooth functions on a Poisson manifold into a
noncommutative \emph{star product} as a formal associative deformation
in the sense of Gerstenhaber \cite{gerstenhaber:1964a} with
deformation parameter $\hbar$. Here the general existence and
classification for arbitrary Poisson manifolds is known and follows
from Kontsevich's formality theorem \cite{kontsevich:2003a}, see
\cite{waldmann:2007a} for a introductory textbook.

However, for many reasons formal deformations are not sufficient: for
the original application to quantum mechanics one has to treat $\hbar$
as a positive number and not just as a formal parameter. But also
applications beyond quantum theory require a more analytic framework.
In particular, deformation quantization provides fundamental examples
in noncommutative geometry where a $C^*$-algebraic formulation is
needed. 

In \cite{rieffel:1993a}, Rieffel introduced a very general way to
construct $C^*$-algebraic deformations based on a strongly continuous
action of $\mathbb{R}^d$ on a $C^*$-algebra $\mathcal{A}$. For the
smooth vectors $\mathcal{A}^\infty$ with respect to the action a
product formula based on an oscillatory integral was established,
generalizing the well-known Weyl quantization of $\mathbb{R}^{2n}$. In
a second step, a matching $C^*$-norm is constructed leading to a
continuous field of $C^*$-algebras over the parameter space of $\hbar
\in \mathbb{R}$. This construction and variants of it have by now
found many applications in noncommutative geometry \cite{connes:1994a,
  rieffel:1993a, GayralGraciaBondiaIochumSchuckerVarilly:2004} and
quantum physics, in particular in the context of quantum field theory
on noncommutative spacetimes
\cite{doplicher.fredenhagen.roberts:1995a, bahns.waldmann:2007a,
  heller.neumaier.waldmann:2007a, GrosseLechner:2008,
  BuchholzLechnerSummers:2010}.

While for the construction of deformed $C^*$-algebras
Rieffel's work is sufficient, it turns out that the first step of
deforming the smooth vectors $\mathcal{A}^\infty$ is of interest
already for it's own sake: Rieffel worked with a Fr\'echet algebra
with an isometric action.

It is this situation which we want to generalize in various directions
in the present work. First the restriction to a Fr\'echet algebra has
to be overcome as there are several examples of interest which do not
fall into this class. When interested in noncommutative spacetimes, a
smooth structure in form of a deformation of the smooth functions is
needed for many reasons. Thus we are interested in e.g.\ deformations
of $\Cinfty_0(M)$. Moreover, together with a deformation of algebras
one is interested in a possible deformation of modules as well. In the
above example one might also be interested in a corresponding
deformation of the distribution spaces $\Cinfty_0(M)'$. Hence we
clearly have to pass beyond a Fr\'echet situation. Here we face
several new phenomena: first of all products or module structures may
be separately continuous without being continuous. In fact, there are
many natural examples like this. Second, sequentially complete locally
convex spaces need not be complete, with the distribution spaces as
the most prominent examples. Third, the restriction to isometric
actions, which is natural in the original $C^*$-setting, seems to be
too restrictive in a more general locally convex framework. Again,
many examples of interest show that one has to overcome this
restriction.

As is well known, scalar-valued oscillatory integrals can be defined
for more general functions than smooth
functions with bounded derivatives -- here the Hörmander symbols are a
natural candidate. Thus we will adapt the notion of a symbol to the
vector-valued case and study oscillatory integrals. These will be
needed to handle actions of $\mathbb{R}^n$ which are not isometric
but satisfy certain polynomial growth conditions. Compared to the
scalar case the new feature is that for every continuous seminorm (of
a defining system of seminorms) of the target space we have to allow
for a specific growth. The examples show that we cannot expect to have
a uniform growth for all seminorms.

The main result of this work is the construction of a Rieffel
deformation for a sequentially complete locally convex algebra with a
separately continuous product with respect to a smooth polynomially
bounded action of $\mathbb{R}^n$ by automorphisms. Analogously, we
give the corresponding deformation for a sequentially complete locally
convex module with separately continuous module structure, provided
the module structure is covariant for the $\mathbb{R}^n$-action.  To
this end we introduce the relevant symbol spaces and their oscillatory
integrals based on a Riemann integral as we want to include
sequentially complete spaces as well. This part is clearly of
independent interest.  We discuss several known examples within this
framework and provide one new example of an action of $\mathbb{R}^n$
with compact support. A priori one can only guarantee exponential
bounds for the derivatives of such an action, but by a particular
construction we achieve polynomial growth behaviour. Actions of this
type are needed in models of locally noncommutative spacetimes as
introduced in \cite{bahns.waldmann:2007a,
  heller.neumaier.waldmann:2007a}. In fact, the wish to have a smooth
version of \cite{heller.neumaier.waldmann:2007a} was one of the main
motivations to develop the above generalization of Rieffel's original
work as a compactly supported action cannot be expected to be
isometric for the seminorms of smooth functions. In the diploma thesis
\cite[Sect.~6.2]{heller:2006a} some aspects of the vector-valued
oscillatory integrals were already anticipated.

It should be mentioned that there are still generalizations
possible. One important step beyond Rieffel's original setting is to
include actions of other Lie groups than $\mathbb{R}^n$. Here one
first needs to find an analogue of Weyl quantization which then serves
as universal deformation formula. This point of view was taken in the
works of Bieliavsky et al., see
e.g.~\cite{bieliavsky.detournay.spindel:2009a,
  bieliavsky.bonneau.maeda:2007a, bieliavsky:2002a,
  bieliavsky.massar:2001b}. While these works mainly deal with the
$C^*$-algebraic deformation, in a more recent work
\cite{bieliavsky.gayral:2011a}, Bieliavsky and Gayral discuss also
deformation aspects of Fr\'echet algebras based on symbol spaces and
oscillatory integrals similar to ours. We leave it to a future
investigation of whether their construction can be extended beyond the
Fr\'echet case: in principal this looks very promising.

The paper is organized as follows. In
Section~\ref{sec:VectorValuedSymbols} we introduce vector-valued
symbols in the spirit of Hörmander symbols. However, the order as well
as the type of the symbol may depend on the seminorm of the target
space, a generalization needed to deal with the examples discussed
later. We introduce in detail various symbol spaces, investigate the
continuity properties of the usual algebraic manipulations, and show
that the affine symmetries of the domain give continuous group actions
on the symbol spaces. In particular, the translations act smoothly. In
Section~\ref{section:integrals} we discuss the oscillatory
integrals. Our approach follows the usual scalar case with the
technical complication that we have to deal with many seminorms on the
target instead of one. Thus a careful investigation of the polynomial
growth is presented. The integrals are based on a Riemann integral for
the smooth compactly supported functions as we want to include targets
which are only sequentially complete. After this preparation,
Section~\ref{section:rieffel} is devoted to the deformation
program. Based on the developed oscillatory integrals we extend
Rieffel's construction to actions of $\mathbb{R}^n$ by automorphisms
on sequentially complete locally convex algebras with separately
continuous products and their modules. In case where the products are
continuous also the resulting deformed products are
continuous. Finally, Section~\ref{section:examples} contains the
examples. First we discuss the usual action of $\mathbb{R}^{2n}$ on
itself by translations and the induced action on various function
spaces. Here in particular the scalar symbol spaces, the Schwartz
space, and certain distribution spaces are discussed. This way we show
that the well-known Weyl product formula, being defined pointwise for
these spaces, can be understood as resulting from the oscillatory
integral formulas. This is a nontrivial statement as in all cases the
action is \emph{not} isometric. The second example will be used in a
future project for the construction of locally noncommutative
spacetimes and corresponding quantum field theory models. It provides
 an action of $\mathbb{R}^n$ on $\mathbb{R}^n$
with compact support inside a given compact subset such that the
induced action on the smooth functions is polynomially bounded. The
difficulty is to pass from a trivially given exponential growth of the
derivatives to a polynomial growth.
\\
%
%
\\
\textbf{Acknowledgements:} It is a pleasure to thank Pierre
Bieliavsky, Victor Gayral, and Ryszard Nest for various
discussions. We gratefully acknowledge the hospitality and
support extended to us by the Erwin-Schr\"odinger Institute (S.W.)
and the University of Freiburg (G.L.) during different stages of
this work. Finally, we would like to thank the participants of
the Scalea conference 2011, where the results have been presented,
for fruitful discussions. The work of G.L. is supported by the
FWF project P22929--N16 ``Deformations of Quantum Field Theories''.

\section{Vector-valued symbols}
\label{sec:VectorValuedSymbols}

In this section we introduce vector-valued symbols as smooth functions
$F:\Rl^n\to V$ which take values in a sequentially complete locally
convex vector space $V$ over $\mathbb{C}$ and satisfy polynomial
growth conditions for all their derivatives. At the present stage, all
definitions also make sense for a $V$ being a real vector space only,
but as soon as we discuss the oscillatory integral, complex phases
will enter the game. In the case of scalar functions, with target
space $V=\Cl$, these spaces are closely related to Hörmander's symbol
classes \cite[Section~7.8]{Hormander:1990}, see also
\cite{Hormander:1971, GrigisSjostrand:1994}.
\begin{definition}
    \label{definition:SymbolSeminorm}%
    Let $V$ be a sequentially complete locally convex space and let $F
    \in \Fun(\mathbb{R}^n, V)$ where $k \in \mathbb{N}_0 \cup
    \{+\infty\}$. Then for every continuous seminorm $\halbnorm{q}$ on
    $V$, every multiindex $\mu \in \mathbb{N}_0^n$ with $|\mu| \le k$,
    and $m, \rho \in \mathbb{R}$, we define
    \begin{equation}
        \label{eq:SymbolSeminorm}
        \symnorm{F}_{\halbnorm{q}, \mu}^{m, \rho}
        :=
        \sup_{x \in \mathbb{R}^n}
        \left(1 + \norm{x}^2\right)^{-\frac{1}{2}(m - \rho|\mu|)}
        \halbnorm{q}\left(
            \frac{\partial^{|\mu|}F}{\partial x^\mu}(x)
        \right)
        \in
        [0, +\infty].
    \end{equation}
\end{definition}

The quantity $\symnorm{F}_{\halbnorm{q}, \mu}^{m, \rho}$ controls how
the $\mu$-th derivative of $F$ grows at infinity with respect to the
seminorm $\halbnorm{q}$, compared to a polynomial of order $m$: A
polynomial $P\in V[x_1, \ldots, x_n]$ of order $m$ clearly satisfies
$\symnorm{P}^{m,1}_{\qn,\mu}<\infty$ for all multiindices $\mu$, and
$\symnorm{F}^{m,\rho}_{\qn,\mu}<\infty$ with $\rho>1$ (respectively
$\rho<1$) indicates that the derivatives of $F$ grow slower
(respectively faster) than those of a polynomial.

In order to define the symbol classes we now fix a defining system
$\mathcal{Q}$ of continuous seminorms on $V$. The canonical choice is
of course to take all continuous seminorms, but sometimes it will be
advantageous to take only a small and manageable system. The following
definitions will formally depend on this choice, but the effect is
only minor. Later we will see that the oscillatory integrals are
independent of the particular choice of $\mathcal{Q}$.

We assign to every $\halbnorm{q} \in \mathcal{Q}$ real numbers
$\order{m}(\qn)$ and $\type{\rho}(\qn)$. The corresponding map
\begin{equation}
    \label{eq:OrderMap}
    \order{m}\colon \mathcal{Q} \ni \halbnorm{q}
    \; \mapsto \;
    \order{m}(\halbnorm{q}) \in \mathbb{R}
\end{equation}
will be called an \emph{order} for $\mathcal{Q}$ and the map
\begin{equation}
    \label{eq:TypeMap}
    \type{\rho}\colon \mathcal{Q} \ni \halbnorm{q}
    \; \mapsto \;
    \type{\rho}(\halbnorm{q}) \in \mathbb{R}
\end{equation}
is referred to as a \emph{type} for $\mathcal{Q}$.

The natural ordering of $\mathbb{R}$ induces one for the set of all
orders as well as for the set of all types. For two orders
$\order{m}$, $\order{m}'$ we write
\begin{equation}
    \label{eq:OrdermKleinerOrdermPrime}
    \order{m} \le \order{m}'
    \quad
    \textrm{if}
    \quad
    \order{m}(\halbnorm{q}) \le \order{m}'(\halbnorm{q})
\end{equation}
for all $\halbnorm{q}\in\mathcal{Q}$.  Then ``$\le$'' makes the set of
all orders a directed set, and we also write $\order{m} < \order{m}'$
if $\order{m}(\halbnorm{q}) < \order{m}'(\halbnorm{q})$ for all
$\halbnorm{q} \in \mathcal{Q}$.

If we set $\order{m}(\halbnorm{q}) := m \in \mathbb{R}$ for all
$\halbnorm{q} \in \mathcal{Q}$ we get an order, called the
\emph{constant order}, and analogously the constant type
$\type{\rho}(\halbnorm{q}) = \rho \in \mathbb{R}$. It will turn out
that this is usually too restrictive and we need more freedom in
choosing an order and a type. More generally, an order $\order{m}$ is
called \emph{bounded} from above or from below by some number $\alpha$
or $\beta$ if for all $\halbnorm{q} \in
\mathcal{Q}$ one has $\order{m}(\halbnorm{q}) \le \alpha$ or
$\order{m}(\halbnorm{q}) \ge \beta$, respectively.

In the following it will be reasonable to ask for the condition
\begin{equation}
    \label{eq:orderOnMultipleSeminorm}
    \order{m}(c\halbnorm{q}) = \order{m}(\halbnorm{q})
\end{equation}
whenever $\halbnorm{q}, c\halbnorm{q} \in \mathcal{Q}$ for a constant
$c > 0$. In particular, for a Banach space $V$ we usually take only
the constant orders by specifying their value on the norm. If $V$ is a
Fr\'echet space, we will usually take a countable system
$\mathcal{Q}$, and often consider unbounded orders.

\begin{definition}[Symbols]
    \label{definition:Symbols}%
    Let $V$ be a sequentially complete locally convex space with
    defining system of seminorms $\mathcal{Q}$, and let $\order{m}$
    and $\type{\rho}$ be an order and a type for $\mathcal{Q}$.
    \begin{definitionlist}
    \item \label{item:SymbolOrdermTyperho} A function $F \in
        \Cinfty(\mathbb{R}^n, V)$ is called a symbol of order
        $\order{m}$ and type $\type{\rho}$ if for all
        $\halbnorm{q}\in\mathcal{Q}$ and $\mu \in \mathbb{N}_0^n$ one
        has
        \begin{equation}
            \label{eq:SymbolNorm}
            \symnorm{F}^{\order{m}, \type{\rho}}_{\halbnorm{q}, \mu}
            := \symnorm{F}^{\order{m}(\halbnorm{q}),
              \type{\rho}(\halbnorm{q})}_{\halbnorm{q}, \mu}
            < \infty.
        \end{equation}
    \item \label{item:SymbolSpace} The set of all symbols of order
        $\order{m}$ and type $\type{\rho}$ is denoted by
        $\Symbol^{\order{m}, \type{\rho}}(\mathbb{R}^n, V)$.
    \end{definitionlist}
\end{definition}

Sometimes we will abbreviate the space of symbols simply by
$\Symbol^{\order{m}, \type{\rho}}(V)$ or even by $\Symbol^{\order{m},
  \type{\rho}}$ if $V$ is clear from the context. However, note that
$\Symbol^{\order{m}, \type{\rho}}(\mathbb{R}^n, V)$ still depends on
the choice of $\mathcal{Q}$.

It is clear that the $\symnorm{\,\cdot\,}_{\halbnorm{q},
  \mu}^{\order{m}, \type{\rho}}$ are seminorms on
$\Symbol^{\order{m},\type{\rho}}(\mathbb{R}^n, V)$. We will endow the
space of symbols with the corresponding locally convex topology,
called the \emph{$\Symbol^{\order{m}, \type{\rho}}$-topology}. This
makes $\Symbol^{\order{m}, \type{\rho}}(\mathbb{R}^n, V)$ a
\emph{Hausdorff} locally convex space since $V$ is Hausdorff and the
prefactor $(1 + \norm{x}^2)^{-\frac{1}{2}(\order{m}(\halbnorm{q}) -
  \type{\rho}(\halbnorm{q})|\mu|)}$ is nowhere vanishing.

In the following, the smooth functions $\Cinfty(\Rl^n,V)$ will always
be equipped with the topology determined by all seminorms
\begin{align}\label{eq:CinftyNorms}
    \halbnorm{p}_{K,\ell,\qn}(F)
    :=
    \sup_{x\in K\atop |\mu|\leq \ell}
    \qn(\partial_x^\mu F(x)),
\end{align}
where $K$ runs over compact subsets of $\Rl^n$, $l\in\Nl_0$, and
$\qn\in\mathcal{Q}$, and the smooth functions of compact support
$\Cinfty_0(\Rl^n,V)$ carry their usual inductive limit topology.
\begin{proposition}
    \label{proposition:SymbolFirstProperties}%
    Let $V$ be a sequentially complete locally convex space with a
    defining system of seminorms $\mathcal{Q}$, and let $\order{m}$ and
    $\type{\rho}$ be an order and a type for $\mathcal{Q}$.
    \begin{propositionlist}
    \item \label{item:SymbolInclusion} We have continuous inclusions
        \begin{equation}
            \label{eq:SymbolInclusions}
            \Cinfty_0(\mathbb{R}^n, V)
            \longrightarrow
            \Symbol^{\order{m}, \type{\rho}}(\mathbb{R}^n, V)
            \longrightarrow
            \Cinfty(\mathbb{R}^n, V).
        \end{equation}
    \item \label{item:SymbolsDenseInCinfty} The symbols
        $\Symbol^{\order{m}, \type{\rho}}(\mathbb{R}^n, V)$ are dense
        in $\Cinfty(\mathbb{R}^n, V)$.
    \item \label{item:SymbolsComplete} The symbols
        $\Symbol^{\order{m}, \type{\rho}}(\mathbb{R}^n, V)$ are
        sequentially complete and even complete if $V$ is complete.
    \item \label{item:SymbolsInSymbols} For $\order{m} \le \order{m}'$
        and $\type{\rho} \ge \type{\rho}'$ we have the continuous
        inclusion
        \begin{equation}
            \label{eq:SymbolInSymbol}
            \Symbol^{\order{m}, \type{\rho}}(\mathbb{R}^n, V)
            \longrightarrow
            \Symbol^{\order{m}', \type{\rho}'}(\mathbb{R}^n, V).
        \end{equation}
        More precisely, we have for all $F \in \Symbol^{\order{m},
          \type{\rho}}$, all $\halbnorm{q} \in \mathcal{Q}$, and all
        $\mu \in \mathbb{N}_0^n$
        \begin{equation}
            \label{eq:SymbolNormViaSymbolNorm}
            \symnorm{F}^{\order{m}', \type{\rho}'}_{\halbnorm{q}, \mu}
            \le
            \symnorm{F}^{\order{m}, \type{\rho}}_{\halbnorm{q}, \mu}.
        \end{equation}
    \end{propositionlist}
\end{proposition}
\begin{proof}
    Clearly, we have a set-theoretic inclusion in
    \eqref{eq:SymbolInclusions} as compactly supported functions decay
    fast enough to have finite symbol norms \eqref{eq:SymbolSeminorm}
    for any choices of the orders or types. With a compact set $K
    \subseteq\Rl^n$, and $F \in\Cinfty_K(\mathbb{R}^n, V)$, we get
    \begin{align*}
        \symnorm{F}^{\order{m}, \type{\rho}}_{\halbnorm{q}, \mu}
        &=
        \max_{x \in K}
        \left(
            1 + \norm{x}^2
        \right)^{-\frac{1}{2}
          (\order{m}(\halbnorm{q}) - \type{\rho}(\halbnorm{q})|\mu|)
        }
        \halbnorm{q}\left(
            \frac{\partial^{|\mu|} F}{\partial x^\mu}(x)
        \right)
	 \\
        &\le
          \max_{x \in K}
          \left(
              1 + \norm{x}^2
          \right)^{-\frac{1}{2}
            (\order{m}(\halbnorm{q}) - \type{\rho}(\halbnorm{q})|\mu|)}
       \cdot
        \seminorm[K, |\mu|, \halbnorm{q}](F),
    \end{align*}
    with the seminorm \eqref{eq:CinftyNorms}. Since the maximum over
    the first factor is finite, we see that for every compact subset
    $K$, the inclusion
    \[
    \Cinfty_K(\Rl^n, V)
    \longrightarrow
    \Symbol^{\order{m}, \type{\rho}}(\mathbb{R}^n, V)
    \]
    is continuous. By the universal property of the inductive limit
    topology of $\Cinfty_0(\mathbb{R}^n, V)$, this is equivalent to
    the continuity of the first inclusion in
    \eqref{eq:SymbolInclusions}.  For the second inclusion, we fix a
    compact subset $K \subseteq \mathbb{R}^n$ as well as $\ell \in
    \mathbb{N}_0$ and $\halbnorm{q} \in \mathcal{Q}$. Then for a
    symbol $F \in \Symbol^{\order{m}, \type{\rho}}(\mathbb{R}^n, V)$
    we have
    \begin{align*}
        \seminorm[K, \ell, \halbnorm{q}](F)
        &=
        \max_{\substack{x \in K \\ |\mu| \le \ell}}
        \halbnorm{q}\left(
            \frac{\partial^{|\mu|} F}{\partial x^\mu}(x)
        \right) \\
        &\le
          \max_{\substack{x \in K \\ |\mu| \le \ell}}
          \left(
              1 + \norm{x}^2
          \right)^{\frac{1}{2}
            (\order{m}(\halbnorm{q}) - \type{\rho}(\halbnorm{q})|\mu|)
          }
        \sup_{\substack{x \in \mathbb{R}^n \\ |\mu| \le \ell}}
        \left(
            1 + \norm{x}^2
        \right)^{-\frac{1}{2}
          (\order{m}(\halbnorm{q}) - \type{\rho}(\halbnorm{q})|\mu|)
        }
        \halbnorm{q}\left(
            \frac{\partial^{|\mu|} F}{\partial x^\mu}(x)
        \right) \\
        &=
        c \max_{|\mu| \le \ell}
        \symnorm{F}^{\order{m}, \type{\rho}}_{\halbnorm{q}, \mu},
    \end{align*}
    with a constant $c>0$, where we used the fact that the continuous
    function $x \mapsto (1 +
    \norm{x}^2)^{-\frac{1}{2}(\order{m}(\halbnorm{q}) -
      \type{\rho}(\halbnorm{q})|\mu|)}$ is nowhere vanishing and hence
    has a locally bounded inverse. This shows the continuity of the
    second inclusion in \eqref{eq:SymbolInclusions}. But then the
    second part is clear since already $\Cinfty_0(\mathbb{R}^n, V)$ is
    dense in $\Cinfty(\mathbb{R}^n, V)$. In order to show sequential
    completeness, let $(F_i)_{i \in \mathbb{N}}$ be a Cauchy sequence
    in $\Symbol^{\order{m}, \type{\rho}}(\mathbb{R}^n, V)$. Since the
    $\Cinfty$-topology is coarser than the $\Symbol^{\order{m},
      \type{\rho}}$-topology by the first part, and since
    $\Cinfty(\mathbb{R}^n, V)$ is sequentially complete, we
    get convergence $F_i \longrightarrow F$ to some smooth function $F
    \in \Cinfty(\mathbb{R}^n, V)$ in the $\Cinfty$-topology. We have
    to show that $F \in \Symbol^{\order{m}, \type{\rho}}(\mathbb{R}^n,
    V)$ and $F_i \longrightarrow F$ in the $\Symbol^{\order{m},
      \type{\rho}}$-topology. Thus let $\epsilon > 0$, $\halbnorm{q}
    \in \mathcal{Q}$, and $\mu \in \mathbb{N}_0^n$ be given. Then fix
    $N \in \mathbb{N}_0$ such that $i, j \ge N$ gives $\symnorm{F_i -
      F_j}^{\order{m}, \type{\rho}}_{\halbnorm{q}, \mu} < \epsilon$ by
    the Cauchy condition. For a point $x \in \mathbb{R}^n$ we have by
    the pointwise convergence $\frac{\partial^{|\mu|}F_j}{\partial
      x^\mu}(x) \longrightarrow \frac{\partial^{|\mu|} F}{\partial
      x^\mu}(x)$ a $N' \ge N$ depending on $x$ such that
    \[
    \left(
        1 + \norm{x}^2
    \right)^{-\frac{1}{2}
      (\order{m}(\halbnorm{q}) - \type{\rho}(\halbnorm{q})|\mu|)
    }
    \halbnorm{q}\left(
        \frac{\partial^{|\mu|} F_j}{\partial x^\mu}(x)
        -
        \frac{\partial^{|\mu|} F}{\partial x^\mu}(x)
    \right)
    < \epsilon
    \tag{$*$}
    \]
    for all $j \ge N'$. Thus for $i \ge N$ we get
    \begin{align*}
    &\left(
        1 + \norm{x}^2
    \right)^{-\frac{1}{2}
      (\order{m}(\halbnorm{q}) - \type{\rho}(\halbnorm{q})|\mu|)
    }
    \halbnorm{q}\left(
        \frac{\partial^{|\mu|} (F-F_i)}{\partial x^\mu}(x)
    \right) \\
    &\qquad\le
    \left(
        1 + \norm{x}^2
    \right)^{-\frac{1}{2}
      (\order{m}(\halbnorm{q}) - \type{\rho}(\halbnorm{q})|\mu|)
    }
    \left(
        \halbnorm{q}\left(
            \frac{\partial^{|\mu|} (F-F_j)}{\partial x^\mu}(x)
        \right)
        +
        \halbnorm{q}\left(
            \frac{\partial^{|\mu|} (F_j-F_i)}{\partial x^\mu}(x)
        \right)
    \right) \\
    &\qquad\stackrel{(*)}{\le}
    \epsilon
    +
    \symnorm{F_i-F_j}^{\order{m}, \type{\rho}}_{\halbnorm{q}, \mu}
    \\
    &\qquad\le 2\epsilon.
    \end{align*}
    Since this estimate can be done for all $x \in \mathbb{R}^n$, we
    can take the supremum over all $x \in \mathbb{R}^n$ and get
    $\symnorm{F-F_i}^{\order{m}, \type{\rho}}_{\halbnorm{q}, \mu} \le
    2\epsilon$. Hence $F -F_i \in \Symbol^{\order{m},
      \type{\rho}}(\mathbb{R}^n, V)$ for $i \ge N$ and thus also $F
    \in \Symbol^{\order{m}, \type{\rho}}(\mathbb{R}^n, V)$. Moreover,
    we conclude that $F_i \longrightarrow F$ in the
    $\Symbol^{\order{m}, \type{\rho}}$-topology. Clearly, if $V$ is
    even complete we can repeat the argument with nets instead of
    sequences. For the last part, it is sufficient to show the
    estimate \eqref{eq:SymbolNormViaSymbolNorm}. Since for
    $\order{m}(\halbnorm{q}) \le \order{m}'(\halbnorm{q})$ and
    $\type{\rho}(\halbnorm{q}) \ge \type{\rho}'(\halbnorm{q})$ we have
    \[
    \left(
        1 + \norm{x}^2
    \right)^{-\frac{1}{2}
      (\order{m}(\halbnorm{q}) - \type{\rho}(\halbnorm{q})|\mu|)
    }
    \ge
    \left(
        1 + \norm{x}^2
    \right)^{-\frac{1}{2}
      (\order{m}'(\halbnorm{q}) - \type{\rho}'(\halbnorm{q})|\mu|)
    }
    \]
    for every point $x \in \mathbb{R}^n$ and every $\mu \in
    \mathbb{N}_0^n$, the estimate \eqref{eq:SymbolNormViaSymbolNorm}
    follows.
\end{proof}

In case $V$ is a Banach space, we choose just its norm $\|\cdot\|$ in
order to define the space of symbols. In this case, the order
$m:=\order{m}(\|\cdot\|)$ and the type $\rho:=\type{\rho}(\|\cdot\|)$
are just single numbers, and we write $\|\cdot\|^{m,\rho}_\mu$ instead
of $\|\cdot\|^{m,\rho}_{\|\cdot\|,\mu}$. However, $\Symbol^{m,
  \rho}(\mathbb{R}^n, V)$ is no longer a Banach space but a Fréchet
space since we have to take care of countably many
differentiations. For a Fréchet space $V$, we take a countable
defining system of seminorms and hence an order is determined by
fixing countably many numbers $\order{m}(\halbnorm{q}_n)$. Thus, in
this situation the symbols are again a Fréchet space.

Note that the inclusion
$\Cinfty_0(\Rl^n,V)\subset\Symbol^{\bm,\brho}(\Rl^n,V)$ is in general
not (sequentially) dense in the
$\Symbol^{\bm,\brho}$-topology. However, we will show later
(Proposition~\ref{proposition:ApproximateSymbols},
\refitem{item:CinftyDenseInSymbols}), that it is dense in a weaker
topology. As a preparation for this, we need to study the
multiplication properties of symbols.

\subsection{Operations with symbols}

In this subsection we discuss several operations one can perform with
symbols, like differentiation, multiplication, composition with linear
maps, and restriction. We begin by showing that the topologies of the
symbol spaces are compatible with differentiation.
\begin{proposition}
    \label{proposition:PartialDerivativesOnSymbols}%
    Let $V$ be a sequentially complete locally convex space with a
    defining system of seminorms $\mathcal{Q}$, and $\bm,\brho$ an
    order and a type for $\mathcal Q$. Then the partial derivatives
    are continuous linear maps
    \begin{equation}
        \label{eq:PartialDerivativesContinuousSymbol}
        \frac{\partial^{|\nu|}}{\partial x^\nu}\colon
        \Symbol^{\order{m}, \type{\rho}}(\mathbb{R}^n, V)
        \longrightarrow
        \Symbol^{\order{m} - \type{\rho}|\nu|, \type{\rho}}
        (\mathbb{R}^n, V).
    \end{equation}
    More precisely, we have for all $\mu \in \mathbb{N}_0^n$ and $F
    \in \Symbol^{\order{m}, \type{\rho}}(\mathbb{R}^n, V)$
    \begin{equation}
        \label{eq:SymnormOfDerivative}
        \symnorm{
          \frac{\partial^{|\nu|} F}{\partial x^\nu}
        }^{
          \order{m} - \type{\rho}|\nu|, \type{\rho}
        }_{\halbnorm{q}, \mu}
        =
        \symnorm{F}^{\order{m}, \type{\rho}}_{\halbnorm{q}, \mu + \nu}.
    \end{equation}
\end{proposition}
\begin{proof}
    We only have to show the second statement which is clear from the
    definition.
\end{proof}

For a general discussion of multiplication of symbols, we now consider
three sequentially complete locally convex spaces $V$, $W$, and $U$
together with a bilinear map
\begin{equation}
    \label{eq:mVWU}
    \mu\colon V \times W \longrightarrow U.
\end{equation}
For simplicity, we require that $\mu$ is continuous and not just
separately continuous or sequentially continuous. In many
applications, this will be the case. Now we fix a defining system
$\mathcal{R}$ of seminorms on $U$ and filtrating defining systems of
seminorms $\mathcal{Q}$ and $\mathcal{Q}'$ on $V$ and $W$,
respectively. Then by continuity of $\mu$ we get for every
$\halbnorm{r} \in \mathcal{R}$ seminorms $\halbnorm{q} \in
\mathcal{Q}$ and $\halbnorm{q}' \in \mathcal{Q}'$ and a constant $c$
such that
\begin{equation}
    \label{eq:ContinuityOfProductm}
    \halbnorm{r}(\mu(v, w))
    \le
    c \halbnorm{q}(v) \halbnorm{q}'(w)
      \,,\qquad v\in V\,,w\in W\,.
\end{equation}
For two orders $\order{m}$ and $\order{m}'$ on $V$ and $W$ we consider
an order $\order{m}''$ on $U$ such that for all $\halbnorm{r} \in
\mathcal{R}$ we have $\halbnorm{q} \in \mathcal{Q}$ and $\halbnorm{q}'
\in \mathcal{Q}'$ such that \eqref{eq:ContinuityOfProductm} holds
\emph{and}
\begin{equation}
    \label{eq:mrGroessermqmq}
    \order{m}''(\halbnorm{r})
    \ge
    \order{m}(\halbnorm{q})
    +
    \order{m}'(\halbnorm{q}').
\end{equation}
In this case, we symbolically write $\order{m}'' \ge \order{m} +
\order{m}'$ by some slight abuse of notation. Note that we relate here
orders on \emph{different} sets of seminorms (even on different
spaces). Clearly, for given orders $\order{m}$ and $\order{m}'$ we can
construct an order $\order{m}''$ with this property by fixing a choice
of seminorms $\halbnorm{q}(\halbnorm{r})$ and
$\halbnorm{q}'(\halbnorm{r})$ satisfying
\eqref{eq:ContinuityOfProductm} and setting
\begin{equation}
    \label{eq:ordermppDefinition}
    \order{m}''(\halbnorm{r})
    =
    \order{m}(\halbnorm{q}(\halbnorm{r}))
    +
    \order{m}''(\halbnorm{q}'(\halbnorm{r}))
\end{equation}
for all $\halbnorm{r} \in \mathcal{R}$. In the same spirit we write
$\type{\rho}'' \le \min(\type{\rho}, \type{\rho}')$ for types
$\type{\rho}$ on $V$, $\type{\rho}'$ on $W$, and $\type{\rho}''$ on
$U$, again with respect to the continuous bilinear map $\mu$.
With these conventions in mind, we can prove the following
statement.
\begin{proposition}
    \label{proposition:SymbolProducts}%
    Let $V$, $W$, and $U$ be sequentially complete locally convex
    spaces, $\mathcal{R}$ a defining system of seminorms on $U$, and
    $\mathcal{Q}$, $\mathcal{Q}'$ filtrating defining systems of
    seminorms on $V$, $W$ respectively. Furthermore, let $\mu\colon V
    \times W \longrightarrow U$ be a continuous bilinear map.
    \begin{propositionlist}
    \item \label{item:ProductInOneVariable} Pointwise evaluation of
        $\mu$ gives a continuous bilinear map
        \begin{align}
            \label{eq:ProductOfSymbols}
            \mu\colon
            \Symbol^{\order{m}, \type{\rho}}(\mathbb{R}^n, V)
            \times
            \Symbol^{\order{m}', \type{\rho}'}(\mathbb{R}^n, W)
            &
            \longrightarrow
            \Symbol^{\order{m}'', \type{\rho}''}(\mathbb{R}^n, U)
            \,,
            \\
            \mu(F,G)(x) &:= \mu(F(x),G(x))\,,
        \end{align}
        whenever $\order{m}'' \ge \order{m} + \order{m}'$ and
        $\type{\rho}'' \le \min(\type{\rho}, \type{\rho}')$ with
        respect to $\mu$. More precisely, for $F \in
        \Symbol^{\order{m}, \type{\rho}}(\mathbb{R}^n, V)$ and $G \in
        \Symbol^{\order{m}', \type{\rho}'}(\mathbb{R}^n, W)$ we get
        \begin{equation}
            \label{eq:EstimateProductSymbol}
            \symnorm{\mu(F,G)}^{
              \order{m}'', \type{\rho}''
            }_{\halbnorm{r}, \kappa}
            \le
            2^{|\mu|} c
            \,
            \max_{\nu \le \kappa}
            \symnorm{F}^{\order{m}, \type{\rho}}_{\halbnorm{q}, \nu}
            \,
            \max_{\nu' \le \kappa}
            \symnorm{G}^{\order{m}', \type{\rho}'}_{\halbnorm{q}', \nu'}
        \end{equation}
        whenever $\halbnorm{r}$, $\halbnorm{q}$, and $\halbnorm{q}'$
        satisfy the continuity property
        \eqref{eq:ContinuityOfProductm} with respect to $\mu$.
    \item \label{item:ProductInTwoVariables} For $F\in
        \Symbol^{\bm,\brho}(\Rl^n,V)$, $G\in
        \Symbol^{\bm',\brho'}(\Rl^{n'},W)$, let
        \begin{align}
            \underline{\mu}(F,G)\colon
            \Rl^n\oplus\Rl^{n'}\to U,
            \qquad
            \underline{\mu}(F,G)(x,y):=\mu(F(x),G(y))\,.
        \end{align}
        Then we have a continuous bilinear map
        \begin{align}
            \underline{\mu} &\colon
            \Symbol^{\bm,\brho}(\Rl^n,V)\times
            \Symbol^{\bm',\brho'}(\Rl^{n'},W)
            \longrightarrow
            \Symbol^{\bm'',\brho''}(\Rl^{n}\oplus\Rl^{n'},U)
        \end{align}
        whenever $\bm''\geq\max\{\bm,\bm',\bm+\bm'\}$ and
        $\brho''\leq\min\{0,\brho,\brho'\}$ with respect to
        $\mu$. Explicitly, for $F\in \Symbol^{\bm,\brho}(\Rl^n,V)$,
        $G\in \Symbol^{\bm',\brho'}(\Rl^{n'},W)$, we get,
        $\nu\in\Rl^n$, $\nu'\in\Rl^{n'}$,
        \begin{align}\label{mu-otimes-bound}
            \|\underline{\mu}(F,G)\|^{\order{m}'',\type{\rho}''}_{\rn,\nu\oplus\nu'}
            \leq
            c\,\|F\|^{\order{m},\type{\rho}}_{\qn,\nu}\,\|G\|^{\order{m}',\type{\rho}'}_{\qn',\nu'}
        \end{align}
        whenever $\rn$, $\qn$, $\qn'$ satisfy the continuity property
        \eqref{eq:ContinuityOfProductm} with respect to $\mu$.
    \end{propositionlist}
\end{proposition}
\begin{proof}
    For \refitem{item:ProductInOneVariable}, even though the
    formulation looks rather technical, this is essentially just the
    Leibniz rule. Let $\halbnorm{r}\in\mathcal{R}$, choose
    corresponding seminorms $\halbnorm{q}\in\mathcal{Q}$ and
    $\halbnorm{q}'\in\mathcal{Q}'$ satisfying
    \eqref{eq:ContinuityOfProductm}, and $\kappa\in\Nl_0^n$. Then we
    have
    \begin{align*}
        &\symnorm{\mu(F, G)}^{
          \order{m}'', \type{\rho}''}_{\halbnorm{r}, \kappa}\\
        &\quad=
        \sup_{x \in \mathbb{R}^n}
        \left(
            1 + \norm{x}^2
        \right)^{
          -\frac{1}{2}
          \left(
              \order{m}''(\halbnorm{r})
              - \type{\rho}''(\halbnorm{r})|\kappa|
          \right)
        }
        \halbnorm{r}\left(
            \frac{\partial^{|\kappa|}\mu(F, G)}{\partial x^\kappa}(x)
        \right) \\
        &\quad=
        \sup_{x \in \mathbb{R}^n}
        \left(
            1 + \norm{x}^2
        \right)^{
          -\frac{1}{2}
          \left(
              \order{m}''(\halbnorm{r})
              - \type{\rho}''(\halbnorm{r})|\kappa|
          \right)
        }
        \halbnorm{r}\left(
            \sum_{\nu + \nu' = \kappa} \binom{\kappa}{\nu}
            \mu
            \left(
                \frac{\partial^{|\nu|}F}{\partial x^\nu}(x),
                \frac{\partial^{|\nu'|}G}{\partial x^{\nu'}}(x)
            \right)
        \right) \\
        &\quad\le
        \sum_{\nu + \nu' = \kappa} \binom{\kappa}{\nu}
        c
        \,
        \sup_{x \in \mathbb{R}^n}
        \left(
            1 + \norm{x}^2
        \right)^{
          -\frac{1}{2}
          \left(
              \order{m}(\halbnorm{q})
              - \type{\rho}(\halbnorm{q})|\nu|
          \right)
        }
        \halbnorm{q}\left(
            \frac{\partial^{|\nu|}F}{\partial x^\nu}(x)
        \right) \\
        &\quad\qquad
        \sup_{x \in \mathbb{R}^n}
        \left(
            1 + \norm{x}^2
        \right)^{
          -\frac{1}{2}
          \left(
              \order{m}'(\halbnorm{q}')
              - \type{\rho'}(\halbnorm{q}')|\nu'|
          \right)
        }
        \halbnorm{q}'\left(
            \frac{\partial^{|\nu'|}G}{\partial x^{\nu'}}(x)
        \right) \\
        &\quad\le
        2^{|\kappa|} c
        \,
        \max_{\nu \le \kappa}
        \symnorm{F}^{\order{m}, \type{\rho}}_{\halbnorm{q}, \nu}
        \,
        \max_{\nu' \le \kappa}
        \symnorm{G}^{\order{m}', \type{\rho}'}_{\halbnorm{q}', \nu'},
    \end{align*}
    since $|\kappa| = |\nu| + |\nu'|$. This shows
    \eqref{eq:EstimateProductSymbol}, which implies the continuity of
    the map \eqref{eq:ProductOfSymbols}.  For
    \refitem{item:ProductInTwoVariables}, it is clear that for any
    $F\in \Symbol^{\order{m},\type{\rho}}(\Rl^n,V)$, $G\in
    \Symbol^{\order{m}',\type{\rho}'}(\Rl^{n'},W)$, we have
    $\underline{\mu}(F,G) \in \Cinfty(\Rl^n\oplus\Rl^{n'},U)$, and
    that $\underline{\mu}$ is bilinear. To prove the continuity of
    this map, we have to verify the bound \eqref{mu-otimes-bound}. The
    estimate necessary for this is based on the fact that given
    $k,k'\in\Rl$, there holds for all $a,b\in\Rl$
    \begin{align}\label{eq:kkkInequality}
        (1+a^2+b^2)^{-K} \leq
        (1+a^2)^{-k}(1+b^2)^{-k'}\,\;\;\text{if}\;\;\,K\geq\max\{k,k',k+k'\}\,.
    \end{align}
    In the situation at hand, we pick seminorms $\rn$, $\qn$, $\qn'$
    satisfying \eqref{eq:ContinuityOfProductm}, multiindices
    $\nu\in\Rl^n$, $\nu'\in\Rl^{n'}$, and set
    $K:=\frac{1}{2}(\bm''(\rn)-\brho''(\rn)|\nu\oplus\nu'|)$,
    $k:=\frac{1}{2}(\bm(\qn)-\brho(\qn)|\nu|)$,
    $k':=\frac{1}{2}(\bm'(\qn')-\brho'(\qn')|\nu'|)$. The assumptions
    on $\bm''$, $\brho''$ then guarantee that $K\geq\max\{k,k',k+k'\}$
    holds, and we can use \eqref{eq:kkkInequality} to estimate
    \begin{align*}
        \|\underline{\mu}(F,G)\|^{\order{m}'',\type{\rho}''}_{\rn,\nu\oplus\nu'}
        &=
        \sup_{x\in\Rl^n\atop{y\in\Rl^{n'}}} \frac{\rn(\partial_x^\nu\partial_y^{\nu'}\mu(F(x),G(y)))}{(1+\|x\|^2+\|y\|^2)^{\frac{1}{2}(\order{m}(\rn)''-\type{\rho}''(\rn)|\nu\oplus\nu'|)}}
        \\
        &\leq
        \sup_{x\in\Rl^n\atop{y\in\Rl^{n'}}} \frac{c\,\qn(\partial_x^\nu
          F(x))\,\qn'(\partial_y^{\nu'}
          G(y))}{(1+\|x\|^2)^{\frac{1}{2}(\bm(\qn)-\brho(\qn)|\nu|)}(1+\|y\|^2)^{\frac{1}{
              2}(\bm'(\qn')-\brho'(\qn')|\nu'|)}}
        \\
        &\leq
        c\,\|F\|^{\bm,\rho}_{\qn,\nu}\,\|G\|^{\bm',\brho'}_{\qn',\nu'}\,.
    \end{align*}
\end{proof}

For continuous {\em linear} maps, a similar result holds.
\begin{proposition}
    \label{proposition:LinearMapsOnSymbols}%
    Let $A\colon V \longrightarrow W$ be a continuous linear map
    between sequentially complete locally convex spaces $V, W$ with
    defining systems of seminorms $\mathcal{Q}$, $\mathcal{Q}'$, and
    let $\mathcal{Q}$ be filtrating. Furthermore, let orders
    $\order{m}$ and $\order{m}'$ and types $\type{\rho}$ and
    $\type{\rho}'$ for $\mathcal{Q}$ and $\mathcal{Q}'$ be
    given. Suppose for every seminorm $\halbnorm{q}' \in \mathcal{Q}'$
    we find a seminorm $\halbnorm{q} \in \mathcal{Q}$ such that
    \begin{equation}
        \label{eq:qprimeAvqvSymbolStuff}
        \halbnorm{q}'(Av) \le c \halbnorm{q}(v),
        \quad
        \order{m}(\halbnorm{q})
        \le
        \order{m}'(\halbnorm{q}'),
        \quad
        \textrm{and}
        \quad
        \type{\rho}(\halbnorm{q}) \ge \type{\rho}'(\halbnorm{q}')
    \end{equation}
    for some $c > 0$ and all $v \in V$. Then pointwise evaluation of
    $A$ gives a continuous linear map
    \begin{equation}
        \label{eq:ALinearMapOnSymbols}
        A\colon
        \Symbol^{\order{m}, \type{\rho}}(\mathbb{R}^n, V)
        \longrightarrow
        \Symbol^{\order{m}', \type{\rho}'}(\mathbb{R}^n, W).
    \end{equation}
    More precisely, for every $F \in \Symbol^{\order{m},
      \type{\rho}}(\mathbb{R}^n, V)$ and every $\mu \in
    \mathbb{N}_0^n$, we have
    \begin{equation}
        \label{eq:SymbolNormLinearMap}
        \symnorm{A F}^{\order{m}', \type{\rho}'}_{\halbnorm{q}', \mu}
        \le
        c
        \symnorm{F}^{\order{m}, \type{\rho}}_{\halbnorm{q}, \mu}.
    \end{equation}
\end{proposition}
\begin{proof}
    Note that the first condition $\halbnorm{q}'(Av) \le
    c\halbnorm{q}(v)$ can always be satisfied since $\mathcal{Q}$ was
    assumed to be filtrating and $A$ is continuous. Thus assume that
    the other two requirements in \eqref{eq:qprimeAvqvSymbolStuff} are
    fulfilled as well. Then we have for $F \in \Symbol^{\order{m},
      \type{\rho}}(\mathbb{R}^n, V)$
    \begin{align*}
        \symnorm{AF}^{\order{m}', \type{\rho}'}_{\halbnorm{q}', \mu}
        &=
        \sup_{x \in \mathbb{R}^n}
        \left(
            1 + \norm{x}^2
        \right)^{
          -\frac{1}{2}
          (\order{m}'(\halbnorm{q}') -
          \type{\rho}'(\halbnorm{q}')|\mu|)
        }
        \halbnorm{q}'\left(
            \frac{\partial^{|\mu|} AF}{\partial x^\mu}(x)
        \right) \\
        &\leq
        \sup_{x \in \mathbb{R}^n}
        \left(
            1 + \norm{x}^2
        \right)^{
          -\frac{1}{2}
          (\order{m}(\halbnorm{q}) -
          \type{\rho}(\halbnorm{q})|\mu|)
        }
        \halbnorm{q}'\left(
            A \frac{\partial^{|\mu|} F}{\partial x^\mu}(x)
        \right) \\
        &\le
        c
        \symnorm{F}^{\order{m}, \type{\rho}}_{\halbnorm{q}, \mu}.
    \end{align*}
    Since the seminorms $\symnorm{\argument}^{\order{m}',
      \type{\rho}'}_{\halbnorm{q}', \mu}$ determine the
    $\Symbol^{\order{m}', \type{\rho}'}$-topology, this yields the
    continuity of \eqref{eq:ALinearMapOnSymbols}.
\end{proof}

The main application of Proposition~\ref{proposition:SymbolProducts},
\refitem{item:ProductInOneVariable}, is to multiply vector-valued
symbols with scalar symbols: Choosing one target space just to be
$\mathbb{C}$, with seminorms just the absolute value, we get for every
order $m \in \mathbb{R}$ and every type $\rho \in \mathbb{R}$ the
space of scalar symbols
\begin{equation}
    \label{eq:ScalarSymbols}
    \Symbol^{m, \rho}(\mathbb{R}^n)
    =
    \Symbol^{m, \rho}(\mathbb{R}^n, \mathbb{C}).
\end{equation}
Note that here the order and the type are indeed just single
numbers. We now formulate two corollaries about multiplications of
symbols. Their proofs follow immediately from
Proposition~\ref{proposition:SymbolProducts} and are therefore
omitted.
\begin{corollary}
    \label{corollary:SymbolModuleStuff}%
    Let $V$ be a sequentially complete locally convex space with a
    defining system of seminorms $\mathcal{Q}$.
    \begin{corollarylist}
    \item \label{item:SymbolModule} For all orders $\order{m}$ and
        types $\type{\rho}$ for $\mathcal{Q}$, and all $m, \rho \in
        \mathbb{R}$, the pointwise multiplication gives a continuous
        bilinear map
        \begin{equation}
            \label{eq:SymbolModule}
            \Symbol^{m, \rho}(\mathbb{R}^n, \mathbb{C})
            \times
            \Symbol^{\order{m}, \order{\rho}}(\mathbb{R}^n, V)
            \longrightarrow
            \Symbol^{\order{m}+m, \min(\order{\rho},\rho)}
            (\mathbb{R}^n, V).
        \end{equation}
    \item \label{item:BoundedTypeModule} In particular, if the type
        $\type{\rho}$ is bounded by $\rho \in \mathbb{R}$, then
        \begin{equation}
            \label{eq:BoundedTypeModule}
            \Symbol^{m, \rho}(\mathbb{R}^n, \mathbb{C})
            \times
            \Symbol^{\order{m}, \order{\rho}}(\mathbb{R}^n, V)
            \longrightarrow
            \Symbol^{\order{m}+m, \order{\rho}}
            (\mathbb{R}^n, V)
        \end{equation}
	is a continuous bilinear map.
    \item \label{item:SymbolsFrechetAlgebra} If $m \le 0$ then
        $\Symbol^{m, \type{\rho}}(\mathbb{R}^n, \mathbb{C})$ is a
        Fréchet algebra and $\Symbol^{\order{m},
          \type{\rho}}(\mathbb{R}^n, V)$ is a continuous module over
        it for all bounded $\type{\rho} \le \rho$.
    \end{corollarylist}
\end{corollary}
\begin{corollary}
    \label{corollary:SymbolAlgebra}%
    Let $\mathcal{A}$ be a sequentially complete locally convex
    algebra with a defining system of seminorms $\mathcal{Q}$. Then
    the multiplication induces a continuous product
    \begin{equation}
        \label{eq:ProductForAlgebraSymbols}
        \Symbol^{\order{m}, \order{\rho}}
        (\mathbb{R}^n, \mathcal{A})
        \times
        \Symbol^{\order{m}', \order{\rho}}
        (\mathbb{R}^n, \mathcal{A})
        \longrightarrow
        \Symbol^{\order{m} + \order{m}', \order{\rho}}
        (\mathbb{R}^n, \mathcal{A}).
    \end{equation}
    In particular, for $\order{m} \le 0$ the symbols
    $\Symbol^{\order{m}, \order{\rho}}(\mathbb{R}^n, \mathcal{A})$
    form a sequentially complete locally convex algebra themselves and
    any $\Symbol^{\order{m}', \order{\rho}}(\mathbb{R}^n,
    \mathcal{A})$ is a sequentially complete locally convex continuous
    module over them.
\end{corollary}

For later applications, we also note the following simple lemma about
powers of scalar symbols.
\begin{lemma}
    \label{lemma:ScalarSymbolPower}%
    Let $f \in \Symbol^{m, \rho}(\mathbb{R}^n, \mathbb{C})$ be a
    scalar symbol of order $m$ and type $\rho$ with $f(x) \in
    \mathbb{C} \setminus [0, - \infty)$ for all $x \in
    \mathbb{R}^n$. Then for all $\alpha \in \mathbb{C}$ with
    $\RE(\alpha)\geq0$ we have $f^\alpha \in \Symbol^{\RE(\alpha)m,
      \rho}(\mathbb{R}^n, \mathbb{C})$.
\end{lemma}
\begin{proof}
    Since $f$ takes values only in the complex plane without the
    closed negative real half axis, we can use the (smooth) principal
    branch of the logarithm to define the powers $f^\alpha =
    \exp(\alpha \log(f)) \in \Cinfty(\mathbb{R}^n, \mathbb{C})$ for
    all $\alpha \in \mathbb{C}$. Note that we do not have to take care
    of the values $\alpha \in \mathbb{N}_0$ as these particular cases
    are already settled by Corollary~\ref{corollary:SymbolAlgebra} by
    induction. By the chain rule we get
    \[
    \frac{\partial^{|\mu|} f^\alpha}{\partial x^\mu}
    =
    \sum_{
      \substack{1 \le k \le |\mu| \\ \nu_1,\ldots, \nu_k\in\mathbb{N}_0^n\\ |\nu_1|+\cdots+|\nu_k|=|\mu|}
    }
    C^{k,\alpha}_{\nu_1,\ldots,\nu_k}
    f^{\alpha - k}
    \,
    \frac{\partial^{|\nu_1|} f}{\partial x^{\nu_1}}
    \cdots
    \frac{\partial^{|\nu_k|} f}{\partial x^{\nu_k}}
    \]
    with some coefficients $C^{k,\alpha}_{\nu}\in\Cl$. We first note
    that for $f^\alpha$ without derivatives we get the estimate
    \begin{align*}
        \symnorm{f^\alpha}^{\RE(\alpha)m, \rho}_0
        &= \sup_{x \in \mathbb{R}^n}
        \left(
            1 + \norm{x}^2
        \right)^{-\frac{1}{2} \RE(\alpha)m}
        |f^\alpha(x)| \\
        &\le
        e^{\pi|\IM(\alpha)|}
	 \sup_{x\in\Rl^n}
        \left(
            \left(
                1 + \norm{x}^2
            \right)^{-\frac{m}{2}}
            |f(x)|
        \right)^{\RE(\alpha)} \\
        &=
        e^{\pi|\IM(\alpha)|}
        \left(\symnorm{f}^{m, \rho}_0\right)^{\RE(\alpha)}
    \end{align*}
    since for any complex number $z \in \mathbb{C} \setminus [0,
    -\infty)$ we have $|z^\alpha| \le |z|^{\RE(\alpha)}
    e^{\pi|\IM(\alpha)|}$, and since $\RE(\alpha)\geq 0$. Thus we need
    the prefactor $(1 + \norm{x}^2)^{-\frac{1}{2}\RE(\alpha)m}$ to
    compensate the growth of $f^\alpha$. To estimate the derivatives,
    we get
    \begin{align*}
        &\symnorm{f^\alpha}^{\RE(\alpha)m, \rho}_\mu \\
        &=
        \sup_{x \in \mathbb{R}^n}
        \left(
            1 + \norm{x}^2
        \right)^{-\frac{1}{2} (\RE(\alpha)m - \rho|\mu|)}
        \left|
            \frac{\partial^{|\mu|} f^\alpha}{\partial x^\mu}(x)
        \right| \\
        &\le
        \sup_{x \in \mathbb{R}^n}
        \left(
            1 + \norm{x}^2
        \right)^{-\frac{1}{2} (\RE(\alpha)m - \rho|\mu|)}
        \sum_{
	\substack{1 \le k \le |\mu| \\ \nu_1,\ldots,\nu_k\in\mathbb{N}_0^n\\ |\nu_1|+\cdots+|\nu_k|=|\mu|}
	}
        |C^{k,\alpha}_{\nu_1,\ldots,\nu_k}|
        \left|f^{\alpha - k}(x)\right|
        \left|\frac{\partial^{|\nu_1|} f}{\partial x^{\nu_1}}\right|
        \cdots
        \left|\frac{\partial^{|\nu_k|} f}{\partial x^{\nu_k}}\right|
        \\
        &\le
	\sum_{
	\substack{1 \le k \le |\mu| \\ \nu_1,\ldots, \nu_k\in\mathbb{N}_0^n\\ |\nu_1|+\cdots+|\nu_k|=|\mu|}
	}
        |C^{k,\alpha}_{\nu_1,\ldots,\nu_k}|
        \sup_{x \in \mathbb{R}^n}
        \left(
            1 + \norm{x}^2
        \right)^{-\frac{1}{2}(\RE(\alpha)m - km)}
        \left|f^{\alpha - k}(x)\right| \\
        &\qquad
        \sup_{x \in \mathbb{R}^n}
        \left(
            1 + \norm{x}^2
        \right)^{-\frac{1}{2} (m - \rho|\nu_1|)}
        \left|\frac{\partial^{|\nu_1|} f}{\partial x^{\nu_1}}\right|
        \cdots
        \sup_{x \in \mathbb{R}^n}
        \left(
            1 + \norm{x}^2
        \right)^{-\frac{1}{2} (m - \rho|\nu_k|)}
        \left|\frac{\partial^{|\nu_k|} f}{\partial x^{\nu_k}}\right|
        \\
        &\le
	\sum_{
	\substack{1 \le k \le |\mu| \\ \nu_1,\ldots,\nu_k\in\mathbb{N}_0^n\\ |\nu_1|+\cdots+|\nu_k|=|\mu|}
	}
        |C^{k,\alpha}_{\nu_1,\ldots,\nu_k}|
	e^{\pi|\IM(\alpha)|}
        \left(\symnorm{f}^{m, \rho}_0\right)^{\RE(\alpha) - k}
        \symnorm{f}^{m, \rho}_{\nu_1}
        \cdots
        \symnorm{f}^{m, \rho}_{\nu_k},
    \end{align*}
    which proves that $f^\alpha \in \Symbol^{\RE(\alpha)m,
      \rho}(\mathbb{R}^n, \mathbb{C})$ as claimed.
\end{proof}

We come now to the approximation of symbols by compactly supported
functions, which will be important for our subsequent construction of
an oscillatory integral. As usual, balls in $\Rl^n$ will be denoted
$\Ball_r(0):=\{x\in\Rl^n\,:\,\|x\|\leq r\}$.
\begin{proposition}
    \label{proposition:ApproximateSymbols}%
    Let $\chi \in \Cinfty_0(\mathbb{R}^n)$ be a compactly supported
    smooth function with
    \begin{equation}
        \label{eq:NiceCutOffFunction}
        \chi_{\Ball_r(0)} = 1
        \quad
        \textrm{and}
        \quad
        \supp \chi \subseteq  \Ball_R(0),
    \end{equation}
    where $0 < r < R$ and let $\chi_\epsilon \in
    \Cinfty_0(\mathbb{R}^n)$ be defined by $\chi_\epsilon(x) =
    \chi(\epsilon x)$ for $\epsilon > 0$.
    \begin{propositionlist}
    \item \label{item:ChiEpsilonSymbol} One has $\chi_\epsilon - 1 \in
        \Symbol^{0, \rho}(\mathbb{R}^n, \mathbb{C})$ for all $\rho \in
        \mathbb{R}$.
    \item \label{item:ChiEpsilonToEins} One has
        \begin{equation}
            \label{eq:LimChiEpsilonEins}
            \lim_{\epsilon \longrightarrow 0} \chi_\epsilon = 1
        \end{equation}
        in the $\Symbol^{m, \rho}$-topology for all $m > 0$ and $\rho
        \le 1$.
    \item \label{item:LimChiEpsilonVectorValued} For all $F \in
        \Symbol^{\order{m}, \type{\rho}}(\mathbb{R}^n, V)$ we have
        \begin{equation}
            \label{eq:LimChiEpsilonVectorValued}
            \lim_{\epsilon \longrightarrow 0} \chi_\epsilon F = F
        \end{equation}
        in the $\Symbol^{\order{m}', \type{\rho}'}$-topology for all
        $\order{m}' > \order{m}$ and $\type{\rho}' \le \min(1,
        \type{\rho})$.
    \item \label{item:CinftyDenseInSymbols} For all $\order{m}' >
        \order{m}$ and $\type{\rho}' \le \min(1, \type{\rho})$, the
        compactly supported smooth functions $\Cinfty_0(\mathbb{R}^n,
        V)$ are sequentially dense in $\Symbol^{\order{m},
          \type{\rho}}(\mathbb{R}^n, V)$ with respect to the
        $\Symbol^{\order{m}', \type{\rho}'}$-topology.
    \end{propositionlist}
\end{proposition}
\begin{proof}
    Clearly, $1 \in \Symbol^{0, \rho}(\mathbb{R}^n, \mathbb{C})$ for
    all $\rho \in \mathbb{R}$ and $\chi_\epsilon \in
    \Cinfty_0(\mathbb{R}^n) \subseteq \Symbol^{m, \rho}(\mathbb{R}^n,
    \mathbb{C})$ for all $m, \rho \in \mathbb{R}$ by
    Proposition~\ref{proposition:SymbolFirstProperties},
    \refitem{item:SymbolInclusion}. For the second part, we clearly
    have pointwise convergence and even convergence in the
    $\Cinfty$-topology. For $\mu = 0$ we have
    \begin{align*}
        \symnorm{\chi_\epsilon - 1}_0^{m, \rho}
        &=
        \sup_{x \in \mathbb{R}^n}
        \left(1 + \norm{x}^2\right)^{-\frac{1}{2}m}
        |\chi_\epsilon (x) - 1| \\
        &=
        \sup_{\norm{x} \ge \frac{r}{\epsilon}}
        \left(1 + \norm{x}^2\right)^{-\frac{1}{2}m}
        |\chi_\epsilon (x) - 1| \\
        &\le
        c \sup_{\norm{x} \ge \frac{r}{\epsilon}}
        \left(1 + \norm{x}^2\right)^{-\frac{1}{2}m} \\
        &=
        c \left(
            \frac{r^2 + \epsilon^2}{\epsilon^2}
        \right)^{-\frac{1}{2} m},
    \end{align*}
    where $c = \supnorm{\chi_\epsilon - 1} < \infty$ by the compact
    support of $\chi_\epsilon$. This converges to zero since
    $m > 0$. For $\mu \ne 0$ we have
    \[
    \frac{\partial^{|\mu|}\chi_\epsilon}{\partial x^\mu}(x)
    =
    \epsilon^{|\mu|}
    \frac{\partial^{|\mu|}\chi}{\partial x^\mu}(\epsilon x)
    \]
    and hence
    \begin{align*}
        \symnorm{\chi_\epsilon - 1}_\mu^{m, \rho}
        &=
        \sup_{x \in \mathbb{R}^n}
        \left(1 + \norm{x}^2\right)^{-\frac{1}{2}(m - \rho|\mu|)}
        \left|
            \frac{\partial^{|\mu|}\chi_\epsilon}{\partial x^\mu}(x)
        \right| \\
        &\le
        \sup_{\frac{r}{\epsilon} \le \norm{x} \le \frac{R}{\epsilon}}
        \left(1 + \norm{x}^2\right)^{-\frac{1}{2}(m - \rho|\mu|)}
        \epsilon^{|\mu|} c_\mu,
    \end{align*}
    where $c_\mu =
    \supnorm{\partial_x^\mu \chi_\epsilon(x)}
    < \infty$, again thanks to the compact support. Now either $m -
    \rho|\mu| \ge 0$, then the supremum is taken at the smallest
    possible $\norm{x} = \frac{r}{\epsilon}$, or $m - \rho|\mu| < 0$,
    then the supremum is taken at the largest possible $\norm{x} =
    \frac{R}{\epsilon}$. Thus we get in the first case for $\epsilon
    \le 1$
    \[
    \symnorm{\chi_\epsilon - 1}_\mu^{m, \rho}
    \le
    c_\mu \epsilon^{|\mu|}
    \left(
        \frac{r^2 + \epsilon^2}{\epsilon^2}
    \right)^{-\frac{1}{2}(m - \rho|\mu|)}
    \le
    c_\mu' \epsilon^{m + (1-\rho)|\mu|},
    \]
    and in the second case we get the same estimate with a different
    numerical constant $c_\mu''$ instead of $c_\mu$. For the behaviour
    under $\epsilon \longrightarrow 0$ these factors do not play any
    role but the sign of $m + (1-\rho)|\mu|$ does: If $\rho > 1$ then
    for large enough $|\mu|$ we get divergence and hence
    $\symnorm{\chi_\epsilon - 1}_\mu^{m, \rho}$ does not converge to
    zero. If, on the other hand $\rho \le 1$, then $m + (1-\rho)|\mu|$
    is always strictly positive. In this case we have convergence
    $\symnorm{\chi_\epsilon - 1}_\mu^{m, \rho} \longrightarrow 0$ for
    all $\mu$. This explains the condition $\rho \le 1$ and proves the
    second part. For the third part we rely on the estimates proved in
    Proposition~\ref{proposition:SymbolProducts},
    \refitem{item:ProductInOneVariable}: for $F \in
    \Symbol^{\order{m}, \type{\rho}}(\mathbb{R}^n, V)$ and a fixed
    seminorm $\halbnorm{q}$ from the defining system $\mathcal{Q}$ we
    get the estimate
    \[
    \symnorm{
      (\chi_\epsilon - 1)F
    }_{\halbnorm{q}, \mu}^{\order{m}', \type{\rho}'}
    \le
    2^{|\mu|}
    \max_{\nu \le \mu} \symnorm{\chi_\epsilon - 1}_\nu^{m, \rho}
    \max_{\nu' \le \mu}
    \symnorm{F}_{\halbnorm{q}, \nu}^{\order{m}, \type{\rho}}
    \]
    for every $m$ and $\order{m}'$ provided $\order{m}'(\halbnorm{q})
    \ge \order{m}(\halbnorm{q}) + m$, and every $\rho$ and
    $\type{\rho}'$ provided $\type{\rho}'(\halbnorm{q}) \le \min(\rho,
    \type{\rho}(\halbnorm{q}))$. Now from the second part we know that
    $\symnorm{\chi_\epsilon - 1}_\nu^{m, \rho}$ converges to zero
    whenever $\rho \le 1$ and $m > 0$. This means that for the fixed
    seminorm $\halbnorm{q}$ we get $\symnorm{(\chi_\epsilon -
      1)F}_{\halbnorm{q}, \mu}^{\order{m}', \type{\rho}'}
    \longrightarrow 0$ whenever $\order{m}'(\halbnorm{q}) >
    \order{m}(\halbnorm{q})$ and $\type{\rho}'(\halbnorm{q}) \le
    \min(1, \type{\rho}(\halbnorm{q}))$. Since this is the condition
    for every $\halbnorm{q} \in \mathcal{Q}$ we get the third
    part. Note that we are allowed to make the parameter $m$ depend on
    $\halbnorm{q}$ as long as we have $m > 0$. Thus
    $\order{m}'(\halbnorm{q}) > \order{m}(\halbnorm{q})$ does
    \emph{not} have to be uniformly satisfied. The last part is now
    clear as it suffices to take $\epsilon = \frac{1}{n}$ as usual.
\end{proof}

As the last operation to be discussed, we consider the restriction of
a symbol to a subspace of its domain of definition. To this end, we
take symbols $F:\Rl^{n_1}\oplus\Rl^{n_2}\to V$ depending on two
variables $(x_1,x_2)\in\Rl^{n_1}\oplus\Rl^{n_2}$, and introduce the
embeddings $\iota_j:\Rl^{n_j}\to\Rl^{n_1}\oplus\Rl^{n_2}$, $j=1,2$,
defined as $\iota_1(x_1):=(x_1,0)$ and $\iota_2(x_2):=(0,x_2)$. For a
symbol $F\in\Symbol^{\bm,\brho}(\Rl^{n_1}\oplus\Rl^{n_2},V)$, we write
\begin{align}
    \label{def:Fj}
    \iota_j^* F &:=  F \circ \iota_j\colon \Rl^{n_j} \to V\,,\qquad j=1,2\,.
\end{align}
\begin{lemma}\label{lemma:RestrictionOfSymbols}
    Let $\bm,\brho$ be an order and a type for $\mathcal Q$. Then the
    restriction maps
    \begin{align}
        \iota_j^*\colon
        \Symbol^{\bm,\brho}(\Rl^{n_1}\oplus\Rl^{n_2},V) &\longrightarrow
        \Symbol^{\bm,\brho}(\Rl^{n_j},V)
    \end{align}
    are linear and continuous, $j=1,2$.
\end{lemma}
\begin{proof}
    Let $F\in\Symbol^{\bm,\brho}(\Rl^{n_1}\oplus\Rl^{n_2},V)$. It is
    clear that $\iota_j^*F$ is a smooth map from $\Rl^{n_j}$ to $V$,
    and that $\iota_j^*$ is linear. For $j=1$, we estimate with
    $\qn\in\mathcal Q$, $\mu\in\Nl_0^{n_1}$,
    \begin{align*}
        \|\iota_1^*F\|^{\bm,\brho}_{\qn,\mu}
        &=
        \sup_{x_1\in\Rl^{n_1}} \frac{\qn(\partial_{x_1}^{\mu}
          F(x_1,0))}{(1+\|(x_1,0)\|^2)^{\frac{1}{2}
            (\bm(\qn)-\brho(\qn)|\mu|) } }
        \\
        &\leq
        \sup_{x\in\Rl^{n_1}\oplus\Rl^{n_2}} \frac{\qn(\partial_x^{\mu\oplus0}
          F(x))}{(1+\|x\|^2)^{\frac{1}{2}
            (\bm(\qn)-\brho(\qn)|\mu|) } }
        \\
        &=
        \|F\|^{\bm,\brho}_{\qn,\mu\oplus0}
        \,.
    \end{align*}
    This shows that $\iota_1^* F\in\Symbol^{\bm,\brho}(\Rl^{n_1},V)$,
    and that $\iota_1^*$ is continuous. The case $j=2$ is completely
    analogous.
\end{proof}

\subsection{Symbol spaces}

In this subsection, we introduce various spaces of symbols of
arbitrary order and a vector-valued Schwartz space as suitable unions
respectively intersections of the $\Symbol^{\bm,\brho}(\Rl^n,V)$. To
show that these are intrinsic definitions, we will first discuss how
our definition of the spaces $\Symbol^{\bm,\brho}(\Rl^n,V)$ depends on
the choice of the defining system of seminorms $\mathcal{Q}$.  To this
end, we shall proceed in two steps: First we show how one can pass
from an arbitrary system to a filtrating one, then we compare two
filtrating systems.

Now suppose $\mathcal{Q}$ is an arbitrary defining system of
continuous seminorms for $V$. Then we consider the larger system
\begin{equation}
    \label{eq:DefiningSystemWithMax}
    \tilde{\mathcal{Q}}
    :=
    \left\{
        \halbnorm{q} = \max\{\halbnorm{q}_1, \ldots, \halbnorm{q}_n\}
        \; \big| \;
        n \in \mathbb{N}
        \;
        \textrm{and}
        \;
        \halbnorm{q}_1, \ldots, \halbnorm{q}_n
        \in \mathcal{Q}
    \right\}
\end{equation}
which is filtrating. Suppose now that $\order{m}$ is an order with
respect to $\mathcal{Q}$. Then we want to extend $\order{m}$
to an order on $\tilde{\mathcal{Q}}$ as follows. We define
\begin{equation}
    \label{eq:MaxOrder}
    \order{m}_{\max}(\max\{\halbnorm{q}_1, \ldots, \halbnorm{q}_n\})
    :=
    \max\{
    \order{m}(\halbnorm{q}_1), \ldots, \order{m}(\halbnorm{q}_n)
    \}.
\end{equation}
Clearly, this gives an order on $\tilde{\mathcal {Q}}$ which extends
$\order{m}$. Analogously, for a type $\type{\rho}$ with respect to
$\mathcal{Q}$ we define a type $\type{\rho}_{\min}$ with respect to
$\tilde{\mathcal{Q}}$ extending $\type{\rho}$ by taking the minimum of
the types $\type{\rho}(\halbnorm{q}_i)$ instead of the maximum.
\begin{proposition}
    \label{proposition:ExtendSymbolMaxMinSystem}%
    Let $\mathcal{Q}$ be a defining system of continuous seminorms on
    $V$ and $\tilde{\mathcal{Q}}$ the corresponding filtrating system
    of finite maxima. Then for every order $\order{m}$ and every type
    $\type{\rho}$ with respect to $\mathcal{Q}$ and their
    corresponding extensions $\order{m}_{\max}$ and
    $\type{\rho}_{\min}$ to $\tilde{\mathcal{Q}}$ we have
    \begin{equation}
        \label{eq:SymbolOrderTypeMaxMinEqual}
        \Symbol^{\order{m}, \type{\rho}}
        (\mathbb{R}^n, V)
        =
        \Symbol^{\order{m}_{\max}, \type{\rho}_{\min}}
        (\mathbb{R}^n, V)
    \end{equation}
    as locally convex spaces.
\end{proposition}
\begin{proof}
    First let $F \in \Symbol^{\order{m}, \type{\rho}}(\mathbb{R}^n,
    V)$ and let $\halbnorm{q}_1, \ldots, \halbnorm{q}_n \in
    \mathcal{Q}$ be given. We set $\halbnorm{q} :=
    \max\{\halbnorm{q}_1, \ldots, \halbnorm{q}_n\}$. For $\mu \in
    \mathbb{N}_0^n$ we have the estimate
    \begin{align*}
        \symnorm{F}^{
          \order{m}_{\max}, \type{\rho}_{\min}
        }_{
          \halbnorm{q}, \mu
        }
        &=
        \sup_{x \in \mathbb{R}^n}
        \left(
            1 + \norm{x}^2
        \right)^{
          -\frac{1}{2}\left(
              \order{m}_{\max}(\halbnorm{q})
              -
              \type{\rho}_{\min}(\halbnorm{q})|\mu|
          \right)
        }
        \halbnorm{q} \left(
            \frac{\partial^{|\mu|} F}{\partial x^\mu}(x)
        \right) \\
        &\le
        \sum_{i=1}^n
        \sup_{x \in \mathbb{R}^n}
        \left(
            1 + \norm{x}^2
        \right)^{
          -\frac{1}{2}
          \left(
              \order{m}(\halbnorm{q}_i)
              -
              \type{\rho}(\halbnorm{q}_i)|\mu|
          \right)
        }
        \halbnorm{q}_i \left(
            \frac{\partial^{|\mu|} F}{\partial x^\mu}(x)
        \right) \\
        &=
        \sum_{i=1}^n
        \symnorm{F}^{\order{m}, \type{\rho}}_{\halbnorm{q}_i, \mu}.
    \end{align*}
    This shows $F \in \Symbol^{\order{m}_{\max},
      \type{\rho}_{\min}}(\mathbb{R}^n, V)$ as well as the continuity
    of the inclusion map
    \[
    \Symbol^{\order{m}, \type{\rho}}(\mathbb{R}^n, V)
    \longrightarrow
    \Symbol^{\order{m}_{\max}, \type{\rho}_{\min}}
    (\mathbb{R}^n, V).
    \]
    Conversely, let $F \in \Symbol^{\order{m}_{\max},
      \type{\rho}_{\min}}(\mathbb{R}^n, V)$ be given.  Then $F \in
    \Symbol^{\order{m}, \type{\rho}}(\mathbb{R}^n, V)$ since all the
    seminorms $\symnorm{\argument}^{\order{m},
      \type{\rho}}_{\halbnorm{q}, \mu}$ of the $\Symbol^{\order{m},
      \type{\rho}}$-topology appear also as seminorms of the
    $\Symbol^{\order{m}_{\max}, \type{\rho}_{\min}}$-topology, since
    $\mathcal{Q} \subseteq \tilde{\mathcal{Q}}$ and the order and type
    are extended to the larger system of seminorms. With respect to
    these seminorms $\symnorm{\argument}^{\order{m},
      \type{\rho}}_{\halbnorm{q}, \mu}$, the reverse inclusion
    \[
    \Symbol^{\order{m}_{\max}, \type{\rho}_{\min}}(\mathbb{R}^n, V)
    \longrightarrow
    \Symbol^{\order{m}, \type{\rho}}(\mathbb{R}^n, V)
    \]
    is even isometric and hence continuous, too. Thus we have mutually
    inverse continuous inclusions proving the claim.
\end{proof}

Next we consider two defining systems of seminorms $\mathcal{Q}$ and
$\mathcal{Q}'$ on $V$ where we can assume that they are already
filtrating. Thus for every $\halbnorm{q} \in \mathcal{Q}$ we find a
$\halbnorm{q}' \in \mathcal{Q}'$ with $\halbnorm{q} \le c
\halbnorm{q}'$ for some positive $c > 0$, and vice versa. In this
situation we have the following statement:
\begin{proposition}
    \label{proposition:QQprimeBothFiltratingSymbols}%
    Let $\mathcal{Q}$ and $\mathcal{Q}'$ be defining systems of
    seminorms for $V$ with $\mathcal{Q}'$ being filtrating. Moreover,
    let $\order{m}$, $\order{m}'$ be orders and $\type{\rho},
    \type{\rho}'$ be types for $\mathcal{Q}$, $\mathcal{Q}'$,
    respectively.  If for every $\halbnorm{q} \in \mathcal{Q}$ there
    exists a $\halbnorm{q}' \in \mathcal{Q}'$ such that
    \begin{equation}
        \label{eq:OrderTypeComparing}
        \halbnorm{q} \le c\halbnorm{q}',
        \quad
        \order{m}(\halbnorm{q}) \ge \order{m}'(\halbnorm{q}'),
        \quad
        \textrm{and}
        \quad
        \type{\rho}(\halbnorm{q}) \le \type{\rho}'(\halbnorm{q}'),
    \end{equation}
    then one has a continuous inclusion
    \begin{equation}
        \label{eq:SymbolInclusionDifferentSystems}
        \Symbol^{\order{m}', \type{\rho}'}(\mathbb{R}^n, V)
        \subseteq
        \Symbol^{\order{m}, \type{\rho}}(\mathbb{R}^n, V).
    \end{equation}
\end{proposition}
\begin{proof}
    Let $\halbnorm{q} \in \mathcal{Q}$ be given and choose
    $\halbnorm{q}'$ according to \eqref{eq:OrderTypeComparing}. Then
    for every $\mu \in \mathbb{N}_0^n$ we have
    \begin{align*}
        \symnorm{F}^{\order{m}, \type{\rho}}_{\halbnorm{q}, \mu}
        &=
        \sup_{x \in \mathbb{R}^n}
        \left(
            1 + \norm{x}^2
        \right)^{
          -\frac{1}{2}
            \left(
                \order{m}(\halbnorm{q})
                -
                \type{\rho}(\halbnorm{q})|\mu|
            \right)
        }
        \halbnorm{q} \left(
            \frac{\partial^{|\mu|} F}{\partial x^\mu}(x)
        \right) \\
        &\le
        c
        \sup_{x \in \mathbb{R}^n}
        \left(
            1 + \norm{x}^2
        \right)^{
          -\frac{1}{2}
          \left(
              \order{m}'(\halbnorm{q}')
              -
              \type{\rho}'(\halbnorm{q}')|\mu|
          \right)
        }
        \halbnorm{q}' \left(
            \frac{\partial^{|\mu|} F}{\partial x^\mu}(x)
        \right) \\
        &=
        c \symnorm{F}^{\order{m}', \type{\rho}'}_{\halbnorm{q}', \mu},
    \end{align*}
    which shows the claim.
\end{proof}
\begin{corollary}
    \label{corollary:DifferentSystemsSameSymbols}%
    Let $\mathcal{Q}$ and $\mathcal{Q}'$ be defining systems of
    continuous seminorms on $V$. Moreover, let $\order{m}'$ and
    $\type{\rho}'$ be an order and a type for $\mathcal{Q}'$. Then
    there exists an order $\order{m}$ and a type $\type{\rho}$ for
    $\mathcal{Q}$ such that
    \begin{equation}
        \label{eq:SymbolQprimeInSymbolQ}
        \Symbol^{\order{m}', \type{\rho}'} (\mathbb{R}^n, V)
        \subseteq
        \Symbol^{\order{m}, \type{\rho}} (\mathbb{R}^n, V)
    \end{equation}
    is continuously included. If in addition $\order{m}'$ or
    $\type{\rho}'$ are bounded then $\order{m}$ and $\type{\rho}$ can
    be chosen to satisfy the same bounds, respectively.
\end{corollary}
\begin{proof}
    By Proposition~\ref{proposition:ExtendSymbolMaxMinSystem} we can
    pass to a filtrating system without changing the symbol space on
    the left hand side. Thus we can assume that $\mathcal{Q}'$ is
    already filtrating without restriction.  Let $\halbnorm{q} \in
    \mathcal{Q}$, then we fix a particular choice
    $\halbnorm{q}'(\halbnorm{q}) \in \mathcal{Q}'$ with $\halbnorm{q}
    \le c \halbnorm{q}'$ for some appropriate $c > 0$. This defines a
    map $\mathcal{Q} \longrightarrow \mathcal{Q}'$, existing thanks to
    the fact that $\mathcal{Q}'$ is filtrating. Now we define
    $\order{m}(\halbnorm{q}) :=
    \order{m}'(\halbnorm{q}'(\halbnorm{q}))$ and
    $\type{\rho}(\halbnorm{q}) :=
    \type{\rho}(\halbnorm{q}'(\halbnorm{q}))$. Then clearly the
    condition \eqref{eq:OrderTypeComparing} from
    Proposition~\ref{proposition:QQprimeBothFiltratingSymbols} is
    satisfied and we get \eqref{eq:SymbolQprimeInSymbolQ}. The
    statement about the bounds is then clear.
\end{proof}
\begin{corollary}
    \label{corollary:SymbolsIndependentOfSystems}%
    Let $\mathcal{Q}$ and $\mathcal{Q}'$ be two defining systems of
    seminorms for $V$ and let $F\in \Cinfty(\mathbb{R}^n, V)$ be a
    smooth function. Then $F$ is a symbol of some (bounded) order
    $\order{m}$ and some (bounded) type $\type{\rho}$ for
    $\mathcal{Q}$ iff $F$ is a symbol of some (bounded) order
    $\order{m}'$ and some (bounded) type $\type{\rho}'$ for
    $\mathcal{Q}'$ (and the same bounds).
\end{corollary}
\begin{proof}
    By Proposition~\ref{proposition:ExtendSymbolMaxMinSystem} we can
    assume to have filtrating systems from the beginning. Since the
    extension of the order and the type according to that Proposition
    clearly preserves the bounds,
    Corollary~\ref{corollary:DifferentSystemsSameSymbols} can be
    applied in both directions.
\end{proof}

We can now use the last corollaries to speak about the space of
\emph{all} symbols: there are two alternatives whether or not we allow
for bounded orders only:
\begin{definition}
    \label{definition:SymbolsOfInfiniteOrder}%
    Let $V$ be a sequentially complete locally convex space. Then we
    define for a given type $\type{\rho}$ for a defining
    system of seminorms $\mathcal{Q}$
    \begin{equation}
        \label{eq:SymbolInftyRho}
        \Symbol^{\infty, \type{\rho}}(\mathbb{R}^n, V)
        :=
        \bigcup_{\order{m} \; \textrm{bounded}}
        \Symbol^{\order{m}, \type{\rho}}(\mathbb{R}^n, V)
    \end{equation}
    and
    \begin{equation}
        \label{eq:SymbolInftyPlusRho}
        \Symbol^{\infty+, \type{\rho}}(\mathbb{R}^n, V)
        :=
        \bigcup_{\order{m}}
        \Symbol^{\order{m}, \type{\rho}}(\mathbb{R}^n, V).
    \end{equation}
    Moreover, we set
    \begin{equation}
        \label{eq:SymbolInfty}
        \Symbol^\infty(\mathbb{R}^n, V)
        :=
        \Symbol^{\infty, 1}(\mathbb{R}^n, V)
        \,,\qquad
        \Symbol^{\infty+}(\mathbb{R}^n, V)
        :=
        \Symbol^{\infty+, 1}(\mathbb{R}^n, V)
    \end{equation}
   and
   \begin{equation}
	\label{eq:IntegrableSymbols}
   	\underline{\Symbol}(\mathbb{R}^n, V)
        :=
        \bigcup_{-1<\brho\leq 1}\Symbol^{\infty+, \brho}(\mathbb{R}^n, V)\,.
   \end{equation}
\end{definition}
It follows that for another defining system of seminorms
$\mathcal{Q}'$ we get the same spaces $\Symbol^\infty(\mathbb{R}^n,
V)$, $\Symbol^{\infty+}(\mathbb{R}^n, V)$,
$\underline{\Symbol}(\Rl^n,V)$, which are therefore intrinsically
defined. The space $\IntSymbol(\Rl^n,V)$ will be relevant in the
context of oscillatory integrals, presented in
Section~\ref{section:integrals}.

Note that for a Banach space $V$ with
$\mathcal{Q}$ consisting just of the norm itself we have
$\Symbol^{\infty, \rho}(\mathbb{R}^n, V) = \Symbol^{\infty+,
  \rho}(\mathbb{R}^n, V)$ for all types $\rho \in
\mathbb{R}$. However, in general we have a proper inclusion
\begin{equation}
    \label{eq:SymbolInftyInSymbolInftyPlus}
    \Symbol^{\infty, \type{\rho}}(\mathbb{R}^n, V)
    \subset
    \Symbol^{\infty+, \type{\rho}}(\mathbb{R}^n, V).
\end{equation}

Also the intersection of all the symbol spaces is of interest: here we
get an analog of the usual Schwartz space. First we note the following
simple facts:
\begin{lemma}
    \label{lemma:SchwartzVectorValued}%
    Let $V$ be a sequentially complete locally convex space and $F\in
    \Cinfty(\mathbb{R}^n, V)$. Then the following statements are
    equivalent:
    \begin{lemmalist}
    \item \label{item:SchwartzSeminormqFinite} For all continuous
        seminorms $\halbnorm{q}$ of a defining system $\mathcal{Q}$,
        for all $\mu \in \mathbb{N}_0^n$ and all $m \in \mathbb{N}_0$
        one has
        \begin{equation}
            \label{eq:SchwartzSeminormqmmu}
            \halbnorm{q}_{m, \mu} (F)
            =
            \sup_{x \in \mathbb{R}^n}
            \left(1 + \norm{x}^2\right)^{\frac{m}{2}}
            \halbnorm{q}\left(
                \frac{\partial^{|\mu|} F}{\partial x^\mu}(x)
            \right)
            < \infty.
        \end{equation}
    \item \label{item:AllOrdersAllTypesSchwartz} For all orders
        $\order{m}$ and all types $\type{\rho}$ for a
        given defining system $\mathcal{Q}$ of continuous seminorms
        one has $F \in \Symbol^{\order{m}, \type{\rho}}(\mathbb{R}^n,
        V)$.
    \item \label{item:AllOrdersOneTypeSchwartz} For all orders
        $\order{m}$ and one type $\type{\rho}$ for a
        given defining system $\mathcal{Q}$ of continuous seminorms
        one has $F \in \Symbol^{\order{m}, \type{\rho}}(\mathbb{R}^n,
        V)$.
    \end{lemmalist}
\end{lemma}
\begin{proof}
    First we note that if \refitem{item:SchwartzSeminormqFinite} holds
    for a defining system of seminorms $\mathcal{Q}$ then it also
    holds for \emph{all} continuous seminorms of $V$. This is
    clear. Thus assume \refitem{item:SchwartzSeminormqFinite} and let
    $\mathcal{Q}$ be a defining system of seminorms. Moreover, fix an
    order $\order{m}$ and a type $\type{\rho}$ for this system
    $\mathcal{Q}$. Then for $\mu \in \mathbb{N}_0^n$ we have
    \[
    \symnorm{F}_{\halbnorm{q}, \mu}^{\order{m}, \type{\rho}}
    =
    \sup_{x \in \mathbb{R}^n}
    \left(
        1 + \norm{x}^2
    \right)^{-\frac{1}{2}\left(
          \order{m}(\halbnorm{q}) - \type{\rho}(\halbnorm{q})|\mu|
      \right)
    }
    \halbnorm{q}\left(
        \frac{\partial^{|\mu|}F}{\partial x^\mu}(x)
    \right)
    \le
    \halbnorm{q}_{m, \mu}(F),
    \]
    with $m$ being any integer larger than $- \order{m}(\halbnorm{q})
    + \type{\rho}(\halbnorm{q})|\mu|$. This shows that $F \in
    \Symbol^{\order{m}, \type{\rho}}(\mathbb{R}^n, V)$. The
    implication \refitem{item:AllOrdersAllTypesSchwartz} $\Rightarrow$
    \refitem{item:AllOrdersOneTypeSchwartz} is trivial. Thus assume
    \refitem{item:AllOrdersOneTypeSchwartz} and hence let $F \in
    \Symbol^{\order{m}, \type{\rho}}(\mathbb{R}^n, V)$ for all orders
    $\order{m}$ and a fixed type $\type{\rho}$. Then let $m \in
    \mathbb{N}_0$ and $\mu \in \mathbb{N}_0^n$ be given. We have
    \begin{align*}
        \halbnorm{q}_{m, \mu}(F)
        &=
        \sup_{x \in \mathbb{R}^n}
        \left(1 + \norm{x}^2\right)^{\frac{m}{2}}
        \halbnorm{q}\left(
            \frac{\partial^{|\mu|}F}{\partial x^\mu}(x)
        \right) \\
        &\le
        \sup_{x \in \mathbb{R}^n}
        \left(
            1 + \norm{x}^2
        \right)^{-\frac{1}{2}\left(
              \order{m}(\halbnorm{q}) - \type{\rho}(\halbnorm{q})|\mu|
          \right)
        }
        \halbnorm{q}\left(
            \frac{\partial^{|\mu|}F}{\partial x^\mu}(x)
        \right) \\
        &=
        \symnorm{F}_{\halbnorm{q}, \mu}^{\order{m}, \type{\rho}},
    \end{align*}
    where we have to choose an order $\order{m}$ such that
    $\order{m}(\halbnorm{q}) - \type{\rho}(\halbnorm{q})|\mu| \le -
    m$. This is clearly possible as we can e.g.\ take the constant
    order with $\order{m} = - m +
    \type{\rho}(\halbnorm{q})|\mu|$. Thus
    \refitem{item:SchwartzSeminormqFinite} follows.
\end{proof}

Thus for the intersection of all symbol spaces the type $\type{\rho}$
does not play any role any more. Also the dependence on the chosen
system of seminorms $\mathcal{Q}$ disappears. This motivates the
following definition:
\begin{definition}[Vector-valued Schwartz space]
    \label{definition:VectorValuedSchwartzSpace}%
    Let $V$ be a sequentially complete locally convex space. Then we
    define the symbols of order $-\infty$ by
    \begin{equation}
        \label{eq:OrderMinusInftySymbols}
        \Symbol^{-\infty}(\mathbb{R}^n, V)
        :=
        \bigcap_{\order{m}, \type{\rho}}
        \Symbol^{\order{m}, \type{\rho}}(\mathbb{R}^n, V).
    \end{equation}
    We also use the notation
    \begin{equation}
        \label{eq:VectorValuedSchwartzSpaceDef}
        \Schwartz(\mathbb{R}^n, V)
        :=
        \Symbol^{-\infty}(\mathbb{R}^n, V)
    \end{equation}
    and call $\Schwartz(\mathbb{R}^n, V)$ the space of $V$-valued
    Schwartz functions.
\end{definition}

Clearly, the $V$-valued Schwartz functions are a straightforward
generalization of the scalar case. The Schwartz space
$\Schwartz(\mathbb{R}^n, V)$ will always be endowed with the locally
convex topology determined by the seminorms $\halbnorm{q}_{m, \mu}$ as
in \eqref{eq:SchwartzSeminormqmmu}. We call this the
$\Symbol^{-\infty}$- or the Schwartz topology of
$\Schwartz(\mathbb{R}^n, V)$. We collect now some easy properties of
the Schwartz space:
\begin{proposition}
    \label{proposition:VectorValuedSchwartz}%
    Let $V$ be a sequentially complete locally convex space with a
    defining system of seminorms $\mathcal{Q}$.
    \begin{propositionlist}
    \item \label{item:SchwartzVectorComplete} The Schwartz space
        $\Schwartz(\mathbb{R}^n, V)$ is sequentially complete and even
        complete when $V$ is complete.
    \item \label{item:SchwartzInSymbols} We have continuous inclusions
        \begin{equation}
            \label{eq:SchwartzInSymbols}
            \Cinfty_0(\mathbb{R}^n, V)
            \longrightarrow
            \Schwartz(\mathbb{R}^n, V)
            \longrightarrow
            \Symbol^{\order{m}, \type{\rho}}(\mathbb{R}^n, V)
        \end{equation}
        for all orders $\order{m}$ and all types $\type{\rho}$ for
        $\mathcal{Q}$.
    \item \label{item:CinftyNullDenseInSchwartz}
        $\Cinfty_0(\mathbb{R}^n, V)$ is sequentially dense in
        $\Schwartz(\mathbb{R}^n, V)$ and $\Schwartz(\mathbb{R}^n, V)$
        is sequentially dense in $\Symbol^{\order{m},
          \type{\rho}}(\mathbb{R}^n, V)$ in the $\Symbol^{\order{m}',
          \type{\rho}'}$-topology whenever $\order{m}' > \order{m}$
        and $\type{\rho'} \le \min(1, \type{\rho})$.
    \item \label{item:PartialDerivativesSchwartz} The partial
        derivatives are continuous linear maps
        \begin{equation}
            \label{eq:PartialDerivativeSchwartz}
            \frac{\partial^{|\nu|}}{\partial x^\nu} \colon
            \Schwartz(\mathbb{R}^n, V)
            \longrightarrow
            \Schwartz(\mathbb{R}^n, V)
        \end{equation}
        satisfying the estimate (equality)
        \begin{equation}
            \label{eq:SchwartzEstimatePartialDerivatives}
            \halbnorm{q}_{m, \mu}
            \left(
                \frac{\partial^{|\nu|} F}{\partial x^\nu}
            \right)
            =
            \halbnorm{q}_{m, \mu + \nu}(F).
        \end{equation}
    \item \label{item:BilinearSchwartz} If $W$ and $U$ are two other
        sequentially complete locally convex spaces and $\mu\colon V
        \times W \longrightarrow U$ is a continuous bilinear map then
        it induces continuous bilinear maps
        \begin{equation}
            \label{eq:SchwartzSchwartzSchwartz}
           \mu\colon
            \Schwartz(\mathbb{R}^n, V)
            \times
            \Schwartz(\mathbb{R}^n, W)
            \longrightarrow
            \Schwartz(\mathbb{R}^n, U),
        \end{equation}
        \begin{equation}
            \label{eq:SymbolSchwartzSchwartz}
            \mu\colon
            \Symbol^{\order{m}, \type{\rho}}(\mathbb{R}^n, V)
            \times
            \Schwartz(\mathbb{R}^n, W)
            \longrightarrow
            \Schwartz(\mathbb{R}^n, U),
        \end{equation}
        and
        \begin{equation}
            \label{eq:SchwartzSymbolSchwartz}
            \mu\colon
            \Schwartz(\mathbb{R}^n, V)
            \times
            \Symbol^{\order{m}', \type{\rho}'}(\mathbb{R}^n, W)
            \longrightarrow
            \Schwartz(\mathbb{R}^n, U)
        \end{equation}
        for all orders $\order{m}$ and types $\type{\rho}$ for
        $\mathcal{Q}$ and all orders $\order{m}'$ and types
        $\type{\rho}'$ for some defining system of seminorms
        $\mathcal{Q}'$ for $W$.
    \item \label{item:SchwartzModule} For all orders $m \in
        \mathbb{R}$ and types $\rho \in \mathbb{R}$ the pointwise
        multiplication
        \begin{equation}
            \label{eq:SchwartzModule}
            \Symbol^{m, \rho}(\mathbb{R}^n,\Cl)
            \times
            \Schwartz(\mathbb{R}^n, V)
            \longrightarrow
            \Schwartz(\mathbb{R}^n, V)
        \end{equation}
        is a continuous bilinear map.
    \end{propositionlist}
\end{proposition}
\begin{proof}
    The first statement can most easily be checked using the explicit
    seminorms $\halbnorm{q}_{m, \mu}$ in the same spirit as the proof
    of Proposition~\ref{proposition:SymbolFirstProperties},
    \refitem{item:SymbolsComplete}. Then also the second part is clear
    since we get a continuous inclusion of $\Cinfty_K(\mathbb{R}^n,
    V)$ into $\Schwartz(\mathbb{R}^n, V)$ with estimates like
    \[
    \halbnorm{q}_{m, \mu}(F)
    =
    \symnorm{F}^{m, 0}_{\halbnorm{q}, \mu}
    \le
    c_K \halbnorm{p}_{K, |\mu|, \halbnorm{q}}(F)
    \]
    as in the proof of
    Proposition~\ref{proposition:SymbolFirstProperties},
    \refitem{item:SymbolInclusion}, for every compact subset $K
    \subseteq \mathbb{R}^n$. The second inclusion is continuous thanks
    to the estimate $\symnorm{F}^{\order{m},
      \type{\rho}}_{\halbnorm{q}, \mu}(F) \le \halbnorm{q}_{m,
      \mu}(F)$ for $m$ an integer being at least $-
    \order{m}(\halbnorm{q}) + \type{\rho}(\halbnorm{q})|\mu|$, which
    we have established in the proof of
    Lemma~\ref{lemma:SchwartzVectorValued}. The density of
    $\Cinfty_0(\mathbb{R}^n, V)$ in $\Schwartz(\mathbb{R}^n, V)$ is
    checked directly as in the scalar case. The second statement of
    part~\refitem{item:CinftyNullDenseInSchwartz} is clear as
    $\Cinfty_0(\mathbb{R}^n, V)$ has this property by
    Proposition~\ref{proposition:ApproximateSymbols},
    \refitem{item:CinftyDenseInSymbols}. The fourth part is clear
    since the estimate \eqref{eq:SchwartzEstimatePartialDerivatives}
    is obvious by definition. For part~\refitem{item:BilinearSchwartz}
    we first consider the case \eqref{eq:SymbolSchwartzSchwartz}. Thus
    let $F \in \Symbol^{\order{m}, \type{\rho}}(\mathbb{R}^n, V)$ and
    $G \in \Schwartz(\mathbb{R}^n, W)$ be given. The continuity of $\mu$
    means that for a continuous seminorm $\halbnorm{r}$ on $U$ we find
    $\halbnorm{q} \in \mathcal{Q}$ and a continuous seminorm
    $\halbnorm{q}'$ on $W$ such that
    \[
    \halbnorm{r}(\mu(v, w)) \le \halbnorm{q}(v) \halbnorm{q}'(w)
    \]
    for all $v \in V$ and $w \in W$ since we can assume that the
    defining system $\mathcal{Q}$ on $V$ is already filtrating by
    Proposition~\ref{proposition:ExtendSymbolMaxMinSystem}.  Then for
    $m \in \mathbb{N}_0$ and $\kappa\in \mathbb{N}_0^n$ we estimate
    using the Leibniz rule
    \begin{align*}
        \halbnorm{r}_{m, \kappa}(\mu(F,G))
        &=
        \sup_{x \in \mathbb{R}^n}
        \left(1 + \norm{x}^2\right)^{\frac{m}{2}}
        \halbnorm{r}\left(
            \frac{\partial^{|\kappa|} \mu(F, G)}{\partial x^\kappa} (x)
        \right) \\
        &=
        \sup_{x \in \mathbb{R}^n}
        \left(1 + \norm{x}^2\right)^{\frac{m}{2}}
        \halbnorm{r}\left(
            \sum_{\nu  +\nu' = \kappa}
            \binom{\kappa}{\nu}
            \mu \left(
                \frac{\partial^{|\nu|} F}{\partial x^\nu} (x),
                \frac{\partial^{|\nu'|} G}{\partial x^{\nu'}} (x)
            \right)
        \right) \\
        &\le
        2^{|\kappa|}
	\sup_{x\in\mathbb{R}^n}
        \sum_{\nu  +\nu' = \kappa}
        \left(1 + \norm{x}^2\right)^{\frac{m}{2}}
        \halbnorm{q}\left(
            \frac{\partial^{|\nu|} F}{\partial x^\nu}(x)
        \right)
        \halbnorm{q}'\left(
            \frac{\partial^{|\nu'|} G}{\partial x^{\nu'}}(x)
        \right) \\
        &=
        2^{|\kappa|}
        \sup_{x\in\mathbb{R}^n}
        \sum_{\nu  +\nu' = \kappa}
        \left(
            1 + \norm{x}^2
        \right)^{
          -\frac{1}{2}(\order{m}(\halbnorm{q}) -
          \type{\rho}(\halbnorm{q})|\nu|)
        }
        \halbnorm{q}\left(
            \frac{\partial^{|\nu|} F}{\partial x^\nu}(x)
        \right) \\
        &\qquad\qquad
        \left(
            1 + \norm{x}^2
        \right)^{\frac{m}{2}
          +\frac{1}{2}(\order{m}(\halbnorm{q}) -
          \type{\rho}(\halbnorm{q})|\nu|)
        }
        \halbnorm{q}'\left(
            \frac{\partial^{|\nu'|} G}{\partial x^{\nu'}}(x)
        \right) \\
        &\le
        2^{|\kappa|}
        \sum_{\nu  +\nu' = \kappa}
        \symnorm{F}^{\order{m}, \type{\rho}}_{\halbnorm{q}, \nu}
        \halbnorm{q}'_{m', \nu'} (G),
    \end{align*}
    with $m' = m + \order{m}(\halbnorm{q}) -
    \type{\rho}(\halbnorm{q})|\nu|$. This shows the continuity of
    \eqref{eq:SymbolSchwartzSchwartz} and
    \eqref{eq:SchwartzSymbolSchwartz} is analogous. But then
    \eqref{eq:SchwartzSchwartzSchwartz} follows as well since
    $\Schwartz(\mathbb{R}^n, V) \longrightarrow \Symbol^{\order{m},
      \type{\rho}}(\mathbb{R}^n, V)$ is continuous. In fact, one can
    also estimate the bilinear expression directly. The last part is
    then clear.
\end{proof}
\begin{corollary}
    \label{corollary:SchwartzAlgebra}%
    Let $\mathcal{A}$ be a (sequentially) complete locally convex
    algebra and let $\mathcal{M}$ be a (sequentially) complete locally
    convex topological module over $\mathcal{A}$. Then
    $\Schwartz(\mathbb{R}^n, \mathcal{A})$ is a (sequentially)
    complete locally convex algebra and $\Schwartz(\mathbb{R}^n,
    \mathcal{M})$ is a (sequentially) complete locally convex
    topological module over $\Schwartz(\mathbb{R}^n, \mathcal{A})$.
\end{corollary}

\subsection{Affine symmetries and symbol-valued symbols}
\label{subsec:FurtherPropertiesSymbols}

In this subsection we investigate the action of the affine group of
$\mathbb{R}^n$ on the symbol spaces. For $A\in
\group{GL}_n(\mathbb{R})$ and a translation $y \in \mathbb{R}^n$, we
denote their pullback by $(A^*F)(x) = F(Ax)$ and $(\tau_y^* F)(x) =
F(x + y)$ as usual. We start with the following basic observations:
\begin{lemma}
    \label{lemma:GlnActsOnSymbols}%
    Let $\halbnorm{q}$ be a continuous seminorm on $V$ and $m, \rho
    \in \mathbb{R}$. Then for $F \in \Cinfty(\mathbb{R}^n, V)$ we have
    for all $\mu \in \mathbb{N}_0^n$ and all $A \in
    \group{GL}_n(\mathbb{R})$
    \begin{equation}
        \label{eq:symnormAstarf}
        \symnorm{A^*F}^{m, \rho}_{\halbnorm{q}, \mu}
        \le
        c^{m, \rho}_\mu(A)
	\sum_{\substack{\nu\in\mathbb{N}^n_0 \\ |\nu| = |\mu| }}
        \symnorm{F}^{m, \rho}_{\halbnorm{q}, \nu}
    \end{equation}
    with some $c^{m, \rho}_\mu(A) > 0$ depending continuously on $A$
    and satisfying $c^{m, \rho}_\mu(\Unit) = 1$.
\end{lemma}
\begin{proof}
    As usual, this is to be understood as an inequality in $[0,
    +\infty]$. First we note that with the operator norm of $A$ we
    have
    \[
    \halbnorm{q}\left(
        \frac{\partial^{|\mu|} (A^*F)}{\partial x^\mu}(x)
    \right)
    \le
    \norm{A}^{|\mu|}
    \sum_{\substack{\nu\in\mathbb{N}^n_0 \\ |\nu| = |\mu| }}
    \halbnorm{q}\left(
        \frac{\partial^{|\nu|} F}{\partial x^\nu}(Ax)
    \right)
    \]
    for all $x \in \mathbb{R}^n$ by the chain rule. Next we recall
    that
    \[
    \frac{1}{\norm{A}} \norm{y}
    \le
    \norm{A^{-1} y}
    \le
    \norm{A^{-1}} \norm{y}
    \]
    for an invertible $A \in \group{GL}_n(\mathbb{R})$. We have to
    distinguish a few cases. Depending on the sign of $m - \rho|\mu|$
    we first get
    \[
    \left(
        1 + \norm{A^{-1}y}^2
    \right)^{-\frac{1}{2}(m - \rho|\mu|)}
    \le
    \begin{cases}
        \left(
            1 + \norm{A^{-1}}^2 \norm{y}^2
        \right)^{-\frac{1}{2}(m - \rho|\mu|)}
        & \textrm{for} \; m - \rho|\mu| < 0 \\
        \left(
            1 + \norm{A}^{-2} \norm{y}^2
        \right)^{-\frac{1}{2}(m - \rho|\mu|)}
        & \textrm{for} \; m - \rho|\mu| \ge 0.
    \end{cases}
    \]
    In the case $\norm{A^{-1}} \le 1$ and hence $\norm{A} \ge 1$ we
    can continue the estimate by
    \[
    \left(
        1 + \norm{A^{-1}y}^2
    \right)^{-\frac{1}{2}(m - \rho|\mu|)}
    \le
    \begin{cases}
        \left(
            1 + \norm{y}^2
        \right)^{-\frac{1}{2}(m - \rho|\mu|)}
        & \textrm{for} \; m - \rho|\mu| < 0 \\
        \norm{A}^{m - \rho|\mu|}
        \left(
            1 + \norm{y}^2
        \right)^{-\frac{1}{2}(m - \rho|\mu|)}
        & \textrm{for} \; m - \rho|\mu| \ge 0.
    \end{cases}
    \tag{$*$}
    \]
    Conversely, in the case $\norm{A^{-1}} \ge 1$ we get
    \[
    \left(
        1 + \norm{A^{-1}y}^2
    \right)^{-\frac{1}{2}(m - \rho|\mu|)}
    \le
    \norm{A^{-1}}^{- m + \rho|\mu|}
    \left(
        1 + \norm{y}^2
    \right)^{-\frac{1}{2}(m - \rho|\mu|)}
    \tag{$**$}
    \]
    for the case $m - \rho|\mu| < 0$. For $m - \rho|\mu| \ge 0$ we
    still have to distinguish the two possibilities $\norm{A} \le 1$
    and $\norm{A} \ge 1$. For $\norm{A} \le 1$ we have
    \[
    \left(
        1 + \norm{A^{-1}y}^2
    \right)^{-\frac{1}{2}(m - \rho|\mu|)}
    \le
    \left(
        1 + \norm{y}^2
    \right)^{-\frac{1}{2}(m - \rho|\mu|)},
    \tag{$\star$}
    \]
    and in the case $\norm{A} > 1$ we finally get
    \[
    \left(
        1 + \norm{A^{-1}y}^2
    \right)^{-\frac{1}{2}(m - \rho|\mu|)}
    \le
    \norm{A}^{m - \rho|\mu|}
    \left(
        1 + \norm{y}^2
    \right)^{-\frac{1}{2}(m - \rho|\mu|)}.
    \tag{$\star\star$}
    \]
    Combining the four cases ($*$), ($**$), ($\star$), and
    ($\star\star$) we get the estimate
    \[
    \norm{A}^{|\mu|}
    \left(
        1 + \norm{A^{-1}y}^2
    \right)^{-\frac{1}{2}(m - \rho|\mu|)}
    \le
    c^{m, \rho}_\mu(A)
    \left(
        1 + \norm{y}^2
    \right)^{-\frac{1}{2}(m - \rho|\mu|)}
    \]
    where
    \[
    c^{m, \rho}_\mu(A)
    =
    \norm{A}^{|\mu|}
    \max\left\{
        1, \norm{A}^{m - \rho|\mu|}, \norm{A^{-1}}^{-m + \rho|\mu|}
    \right\}.
    \]
    Since $A \mapsto A^{-1}$ is continuous
    and since the operator norm is continuous as well this constant
    depends continuously on $A$ and clearly satisfies $c^{m,
      \rho}_\mu(\Unit) = 1$. It is now easy to see that we get the
    estimate \eqref{eq:symnormAstarf}.
\end{proof}
\begin{lemma}
    \label{lemma:TranslationsActOnSymbols}%
    Let $\halbnorm{q}$ be a continuous seminorm on $V$ and $m, \rho
    \in \mathbb{R}$. Then for $F\in \Cinfty(\mathbb{R}^n, V)$ we have
    for all $\mu \in \mathbb{N}_0^n$ and all $y \in \mathbb{R}^n$
    \begin{equation}
        \label{eq:EstimateTranslationSymbolNorm}
        \symnorm{\tau_y^* F}^{m, \rho}_{\halbnorm{q}, \mu}
        \le
        c^{m, \rho}_{\mu}(y) \norm{F}^{m, \rho}_{\halbnorm{q}, \mu},
    \end{equation}
    with some positive $c^{m, \rho}_\mu (y) > 0$ being a scalar symbol
    $c^{m, \rho}_\mu \in \Symbol^{|m - \rho|\mu||, 1}(\mathbb{R}^n)$.
\end{lemma}
\begin{proof}
    We proceed similar as in the previous lemma. First it is clear
    that
    \[
    \halbnorm{q}\left(
        \frac{\partial^{|\mu|} (\tau_y^* F)}{\partial x^\mu}(x)
    \right)
    =
    \halbnorm{q}\left(
        \frac{\partial^{|\mu|} F}{\partial x^\mu}(x+y)
    \right)
    \tag{$*$}
    \]
    by the chain rule. For the prefactor in the definition of the
    seminorm we first consider the following elementary estimate:
    There is a constant $c > 0$ such that for all $x, y \ge 0$ we have
    \begin{align}\label{eq:QuadraticBound}
        \frac{1}{1 + (x-y)^2} \le c \frac{1+y^2}{1+x^2}.
        \tag{$**$}
    \end{align}
    We now use this inequality to consider first the case $m - \rho|\mu| \ge
    0$. There we have
    \begin{align*}
        \left(
            1 + \norm{x-y}^2
        \right)^{-\frac{1}{2}(m - \rho|\mu|)}
        &\le
        \frac{1}{
          \left(
              1 +
              \left|
                  \norm{x} - \norm{y}
              \right|^2
          \right)^{\frac{1}{2}(m - \rho|\mu|)}
        } \\
        &\stackrel{(**)}{\le}
        c^{\frac{1}{2}(m - \rho|\mu|)}
        \left(
            \frac{1 + \norm{y}^2}{1 + \norm{x}^2}
        \right)^{\frac{1}{2}(m - \rho|\mu|)}.
    \end{align*}
    This gives the estimate
    \begin{align*}
        \symnorm{\tau_y^* F}^{m, \rho}_{\halbnorm{q}, \mu}
        &\stackrel{\mathcal{(*)}}{=}
        \sup_{x \in \mathbb{R}^n}
        \left(
            1 + \norm{x}^2
        \right)^{-\frac{1}{2}(m - \rho|\mu|)}
        \halbnorm{q}\left(
            \frac{\partial^{|\mu|} F}{\partial x^\mu}
            (x + y)
        \right) \\
        &=
        \sup_{x \in \mathbb{R}^n}
        \left(
            1 + \norm{x- y}^2
        \right)^{-\frac{1}{2}(m - \rho|\mu|)}
        \halbnorm{q}\left(
            \frac{\partial^{|\mu|} F}{\partial x^\mu}
            (x)
        \right) \\
        &\le
        c^{\frac{1}{2}(m - \rho|\mu|)}
        \sup_{x \in \mathbb{R}^n}
        \left(
            \frac{1 + \norm{y}^2}{1 + \norm{x}^2}
        \right)^{\frac{1}{2}(m - \rho|\mu|)}
        \halbnorm{q}\left(
            \frac{\partial^{|\mu|} F}{\partial x^\mu}
            (x)
        \right) \\
        &=
        c^{\frac{1}{2}(m - \rho|\mu|)}
        \left(
            1 + \norm{y}^2
        \right)^{\frac{1}{2}(m - \rho|\mu|)}
        \symnorm{F}^{m, \rho}_{\halbnorm{q}, \mu}.
    \end{align*}
    The case $m - \rho|\mu| < 0$ is even simpler. Here we have
    $-\frac{1}{2}(m - \rho|\mu|) \ge 0$ and hence
    \begin{align*}
        \left(
            1 + \norm{x-y}^2
        \right)^{-\frac{1}{2}(m - \rho|\mu|)}
        &\le
        \left(
            1 + 2\norm{x}^2 + 2\norm{y}^2
        \right)^{-\frac{1}{2}(m - \rho|\mu|)} \\
        &\le
        2^{-\frac{1}{2}(m - \rho|\mu|)}
        \left(
            1 + \norm{x}^2 + \norm{y}^2
        \right)^{-\frac{1}{2}(m - \rho|\mu|)} \\
        &\le
        2^{-\frac{1}{2}(m - \rho|\mu|)}
        \left(
            1 + \norm{x}^2
        \right)^{-\frac{1}{2}(m - \rho|\mu|)}
        \left(
            1 + \norm{y}^2
        \right)^{-\frac{1}{2}(m - \rho|\mu|)}.
    \end{align*}
    By an analogous argument as for the previous case this results in
    the estimate
    \[
    \symnorm{\tau_y^* F}^{m, \rho}_{\halbnorm{q}, \mu}
    \le
    2^{-\frac{1}{2}(m - \rho|\mu|)}
    \left(
        1 + \norm{y}^2
    \right)^{-\frac{1}{2}(m - \rho|\mu|)}
    \symnorm{F}^{m, \rho}_{\halbnorm{q}, \mu}.
    \]
    So for each $m,\rho\in\Rl$, $\mu\in\Nl_0^n$, we find a constant
    $C_\mu^{m,\rho}>0$ such that $\|\tau_y^* F\|^{m,
      \rho}_{\halbnorm{q}, \mu} \le
    C_\mu^{m,\rho}(1+\norm{y}^2)^{\frac{1}{2}|m - \rho|\mu||}
    \symnorm{F}^{m, \rho}_{\halbnorm{q}, \mu}$ for all $y\in\Rl^n$. As
    the function $y \mapsto 1 + \norm{y}^2$ is clearly in
    $\Symbol^{2,1}(\mathbb{R}^n, \mathbb{C})$, an application of
    Lemma~\ref{lemma:ScalarSymbolPower} yields $y\mapsto
    (1+\norm{y}^2)^{\frac{1}{2}|m - \rho|\mu||}\in\Symbol^{|m -
      \rho|\mu||,1}(\Rl^n,\Cl)$, and the proof is finished.
\end{proof}
\begin{remark}
    \label{remark:TauyPositiveOrderSymbol}%
    Note that for the translations $\tau_y$ the pre-factor $c^{m,
      \rho}_\mu(y)$ in \eqref{eq:EstimateTranslationSymbolNorm} is
    always a symbol of \emph{non-negative} order, even if $m -
    \rho|\mu|$ was negative. Thus the bounds in
    \eqref{eq:EstimateTranslationSymbolNorm} typically \emph{grow}
    with $y$ and are also growing with increasing differentiations
    $\mu$ unless $\rho = 0$.
\end{remark}
As an easy consequence of the two lemmas we get the affine invariance
of the symbol spaces:
\begin{proposition}
    \label{proposition:AffineInvarianceOfSymbols}%
    Let $V$ be a sequentially complete locally convex space and
    $\mathcal{Q}$ a defining system of seminorms. Moreover, let
    $\order{m}$ and $\type{\rho}$ be an order and a type for
    $\mathcal{Q}$. Then the affine group $\group{GL}_n(\mathbb{R})
    \ltimes \mathbb{R}^n$ of $\mathbb{R}^n$ acts on the symbols
    $\Symbol^{\order{m}, \type{\rho}}(\mathbb{R}^n, V)$ via pull-backs
    by continuous endomorphisms.
\end{proposition}
\begin{proof}
    The pull-backs with $A \in \group{GL}_n(\mathbb{R})$ or with a
    translation by $y \in \mathbb{R}^n$ map $\Symbol^{\order{m},
      \type{\rho}}(\mathbb{R}^n, V)$ continuously into itself
    according to Lemma~\ref{lemma:GlnActsOnSymbols} and
    Lemma~\ref{lemma:TranslationsActOnSymbols}, respectively. The fact
    that this gives a (right) group action is clear.
\end{proof}

In a next step we want to refine this statement for the translations:
we want to show that the map $y \mapsto \tau_y^*F$ is actually
smooth. We begin with the following observation:
\begin{lemma}
    \label{lemma:tauySymbolContinuous}%
    Let $F \in \Symbol^{\order{m}, \type{\rho}}(\mathbb{R}^n,
    V)$. Then the map
    \begin{equation}
        \label{eq:TranslationMapSymbol}
        \mathbb{R}^n \ni y
        \; \mapsto \;
        \tau_y^* F \in \Symbol^{\order{m}, \type{\rho}}(\mathbb{R}^n, V)
    \end{equation}
    is continuous at zero provided $\type{\rho} \ge 0$.
\end{lemma}
\begin{proof}
    We have to show that $\tau_y^* F \longrightarrow F$ in the
    $\Symbol^{\order{m}, \type{\rho}}$-topology for $y \longrightarrow
    0$. Let $x \in \mathbb{R}^n$ be given. Then we have for $y \in
    \mathbb{R}^n$ and $\halbnorm{q} \in \mathcal{Q}$ by virtue of the
    mean value theorem
    \begin{align*}
        \halbnorm{q}\left(
            \frac{\partial^{|\mu|} (\tau_y^* F)}{\partial x^\mu} (x)
            -
            \frac{\partial^{|\mu|} F}{\partial x^\mu} (x)
        \right)
        &=
        \halbnorm{q}\left(
            \int_0^1
            \left(
                \frac{\partial}{\partial x^i}
                \frac{\partial^{|\mu|} F}{\partial x^\mu}
            \right)(x + t y) \D t y^i
        \right) \\
        &\le
        \sqrt{n}
        \sup_{\substack{t \in [0, 1] \\ i = 1, \ldots, n}}
        \halbnorm{q}
        \left(
            \left(
                \frac{\partial}{\partial x^i}
                \frac{\partial^{|\mu|} F}{\partial x^\mu}
            \right)(x + t y)
        \right)
        \norm{y}
        \tag{$*$}
      .
    \end{align*}
    Now we use the fact that $F$ is a symbol. This means that the
    $(\mu + e_i)$-th derivative of $F$ satisfies
    \[
    \halbnorm{q}
    \left(
        \left(
            \frac{\partial}{\partial x^i}
            \frac{\partial^{|\mu|} F}{\partial x^\mu}
        \right)(x + t y)
    \right)
    \le
    \left(
        1 + \norm{x + ty}^2
    \right)^{\frac{1}{2}(m - \rho|\mu| - \rho)} \symnorm{F}^{m,
      \rho}_{\halbnorm{q}, \mu + e_i},
    \tag{$**$}
    \]
    where for abbreviation we have set $m = \order{m}(\halbnorm{q})$
    and $\rho = \type{\rho}(\halbnorm{q})$. Then we get
    \begin{align*}
        &\symnorm{\tau_y^* F - F}^{\order{m}, \type{\rho}
        }_{\halbnorm{q}, \mu} \\
        &\quad=
        \sup_{x \in \mathbb{R}^n}
        \left(
            1 + \norm{x}^2
        \right)^{-\frac{1}{2}(m - \rho|\mu|)}
        \halbnorm{q}\left(
            \frac{\partial^{|\mu|} (\tau_y^* F)}{\partial x^\mu} (x)
            -
            \frac{\partial^{|\mu|} F}{\partial x^\mu} (x)
        \right) \\
        &\quad\stackrel{\mathclap{(*), (**)}}{\le}
        \quad
        \sqrt{n}
        \norm{y}
        \sup_{\substack{x \in \mathbb{R}^n \\
            t \in [0, 1] \\
            i = 1, \ldots, n}
        }
        \left(
            1 + \norm{x}^2
        \right)^{-\frac{1}{2}(m - \rho|\mu|)}
        \left(
            1 + \norm{x + ty}^2
        \right)^{\frac{1}{2}(m - \rho|\mu| - \rho)}
        \symnorm{F}^{\order{m}, \type{\rho}}_{\halbnorm{q}, \mu + e_i}
        \\
        &\quad=
        \sqrt{n}
        \norm{y}
        \sup_{\substack{x \in \mathbb{R}^n \\
            t \in [0, 1] \\
            i = 1, \ldots, n}
        }
        \left(
            1 + \norm{x - ty}^2
        \right)^{-\frac{1}{2}(m - \rho|\mu|)}
        \left(
            1 + \norm{x}^2
        \right)^{\frac{1}{2}(m - \rho|\mu| - \rho)}
        \symnorm{F}^{\order{m}, \type{\rho}}_{\halbnorm{q}, \mu + e_i}.
    \end{align*}
    We can again estimate the first factor by the same techniques as
    in the Lemma~\ref{lemma:TranslationsActOnSymbols}: we get a
    constant $c$ (depending on $m$, $\rho$, and $\mu$ but not on $ty$
    or $x$) such that we can continue our estimate and get
    \begin{align*}
        &\symnorm{\tau_y^* F - F}^{\order{m}, \type{\rho}
        }_{\halbnorm{q}, \mu} \\
        &\quad\le
        c \norm{y}
        \sup_{\substack{x \in \mathbb{R}^n \\
            t \in [0, 1] \\
            i = 1, \ldots, n}
        }
        \left(
            1 + \norm{ty}^2
        \right)^{\frac{1}{2}\left|m - \rho|\mu|\right|}
        \left(
            1 + \norm{x}^2
        \right)^{-\frac{1}{2}(m - \rho|\mu|) + \frac{1}{2}(m - \rho|\mu| - \rho)}
        \symnorm{F}^{\order{m}, \type{\rho}}_{\halbnorm{q}, \mu + e_i}
        \\
        &\quad\le
        c \norm{y}
        \left(
            1 + \norm{y}^2
        \right)^{\frac{1}{2}\left|m - \rho|\mu|\right|}
        \sup_{x \in \mathbb{R}^n}
        \left(
            1 + \norm{x}^2
        \right)^{-\frac{\rho}{2}}
        \max_{i = 1, \ldots, n}
        \symnorm{F}^{\order{m}, \type{\rho}}_{\halbnorm{q}, \mu + e_i}.
    \end{align*}
    Now if $\rho \ge 0$ then the supremum over all $x \in
    \mathbb{R}^n$ exists and hence we get an estimate of the form
    \[
    \symnorm{\tau_y^* F - F}^{\order{m}, \type{\rho}
    }_{\halbnorm{q}, \mu}
    \le
    c' \norm{y}
    \left(
        1 + \norm{y}^2
    \right)^{\frac{1}{2}\left|m - \rho|\mu|\right|},
    \]
    from which the continuity at $y = 0$ follows.
\end{proof}
\begin{lemma}
    \label{lemma:TauySymbolDiffAtyNull}%
    Let $F \in \Symbol^{\order{m}, \type{\rho}}(\mathbb{R}^n, V)$ be
    given with $\type{\rho} \ge 0$. Then we have
    \begin{equation}
        \label{eq:DifferentiateSymbolsByDiffQuotient}
        \lim_{\epsilon \longrightarrow 0}
        \frac{\tau_{\epsilon e_i}^* F - F}{\epsilon}
        =
        \frac{\partial F}{\partial x^i}
    \end{equation}
    in the $\Symbol^{\order{m}, \type{\rho}}$-topology for all $i = 1,
    \ldots, n$.
\end{lemma}
\begin{proof}
    We proceed analogously to the continuity statement. Let again
    $\halbnorm{q} \in \mathcal{Q}$ and set $m =
    \order{m}(\halbnorm{q})$ and $\rho = \type{\rho}(\halbnorm{q})$
    for abbreviation. First we note that for all $x \in \mathbb{R}^n$
    and $\mu \in \mathbb{N}_0^n$ we have by repeated use of the mean
    value theorem
    \begin{align*}
        &\halbnorm{q}\left(
            \frac{1}{\epsilon}\left(
                \frac{\partial^{|\mu|} (\tau_{\epsilon e_i}^* F)}
                {\partial x^\mu}(x)
                -
                \frac{\partial^{|\mu|} F}{\partial x^\mu}(x)
            \right)
            -
            \frac{\partial}{\partial x^i}
            \frac{\partial^{|\mu|} F}{\partial x^\mu}(x)
        \right) \\
        &\qquad=
        \halbnorm{q}\left(
            \frac{1}{\epsilon}\left(
                \frac{\partial^{|\mu|} F}
                {\partial x^\mu}(x + \epsilon e_i)
                -
                \frac{\partial^{|\mu|} F}{\partial x^\mu}(x)
            \right)
            -
            \frac{\partial}{\partial x^i}
            \frac{\partial^{|\mu|} F}{\partial x^\mu}(x)
        \right) \\
        &\qquad=
        \halbnorm{q}\left(
            \int_0^1
            \frac{\partial}{\partial x^i}
            \frac{\partial^{|\mu|} F}{\partial x^\mu}
            (x + t \epsilon e_i)
            \D t
            -
            \frac{\partial}{\partial x^i}
            \frac{\partial^{|\mu|} F}{\partial x^\mu}(x)
        \right) \\
        &\qquad=
        \halbnorm{q}\left(
            \int_0^1
            \int_0^1
            \frac{\partial}{\partial x^i}
            \frac{\partial}{\partial x^i}
            \frac{\partial^{|\mu|} F}{\partial x^\mu}
            (x + ts\epsilon e_i) t \epsilon \D s \D t
        \right) \\
        &\qquad\le
        \epsilon
        \sup_{s \in [0, 1]}
        \halbnorm{q}\left(
            \frac{\partial^2}{\partial (x^i)^2}
            \frac{\partial^{|\mu|} F}{\partial x^\mu}
            (x + s\epsilon e_i)
        \right) \\
        &\qquad\le
        \epsilon
        \sup_{s \in [0, 1]}
        \left(
            1 + \norm{x + s\epsilon e_i}^2
        \right)^{\frac{1}{2}(m - \rho|\mu| - 2\rho)}
        \symnorm{F}^{\order{m}, \type{\rho}}_{\halbnorm{q}, \mu + 2e_i},
    \end{align*}
    since by assumption $F \in \Symbol^{\order{m},
      \type{\rho}}(\mathbb{R}^n, V)$. Thus we get
    \begin{align*}
        &\symnorm{
          \frac{1}{\epsilon}\left(\tau_{\epsilon e_i}^* F - F\right)
          -
          \frac{\partial F}{\partial x^i}
        }^{\order{m}, \type{\rho}}_{\halbnorm{q}, \mu} \\
        &\qquad=
        \sup_{x \in \mathbb{R}^n}
        \left(
            1 + \norm{x}^2
        \right)^{-\frac{1}{2}(m - \rho|\mu|)}
        \halbnorm{q}\left(
            \frac{1}{\epsilon}\left(
                \frac{\partial^{|\mu|} (\tau_{\epsilon e_i}^* F)}
                {\partial x^\mu}(x)
                -
                \frac{\partial^{|\mu|}F}{\partial x^\mu}(x)
            \right)
            -
            \frac{\partial^{|\mu|}}{\partial x^\mu}
            \frac{\partial F}{\partial x^i}(x)
        \right) \\
        &\qquad\le
        \epsilon \sup_{\substack{x \in \mathbb{R}^n \\ s \in [0, 1]}}
        \left(
            1 + \norm{x}^2
        \right)^{-\frac{1}{2}(m - \rho|\mu|)}
        \left(
            1 + \norm{x + s\epsilon e_i}^2
        \right)^{\frac{1}{2}(m - \rho|\mu| - 2\rho)}
        \symnorm{F}^{\order{m}, \type{\rho}}_{\halbnorm{q}, \mu + 2e_i}
        \\
        &\qquad=
        \epsilon \sup_{\substack{x \in \mathbb{R}^n \\ s \in [0, 1]}}
        \left(
            1 + \norm{x - s\epsilon e_i}^2
        \right)^{-\frac{1}{2}(m - \rho|\mu|)}
        \left(
            1 + \norm{x}^2
        \right)^{\frac{1}{2}(m - \rho|\mu| - 2\rho)}
        \symnorm{F}^{\order{m}, \type{\rho}}_{\halbnorm{q}, \mu + 2e_i}
        \\
        &\qquad\le
        c \epsilon \sup_{\substack{x \in \mathbb{R}^n \\ s \in [0, 1]}}
        \left(
            1 + \norm{s\epsilon e_i}^2
        \right)^{\frac{1}{2}\left|m - \rho|\mu|\right|}
        \left(
            1 + \norm{x}^2
        \right)^{
          -\frac{1}{2}(m - \rho|\mu|)
          + \frac{1}{2}(m - \rho|\mu| - 2\rho)
        }
        \symnorm{F}^{\order{m}, \type{\rho}}_{\halbnorm{q}, \mu + 2e_i}
        \\
        &\qquad=
        c \epsilon
        (1 + \epsilon^2)^{\frac{1}{2}\left|m - \rho|\mu|\right|}
        \sup_{x \in \mathbb{R}^n}
        \left(
            1 + \norm{x}^2
        \right)^{- \rho}
        \symnorm{F}^{\order{m}, \type{\rho}}_{\halbnorm{q}, \mu + 2e_i}.
    \end{align*}
    Using again $\rho \ge 0$ shows that the remaining supremum is
    finite. Thanks to the pre-factor $\epsilon$ we get the desired
    limit \eqref{eq:DifferentiateSymbolsByDiffQuotient}.
\end{proof}

These two lemmas are now enough to conclude the following smoothness
statement of the action of the translations:
\begin{proposition}
    \label{proposition:TranslationActSmoothOnSymbols}%
    Let $V$ be a sequentially complete locally convex space and let
    $\mathcal{Q}$ be a defining system of seminorms for $V$. Let
    $\order{m}$ and $\type{\rho}$ be an order and a type for
    $\mathcal{Q}$ and assume $\type{\rho} \ge 0$. Then the action
    $\tau$ of $\mathbb{R}^n$ on $\Symbol^{\order{m},
      \type{\rho}}(\mathbb{R}^n, V)$ by translations is smooth, i.e.\
    for every $F \in \Symbol^{\order{m}, \type{\rho}}(\mathbb{R}^n,
    V)$ the map $\tau(F)\colon y \mapsto \tau_y^* F$ is a smooth
    map. The derivatives are explicitly given by
    \begin{equation}
        \label{eq:DerivativesTaufSymbol}
        \frac{\partial^{|\mu|}}{\partial y^\mu} \tau_y^*F
        =
        \tau_y^*
        \,
        \frac{\partial^{|\mu|} F}{\partial x^\mu}.
    \end{equation}
\end{proposition}
\begin{proof}
    This is a general argument about group actions of Lie groups: We
    know already that $\tau(F)$ is continuous at $y = 0$ by
    Lemma~\ref{lemma:tauySymbolContinuous}. Moreover, every map
    $\tau_y^*$ is continuous by
    Lemma~\ref{lemma:TranslationsActOnSymbols}. Thus we have by the
    group action property
    \[
    \lim_{y \longrightarrow y'}
    \tau_y^* F
    =
    \lim_{y \longrightarrow 0}
    \tau_{y' + y}^* F
    =
    \lim_{y \longrightarrow 0}
    \tau_y^* \tau_{y'}^*F
    =
    \tau_{y'}^* F
    \]
    in the $\Symbol^{\order{m}, \type{\rho}}$-topology since we have
    continuity at zero. This shows continuity everywhere. Moreover, by
    the same argument
    \[
    \lim_{\epsilon \longrightarrow 0}
    \frac{\tau_{y + \epsilon e_i}^* F- \tau_y^* F}{\epsilon}
    =
    \lim_{\epsilon \longrightarrow 0}
    \tau_y^*
    \,
    \frac{\tau_{\epsilon e_i}^* F - F}{\epsilon}
    =
    \tau_y^*
    \lim_{\epsilon \longrightarrow 0}
    \frac{\tau_{\epsilon e_i}^* F - F}{\epsilon}
    =
    \tau_y^*
    \,
    \frac{\partial F}{\partial x^i},
    \]
    using Lemma~\ref{lemma:TauySymbolDiffAtyNull} and the continuity
    of $\tau_y^*$. This shows that $\tau(F)$ has first partial
    derivatives everywhere given as in
    \eqref{eq:DerivativesTaufSymbol}. Now $\frac{\partial F}{\partial
      x^i} \in \Symbol^{\order{m} - \type{\rho},
      \type{\rho}}(\mathbb{R}^n, V) \subseteq \Symbol^{\order{m},
      \type{\rho}}(\mathbb{R}^n, V)$ thanks to $\type{\rho} \ge 0$ and
    Proposition~\ref{proposition:SymbolFirstProperties},
    \refitem{item:SymbolsInSymbols} as well as
    Proposition~\ref{proposition:PartialDerivativesOnSymbols}. Thus
    the partial derivatives $\frac{\partial}{\partial y^i} \tau(F) =
    \tau(\frac{\partial F}{\partial x^i})$ are again of the form as we
    started with. Hence they are continuous and thus $\tau(F)$ is
    $\Fun[1]$. This
    allows to iterate the above argument finishing the proof.
\end{proof}

As a first application of the affine invariance of the spaces
$\Symbol^{\order{m}, \type{\rho}}(\mathbb{R}^n, V)$ we get the
following generalization of the approximation from
Proposition~\ref{proposition:ApproximateSymbols},
\refitem{item:LimChiEpsilonVectorValued}.
\begin{corollary}
    \label{corollary:ApproximateSymbolsSomehow}%
    Let $\chi \in \Cinfty_0(\mathbb{R}^n)$ satisfy
    $\chi\at{\Ball_r(0)} = 1$ for some $r > 0$. Consider $\tau_y^*
    \chi_\epsilon$ for $\epsilon > 0$ and $y \in \mathbb{R}^n$ where
    as usual $\chi_\epsilon(x) = \chi(\epsilon x)$. Then for every $F
    \in \Symbol^{\order{m}, \type{\rho}}(\mathbb{R}^n, V)$ we have
    \begin{equation}
        \label{eq:LimChiEpsilonTauyfIsf}
        \lim_{\epsilon \longrightarrow 0}
        (\tau_y^* \chi_\epsilon) F = F
    \end{equation}
    in the $\Symbol^{\order{m}', \type{\rho}'}(\mathbb{R}^n,
    V)$-topology provided $\type{\rho}' \le \min(1, \rho)$ and
    $\order{m}' > \order{m}$.
\end{corollary}
\begin{proof}
    We have $(\tau_y^* \chi_\epsilon)F = \tau_y^* (\chi_\epsilon
    \tau_{-y} F)$ and then the continuity of $\tau_y^*$ according to
    Proposition~\ref{proposition:AffineInvarianceOfSymbols} allows to
    exchange $\tau_y^*$ with the limit. Then the result follows from
    Proposition~\ref{proposition:ApproximateSymbols},
    \refitem{item:LimChiEpsilonVectorValued}.
\end{proof}

The smoothness of the translations also allows to consider symbols
taking values in other symbol spaces, as we shall do now. Recall that
given an order $\bm$ and a type $\brho$ for a defining system of
seminorms $\mathcal Q$ on $V$, the symbol space
$\Symbol^{\bm,\brho}(\Rl^n,V)$ is a sequentially complete locally
convex space (Proposition~\ref{proposition:SymbolFirstProperties},
\refitem{item:SymbolsComplete}), which can therefore be used as a
target space instead of $V$. To define symbols taking values in it, we
first have to specify an order $\hat{\bm}$ and a type $\hat{\brho}$ on
the seminorms $\|\cdot\|^{\bm(\qn),\brho(\qn)}_{\qn,\mu}$ generating
the topology of $\Symbol^{\bm,\brho}(\Rl^n,V)$. For $\qn\in\mathcal Q$
and $\mu\in\Nl_0^{n}$, we put
\begin{align}\label{order-and-type-for-S}
    \hat{\bm}(\|\cdot\|^{\bm(\qn),\brho(\qn)}_{\qn,\mu})
    &:= \max\{0,\,\bm(\qn)\}
    ,\qquad
    \hat{\brho}(\|\cdot\|^{\bm(\qn),\brho(\qn)}_{\qn,\mu}) := \brho(\qn)
    \,.
\end{align}
\begin{proposition}\label{proposition:symbol-valued-symbols}
    Let $V$ be a sequentially complete locally convex space with
    defining system of seminorms $\mathcal Q$, and let $\bm$ be an
    order and $\brho\geq0$ a positive type for $\mathcal Q$. Moreover,
    let $F \in \Symbol^{\bm, \brho}(\Rl^{n_1} \oplus \Rl^{n_2}, V)$ be
    given.  Then
    \begin{align}\label{def:F1}
	F_1 &\colon
        \Rl^{n_1} \to \Symbol^{\bm,\brho}(\Rl^{n_2},V)\,,\qquad
        F_1(x_1)\colon x_2\mapsto F(x_1,x_2)
    \end{align}
    is a symbol in $\Symbol^{\hat{\bm},\hat{\brho}}
    (\Rl^{n_1},\,\Symbol^{\bm,\brho}(\Rl^{n_2},V))$ of order
    $\hat{\bm}$ and type $\hat{\brho}$ \eqref{order-and-type-for-S},
    and the map
    \begin{align}
        \Symbol^{\bm,\brho}(\Rl^{n_1}\oplus\Rl^{n_2},V)
        &\longrightarrow
        \Symbol^{\hat{\bm},\hat{\brho}} 
        (\Rl^{n_1},\,\Symbol^{\bm,\brho}(\Rl^{n_2}, V))\\
        F &\longmapsto F_1
    \end{align}
    is linear and continuous. Explicitly, one has the bound
    \begin{align}\label{bnd-double-symbol}
        \|F_1\|^{\hat{\bm},\hat{\brho}}_{\|\cdot\|^{\bm,\brho}_{\qn,\mu},\nu}
        \leq
        \|F\|^{\bm,\brho}_{\qn,\nu\oplus\mu}
        \,
    \end{align}
    for $\qn\in\mathcal Q$, $\nu\in\Nl_0^{n_1},\,\mu\in\Nl_0^{n_2}$.
    Completely analogous statements hold for the map
    \begin{align}
        F_2 &\colon \Rl^{n_2} \to \Symbol^{\bm,\brho}(\Rl^{n_1},V)\,,\qquad
        F_2(x_2)\colon x_1\mapsto F(x_1,x_2)
        \label{def:F2}
        \,.
    \end{align}
\end{proposition}
\begin{proof}
    In terms of the embedding
    $\iota_2\colon\Rl^{n_2}\to\Rl^{n_1}\oplus\Rl^{n_2}$,
    $\iota_2(x_2):=(0,x_2)$, and the previously discussed translations
    $\tau$, the map $F_1$ reads
    $F_1(x_1)=\iota_2^*(\tau_{x_1\oplus\,0}^*F)$. But according to
    Proposition~\ref{proposition:TranslationActSmoothOnSymbols},
    $x_1\mapsto\tau_{x_1\oplus\,0}^*F$ is a smooth map from
    $\Rl^{n_1}$ to $\Symbol^{\bm,\brho}(\Rl^{n_1}\oplus\Rl^{n_2},V)$,
    and according to Lemma~\ref{lemma:RestrictionOfSymbols},
    $\iota_2^*\colon
    \Symbol^{\bm,\brho}(\Rl^{n_1}\oplus\Rl^{n_2},V)\to\Symbol^{\bm,\brho}
    (\Rl^{n_2},V)$ is linear and continuous. Hence
    $F_1\colon\Rl^{n_1}\to\Symbol^{\bm,\brho } (\Rl^{n_2},V)$ is
    smooth.  Since $F\mapsto F_1$ is clearly linear, it only remains
    to verify the estimate \eqref{bnd-double-symbol}. To this end, let
    $\qn\in\mathcal Q$, $\nu\in\Nl_0^{n_1}$, $\mu\in\Nl_0^{n_2}$, and
    put $\hat{\qn}:=\|\cdot\|^{\bm,\brho}_{\qn,\mu}$ for short. The
    seminorm in question is
    \begin{align*}
        \|F_1\|^{\hat{\bm},\hat{\brho}}_{\hat{\qn},\nu}
        &=
        \sup_{x_1\in\Rl^{n_1}}
        \frac{\|\partial_{x_1}^\nu
          F_1(x_1)\|^{\bm,\brho}_{\qn,\mu}}
        {(1+\|x_1\|^2)^{\frac{1}{2}(\hat{\bm}(\hat{\qn
            })-\hat{\brho}(\hat{\qn})|\nu|)}}
        \\
        &=
        \sup_{x_1\in\Rl^{n_1}}
        \frac{1}
        {(1+\|x_1\|^2)^{\frac{1}{2}(\hat{\bm}(\hat{\qn})-\hat{\brho}(\hat{\qn})|\nu|)}}
        \sup_{x_2\in\Rl^{n_2}}
        \frac{
          \qn(\partial_{x_2}^\mu \partial_{x_1}^\nu
          F_1(x_1)(x_2))
        }
        {
          (1+\|x_2\|^2)^{\frac{1}{2}(\bm(\qn)-\brho(\qn)|\mu|)}
        }
        \,.
    \end{align*}
    Note that by definition of $\hat{\bm}$, $\hat{\brho}$, the powers
    $k:=-\frac{1}{2}(\hat{\bm}(\hat{\qn})-\brho(\qn)|\nu|)$ and
    $k':=-\frac{1}{2}(\bm(\qn)-\brho(\qn)|\mu|)$ satisfy
    $\max\{k,k',k+k'\}\leq K$ with
    $K:=-\frac{1}{2}(\bm(\qn)-\brho(\qn)|\nu\oplus\mu|)$ for all
    $\mu,\nu$.  Hence we can use the inequality
    \eqref{eq:kkkInequality} to estimate
    \begin{align*}
	\|F_1\|^{\hat{\bm},\hat{\brho}}_{\hat{\qn},\nu}
        &\leq
        \sup_{x_1\in\Rl^{n_1} \atop x_2\in\Rl^{n_2}}
        \frac{
          \qn(\partial^{\nu\oplus\mu}
          F(x_1,x_2))
        }
        {
          (1+\|x_1\|^2+\|x_2\|^2)^{\frac{1}{2}(\bm(\qn)-\brho(\qn)|\nu\oplus\mu|)}
        }
        =
        \|F\|^{\bm,\brho}_{\qn,\nu\oplus\mu}
        \,.
    \end{align*}
    This establishes \eqref{bnd-double-symbol} and thus in particular
    the continuity of $F\mapsto F_1$. The arguments for $F_2$ are
    completely analogous.
\end{proof}

\section{Oscillatory integrals for vector-valued symbols}
\label{section:integrals}
\subsection{Construction of the integral map}

We now come to the definition of oscillatory integrals of
symbols. Again, we proceed in close analogy to the scalar case, see
\cite[Sect.~7.8]{Hormander:1990} as well as \cite{Hormander:1971}. The
essential idea is to use the Riemann integral for compactly supported
smooth functions and show that it enjoys a remarkable continuity
property with respect to the symbol topologies. We are here not
interested in the most general case, where oscillatory integrals are
used to define maps from test function spaces to distributions, as
discussed in \cite[Sect.~7.8]{Hormander:1990}. Instead we are just
interested in the values of the oscillatory integrals per se. To this
end, we endow $\mathbb{R}^n$ with a non-degenerate bilinear form
$\SP{\argument, \argument}$.  Then we consider for $F \in
\Stetig_0(\mathbb{R}^n, V)$ the integral with an oscillatory phase
\begin{equation}
    \label{eq:Inull}
    I_0 (F)
    :=
    \frac{1}{(2\pi)^n} \int_{\mathbb{R}^{2n}}
    dp\,dx\,
    e^{i \SP{p, x}} F(x, p)\,,
\end{equation}
which is a well-defined Riemann integral thanks to the continuity of
the integrand and the compact support of $F$. The integral defines a
linear map
\begin{equation}
    \label{eq:InullMap}
    I_0\colon \Stetig_0(\mathbb{R}^{2n}, V) \longrightarrow V,
\end{equation}
which is continuous in the $\Stetig_0$-topology. Since the
$\Stetig_0$-topology is coarser than every $\Fun_0$-topology for $k
\in \mathbb{N}_0 \cup \{+\infty\}$, we see that for all $k$ we have a
continuous map
\begin{equation}
    \label{eq:InullCkNullContinuous}
    I_0\colon \Fun_0(\mathbb{R}^{2n}, V) \longrightarrow V.
\end{equation}
Up to now, we have not used any particular properties of the phase
function besides its continuity. However, it turns out that the
continuity with respect to the $\Fun_0$-topologies is not the right
one to extend $I_0$ to the symbol spaces.

Instead we have to show the continuity of $I_0$ with respect to some
appropriate $\Symbol^{\order{m}, \type{\rho}}$-topology, and this step
will make use of more specific properties of the phase function.  We
begin with the following preparations. Consider the polynomial
\begin{equation}
    \label{eq:PolynomialP}
    P(x) := (i+ x_1) \cdots (i + x_n)
\end{equation}
on $\mathbb{R}^n$ which is clearly of degree $n$ and hence a
scalar symbol $P \in \Symbol^{n, 1}(\mathbb{R}^n, \mathbb{C})$.

Since each factor $(i + x_k)$ is non-vanishing, we can define
arbitrary powers $(i + x_k)^s$ for $s\in\mathbb{Z}$, which are symbols
of order $s$. For $s\geq0$, this follows from
Lemma~\ref{lemma:ScalarSymbolPower}, and for $s<0$ by explicit
differentiation. Note that given a symbol $F \in \Symbol^{\order{m},
  \type{\rho}}(\mathbb{R}^{2n}, V)$ of some order $\bm$ and type
$\brho\leq1$, the function
\begin{equation}
    \label{eq:PsPsFIsSymbol}
    \mathbb{R}^{2n} \ni (x, p)
    \; \mapsto \;
    P^s(x) P^s(p) F(x, p) \in V
\end{equation}
is a symbol of order $\order{m} + 2sn$ and type $\type{\rho}$. This
follows directly by application of
Corollary~\ref{corollary:SymbolModuleStuff},
\refitem{item:SymbolModule}, since $(x,p)\mapsto P^s(x)P^s(p)$ is of
order $2sn$ and type $1$, and $\type{\rho} \le 1$ by assumption.

We also note the well-known fact that given $s\in\Nl_0$, there exists
a differential operator
\begin{equation}
    \label{eq:Qs}
    Q_s = \sum_{|\mu|, |\nu| \le s} a^{\mu\nu}_s
    \frac{\partial^{|\mu|}}{\partial x^\mu}
    \frac{\partial^{|\nu|}}{\partial p^\nu}
\end{equation}
with constant coefficients $a^{\mu\nu}_s \in \mathbb{C}$ such that
\begin{equation}
    \label{eq:QsExpPsPsExp}
    Q_s e^{i \SP{p, x}} = P^s(x) P^s(p) e^{i\SP{p, x}}\,.
\end{equation}

After these preparatory remarks, we now derive the crucial estimate of
the integral $I_0$ with respect to the symbol topologies. The proof is
based on the usual technique of converting differentiability
properties of the integrand to damping factors by integration by parts
against $e^{i\SP{p,x}}$. As shown below, for this technique to work we
only have to make a restriction on the type, but not on the order.
\begin{lemma}
    \label{lemma:InullSymbolContinuity}%
    Let $\mathcal Q$ be a defining system of seminorms for $V$, with
    order $\order{m}$ and type $\type{\rho}$ such that $-1 <
    \type{\rho} \le 1$.  Then for every $\halbnorm{q} \in \mathcal{Q}$
    there exists a constant $c > 0$ and $N \in \mathbb{N}_0$ such that
    for all $F \in \Cinfty_0(\mathbb{R}^{2n}, V)$ we have
    \begin{equation}
        \label{eq:SymbolEstimateInull}
        \halbnorm{q}\left(I_0(F)\right)
        \le
        c
        \sum_{|\mu| \le N}
        \symnorm{F}^{\order{m}, \type{\rho}}_{\halbnorm{q}, \mu}.
    \end{equation}
\end{lemma}
\begin{proof}
    Let $F \in \Cinfty_0(\mathbb{R}^{2n}, V)$ have compact support in
    a compact interval $K \subseteq \mathbb{R}^{2n}$. Then we compute
    using \eqref{eq:QsExpPsPsExp}
    \begin{align*}
        I_0(F)
        &=
        \frac{1}{(2\pi)^n} \int_{K}
        e^{i \SP{p, x}} F(x ,p) \D^n x \D^n p \\
        &=
        \frac{1}{(2\pi)^n} \int_{K}
        \frac{1}{P^s(x) P^s(p)} \left(Q_s e^{i \SP{p, x}}\right)
        F(x, p) \D^nx \D^np \\
        &=
        \frac{1}{(2\pi)^n} \int_{K}
        e^{i \SP{p, x}}
        Q_s^\Trans
        \frac{F(x, p)}{P^s(x) P^s(p)}
        \D^nx \D^np,
        \tag{$*$}
    \end{align*}
    where $Q_s^\Trans= \sum_{0 \le |\mu|, |\nu| \le s} (-1)^{|\mu| +
      |\nu|} a_s^{\mu\nu} \frac{\partial^{|\mu|}}{\partial x^\mu}
    \frac{\partial^{|\nu|}}{\partial p^\nu}$ denotes the transposed
    differential operator and $s \in \mathbb{N}_0$ is
    arbitrary. Indeed, the integration by parts is possible since $F$
    has compact support inside the interval $K$. Since ($*$) is valid
    for all $s \in \mathbb{N}_0$, the idea is to use a large enough
    $s$ which produces under the integral an integrable symbol on the
    right hand side, independent of $K$. Since $F$ has compact support
    it is a symbol for any order and any type. Thus also the function
    $(x, p) \mapsto \frac{F(x, p)}{P^s(x) P^s(p)}$ is a symbol, say of
    order $\order{m} - 2sn$ and type $\type{\rho}$. Thus for all $\mu,
    \nu \in \mathbb{N}_0^n$ we have the estimate
    \[
    \halbnorm{q}\left(
        \frac{\partial^{|\mu|}}{\partial x^\mu}
        \frac{\partial^{|\nu|}}{\partial p^\nu}
        \frac{F(x, p)}{P^s(x) P^s(p)}
    \right)
    \le
    \left(
        1 + \norm{(x, p)}^2
    \right)^{\frac{1}{2}\left(
          \order{m}(\halbnorm{q}) - 2sn -
          \type{\rho}(\halbnorm{q})|\mu \oplus \nu|
      \right)
    }
    \symnorm{\frac{F(\argument, \argument)}{P^s(\argument)
        P^s(\argument)}}^{\order{m} - 2sn, \type{\rho}}_{\halbnorm{q},
      \mu \oplus \nu}
    \]
    for all $s \in \mathbb{N}_0$. We know that $|\mu \oplus \nu| =
    |\mu| + |\nu| \le 2sn$ as the operator $Q_s$ is of order $2sn$
    only. Hence the condition $\type{\rho}(\halbnorm{q}) > -1$ shows
    that there is a $s \in \mathbb{N}_0$ such that for all $|\mu|,
    |\nu| \le sn$ we have
    \begin{align}\label{eq:BoundRhoM}
        \order{m}(\halbnorm{q})
        - 2sn
        - \type{\rho}(\halbnorm{q})|\mu \oplus \nu|
        <
        - 2 (n+1).
    \end{align}
    In fact, we get the left hand side as negative as we want by
    taking large enough $s$. Finally, by
    Proposition~\ref{proposition:SymbolProducts},
    \refitem{item:ProductInOneVariable}, we get the estimate
    \[
    \symnorm{\frac{F}{P^s(\argument)
        P^s(\argument)}}^{\order{m} - 2sn, \type{\rho}}_{\halbnorm{q},
      \mu \oplus \nu}
    \le
    2^{|\mu| + |\nu|}
    \max_{\mu' \oplus \nu' \le \mu \oplus \nu}
    \symnorm{\frac{1}{P^s(\argument)P^s(\argument)}}^{-2sn, 1}_{\mu'
      \oplus \nu'}
    \max_{\mu'' \oplus \nu'' \le \mu \oplus \nu}
    \symnorm{F}^{\order{m}, \type{\rho}}_{\halbnorm{q}, \mu'' \oplus
      \nu''},
    \]
    since we have $\type{\rho} \le 1$ and $\frac{1}{P^s(\argument)
      P^s(\argument)}$ is a symbol of order $-2sn$ and type $1$
    according to Lemma~\ref{lemma:ScalarSymbolPower}. Taking now $s$
    large enough so that \eqref{eq:BoundRhoM} is satisfied we get the
    estimate
    \begin{align*}
        \halbnorm{q}\left(I_0(F)\right)
        &=
        \halbnorm{q}\left(
            \frac{1}{(2\pi)^n} \int_K
            e^{i \SP{p, x}}
            Q_s^\Trans
            \frac{F(x, p)}{P^s(x)P^s(p)}
            \D^nx \D^np
        \right) \\
        &\le
        \frac{1}{(2\pi)^n}
        \int_K
        \halbnorm{q}\left(
            \sum_{0 \le \mu, \nu \le s}
            a^{\mu\nu}_s (-1)^{|\mu| + |\nu|}
            \frac{\partial^{|\mu|}}{\partial x^\mu}
            \frac{\partial^{|\nu|}}{\partial p^\nu}
            \frac{F(x, p)}{P^s(x)P^s(p)}
        \right)
        \D^nx\D^np \\
        &\le
        \frac{1}{(2\pi)^n}
        \sum_{0 \le \mu, \nu \le s}
        |a^{\mu\nu}_s|
        \int_K
        \left(1 + \norm{(x, p)}^2\right)^{-(n+1)}
        \D^nx \D^np
        \symnorm{
          \frac{F}{P^s(\argument)P^s(\argument)}
        }^{\order{m} - 2sn, \type{\rho}
        }_{\halbnorm{q}, \mu \oplus \nu}
        \\
        &\le
        c \sum_{0 \le |\mu|, |\nu| \le s}
        \symnorm{F}^{\order{m}, \type{\rho}
        }_{\halbnorm{q}, \mu\oplus\nu},
    \end{align*}
    with the constant
    \[
    c =
    \frac{2^{2sn}}{(2\pi)^n}
    \int_{\mathbb{R}^{2n}}
    \frac{\D^nx \D^np}{\left(1 + \norm{(x, p)}^2\right)^{-(n+1)}}
    \max_{0 \le |\mu|, |\nu| \le s} |a^{\mu\nu}_s|
    \max_{0 \le |\mu'|, |\nu'| \le s}
    \symnorm{
      \frac{1}{P^s(\argument)P^s(\argument)}
    }^{-2sn, 1}_{\mu'\oplus\nu'}
    < \infty.
    \]
    Note that the integral is finite indeed as we were able to make
    the exponent \eqref{eq:BoundRhoM} negative enough such that the
    dependence on the compact interval $K$ disappears.
\end{proof}

We now define oscillatory integrals for symbols $F\in
\Symbol^{\bm,\brho}(\Rl^{2n},V)$ of non-compact support by extending
the integral $I_0$ defined on $\Cinfty_0(\Rl^{2n},V)$. Doing so, we
will rely in an essential manner on the preceding lemma and
Proposition~\ref{proposition:ApproximateSymbols},
\refitem{item:CinftyDenseInSymbols}, and therefore restrict to types
$\brho$ with $-1<\brho\leq 1$. The order $\bm$ will be arbitrary.

To describe the extension procedure, we consider in addition to $\bm$ and
$-1<\brho\leq1$ an auxiliary order $\bm'>\bm$ and  type
$-1<\brho'\leq\brho$ for $\mathcal Q$, and the corresponding inclusions
\begin{align*}
    I_0\colon \Cinfty_0(\Rl^{2n},V) \subset
    \Symbol^{\bm,\brho}(\Rl^{2n},V) \subset
    \Symbol^{\bm',\brho'}(\Rl^{2n},V) \longrightarrow V \,.
\end{align*}
In general, $\Cinfty_0(\Rl^{2n},V) \subset
\Symbol^{\bm,\brho}(\Rl^{2n},V)$ is not (sequentially) dense in the
$\Symbol^{\bm,\brho}$-topology. But according to
Proposition~\ref{proposition:ApproximateSymbols},
\refitem{item:CinftyDenseInSymbols}, the sequential closure of
$\Cinfty_0(\Rl^{2n},V)$ in the weaker $\Symbol^{\bm',\brho'}$-topology
contains $\Symbol^{\bm,\brho}(\Rl^{2n},V)$. Moreover, according to the
bound \eqref{eq:SymbolEstimateInull},
$I_0\colon\Cinfty_0(\Rl^{2n},V)\longrightarrow V$ is a continuous
linear map in the $\Symbol^{\bm',\brho'}$-topology. We can thus extend
$I_0$ to a continuous linear map from the sequential completion of
$\Cinfty_0(\Rl^{2n},V)$ in the $\Symbol^{\bm',\brho'}$-topology to
$V$. The restriction of this extension to
$\Symbol^{\bm,\brho}(\Rl^{2n},V)$ is our definition of oscillatory
integral on $\Symbol^{\bm,\brho}(\Rl^{2n},V)$; it is denoted by
\begin{align}\label{eq:DefImrho}
    I^{\bm,\brho}_{\bm',\brho'}\colon
    \Symbol^{\bm,\brho}(\Rl^{2n},V)
    \longrightarrow V\,.
\end{align}
\begin{theorem}\label{Theorem:OscillatoryIntegral}
    Let $V$ be a sequentially complete locally convex space with
    defining system of seminorms $\mathcal Q$, and $\bm$,
    $-1<\brho\leq1$ an order and a type for $\mathcal Q$.
    \begin{theoremlist}
    \item \label{item:IntegralIndependent} The integrals $I^{\bm,\brho} :=
        I_{\bm',\brho'}^{\bm,\brho}$ \eqref{eq:DefImrho} are independent
        of the order and type $\bm'$, $\brho'$ as long as $\bm'>\bm$ and
        $-1<\brho'\leq\brho$.
    \item \label{item:IntegralContinuuosLinear} $I^{\bm,\brho}\colon
        \Symbol^{\bm,\brho}(\Rl^{2n},V) \longrightarrow V$ is linear
        and continuous.
    \item \label{item:IntegralIsIntegral} For $F\in
        \Cinfty_0(\Rl^{2n},V)$, we have $I^{\bm,\brho}(F)=I_0(F)$.
    \item \label{item:IntegralIsConsistent} For orders $\bm,\bm'$,
        types $-1<\brho,\brho'\leq1$, and $F\in
        \Symbol^{\bm,\brho}(\Rl^{2n},V) \cap
        \Symbol^{\bm',\brho'}(\Rl^{2n},V)$, we have
        \begin{align}\label{I-restricted-to-Sm}
            I^{\bm,\brho}(F) = I^{\bm',\brho'}(F) \,.
        \end{align}
    \end{theoremlist}
\end{theorem}
\begin{proof}
    For the first part, let $\bm',\bm''$ be orders and
    $\brho',\brho''$ types for $\mathcal Q$, with $\bm',\bm''>\bm$ and
    $-1<\brho',\brho''\leq\brho$, and $F\in
    \Symbol^{\bm,\brho}(\Rl^{2n},V)$. We have to show $I
    _{\bm',\brho'}^{\bm,\brho}(F)=
    I_{\bm'',\brho''}^{\bm,\brho}(F)$. By the above construction of
    these maps, there exist sequences $\{F_n'\},\{F_n''\} \subset
    \Cinfty_0$ converging to $F$ in the topology of
    $\Symbol^{\bm',\brho'}(\Rl^{2n},V)$ and
    $\Symbol^{\bm'',\brho''}(\Rl^{2n},V)$, respectively, and
    \begin{align}\label{eq:Immprime}
        I _{\bm',\brho'}^{\bm,\brho}(F)=\lim_{n\to\infty}I_0(F_n')
        \,,\qquad
        I_{\bm'',\brho''}^{\bm,\brho}(F)=\lim_{n\to\infty}I_0(F_n'')
        \,.
    \end{align}
    To show that these limits coincide, let $\bm'''$, $\brho'''$ be an
    order and type with $\bm'''\geq\bm',\bm''$, and
    $-1<\brho'''\leq\brho',\brho''$. Fixing a seminorm $\qn\in\mathcal
    Q$, we can use the bound \eqref{eq:SymbolEstimateInull} and
    \eqref{eq:SymbolNormViaSymbolNorm} to estimate with some constants
    $c>0$, $N\in\Nl_0$,
    \begin{align*}
        \qn(I_0(F_n')-I_0(F_n''))
        &\leq
        c\sum_{|\mu|\leq N}
        \|F_n'-F_n''\|^{\bm''',\brho'''}_{\qn,\mu}
        \\
        &\leq
        c\sum_{|\mu|\leq N} \left(
            \|F_n'-F\|^{\bm''',\brho'''}_{\qn,\mu}
            +
            \|F-F_n''\|^{\bm''',\brho'''}_{\qn,\mu}
        \right)
        \\
        &\leq
        c\sum_{|\mu|\leq N} \left(
            \|F_n'-F\|^{\bm',\brho'}_{\qn,\mu}
            +
            \|F-F_n''\|^{\bm'',\brho''}_{\qn,\mu}
        \right)
        \,.
    \end{align*}
    In view of the approximation properties of the sequences
    $\{F_n'\}$, $\{F_n''\}$, the last expression converges to zero for
    $n\to\infty$, {\em i.e.}, we have $\qn(I_0(F_n')-I_0(F_n''))\to
    0$. Since $\qn $ was arbitary, \eqref{eq:Immprime} now gives $I
    _{\bm',\brho'}^{\bm,\brho}(F)=I_{\bm'',\brho''}^{\bm,\brho}(F)$. From
    now on, we write $I^{\bm,\brho}:=I_{\bm',\brho'}^{\bm,\brho}$ for
    this integral.  For part~\refitem{item:IntegralContinuuosLinear},
    by construction, $I^{\bm,\brho}\colon \Symbol^{\bm,\brho}\to V$ is
    a linear map which is continuous in the
    $\Symbol^{\bm',\brho'}$-topology for $\bm'>\bm$ and
    $-1<\brho'\leq\brho$. But as the topology of
    $\Symbol^{\bm,\brho}(\Rl^{2n},V)$ is stronger than that of
    $\Symbol^{\bm',\brho'}(\Rl^{2n},V)$, this map is continuous in the
    $\Symbol^{\bm,\brho}$-topology as well.  By the very definition of
    $I^{\bm,\brho}$, we have $I^{\bm,\brho}(F)=I_0(F)$ for $F\in
    \Cinfty_0(\Rl^{2n},V)$, {\em i.e.}
    \refitem{item:IntegralIsIntegral} holds. It remains to check
    \refitem{item:IntegralIsConsistent}, and to this end, we consider
    an order $\bm''>\bm,\bm'$ and type
    $-1<\brho''\leq\brho,\brho'$. Then for any $F\in
    \Symbol^{\bm,\brho}(\Rl^{2n},V)\cap
    \Symbol^{\bm',\brho'}(\Rl^{2n},V)$, there exists a sequence
    $\{F_n\}\subset \Cinfty_0(\Rl^{2n},V)$ converging to $F$ in the
    topology of $\Symbol^{\bm'',\brho''}(\Rl^{2n},V)$, and in view of
    \refitem{item:IntegralIndependent}, we have
    \begin{align}
        I^{\bm,\brho}(F)
        =
        I^{\bm,\brho}_{\bm'',\brho''}(F)
        =
        \lim_{n\to\infty}I_0(F_n)
        =
        I^{\bm',\brho'}_{\bm'',\brho''}(F)
        =
        I^{\bm',\brho'}(F)
        \,.
    \end{align}
    This proves \eqref{I-restricted-to-Sm}.
\end{proof}

The compatability \eqref{I-restricted-to-Sm} of the integral maps
$I^{\bm,\brho}$ with the structure of the symbol spaces allows us to
consistently define an oscillatory integral on the space
$\underline{\Symbol}(\Rl^{2n},V)$, see \eqref{eq:IntegrableSymbols},
consisting of all $\Symbol^{\bm,\brho}(\Rl^{2n},V)$, with arbitary
orders $\bm$ and types $-1<\brho\leq 1$.
\begin{definition}
    The oscillatory integral is the linear map $I\colon
    \underline{\Symbol}(\Rl^{2n},V)\to V$ uniquely determined by
    $I|_{\Symbol^{\bm,\brho}(\Rl^{2n},V)}:=I^{\bm,\brho}$,
    $-1<\brho\leq1$. If the target space or the domain of integration
    needs to be emphasized, we write more precisely $I_V$ or
    $I_{\Rl^{2n},V}$ instead of $I$. We also use the symbolic notation
    \begin{align}
        (2\pi)^{-n}\int_{\Rl^{2n}} dp\,dx\,e^{i\SP{p,x}}F(p,x)
        :=
        I(F)\,.
    \end{align}
\end{definition}

Note that according to the discussion in
Section~\ref{subsec:FurtherPropertiesSymbols}, the space
$\IntSymbol(\Rl^{2n},V)$ and the oscillatory integral $I$ do not
depend on a choice of defining system $\mathcal Q$ of seminorms, but
are intrinsically defined.

\subsection{Calculational rules for the oscillatory integral}

We now derive the main properties of the integral $I$. To begin with,
we note how oscillatory integrals can be computed in practice.
\begin{proposition}\label{proposition:integral-explicit}
    \begin{propositionlist}
    \item \label{item:IntegralExplicit} Let $F\in
        \underline{\Symbol}(\Rl^{2n},V)$, $p_0,x_0\in\Rl^n$, and
        $\chi\in \Cinfty_0(\Rl^{2n},\Rl)$ with $\chi(p,x)=1$ for
        $(p,x)$ in some open neighborhood of $(0,0)$. Then the
        oscillatory integral of $F$ is the limit of Riemann integrals
        \begin{align}\label{integral-explicit}
            I(F)
            &=
            (2\pi)^{-n}\lim_{\eps\to0}
            \int_{\mathbb{R}^{2n}}
            dp\, dx \, e^{i\SP{p,x}}\chi(\eps
            (p-p_0),\eps (x-x_0))F(p,x)\,.
        \end{align}
    \item \label{item:IntegralIsRiemann} Let $\mathcal Q$ be a
        defining system of seminorms on $V$, with order $\bm$ and type
        $\brho$ such that there exist constants $C_1,C_2\in\Rl$
        satisfying
        \begin{align}\label{eq:m-rho-bounded}
            \bm(\qn)\leq C_1\,,\qquad 1\geq \brho(\qn)\geq C_2 > -1
        \end{align}
        for all $\qn\in\mathcal Q$. Then there exists $s\in\Nl$, and
        $b_{\mu\nu}\in\Cl$, $\mu,\nu\in\Nl_0^n$, $|\mu|,|\nu|\leq s$,
        such that for all $F\in \Symbol^{\bm,\brho}(\Rl^{2n},V)$, the
        oscillatory integral is given by a convergent Riemann integral
        \begin{align}\label{I0-as-Riemann}
            I(F)
            &=
            \sum_{|\mu|,|\nu|\leq s}
            b_{\mu\nu}
            \int_{\mathbb{R}^{2n}}
            dp\, dx \, e^{i\SP{p,x}}
            \frac{\partial^{|\mu|}}{\partial p^\mu}\frac{\partial^{|\nu|}}{\partial x^\nu}
            \left(
                \frac{F(p,x)}{\prod_{k=1}^n (i+p_k)^s(i+x_k)^s}
            \right)
            .
        \end{align}
    \end{propositionlist}
\end{proposition}
\begin{proof}
    Fixing a defining system of seminorms $\mathcal Q$ on $V$, we
    consider a symbol $F\in \Symbol^{\bm,\brho}(\Rl^{2n},V)$ for some
    order $\bm$ and type $-1<\brho\leq1$ for $\mathcal
    Q$. Furthermore, let $\bm',\brho'$ be an auxiliary order and type
    for $\mathcal Q$ such that $\bm'>\bm$ and $-1<\brho'\leq\brho$. It
    has been shown in
    Corollary~\ref{corollary:ApproximateSymbolsSomehow} that
    $(\chi_\eps F)(p,x):=\chi(\eps (p-p_0),\eps (x-x_0))F(p,x)$
    converges to $F$ in the topology of
    $\Symbol^{\bm',\brho'}(\Rl^{2n},V)$ as $\eps\to 0$. Since
    $\chi_\eps F \in \Cinfty_0(\Rl^{2n},V)$, the formula
    $I(F)=\lim_{\eps \to 0} I_0(\chi_\eps F)$
    \eqref{integral-explicit} holds by definition of $I$ as the
    $\Symbol^{\bm,\brho}$-continuous extension of $I_0$. This proves
    the first part.  The second part is basically a corollary of the
    proof of Lemma~\ref{lemma:InullSymbolContinuity}: One first checks
    that if \eqref{eq:m-rho-bounded} holds, then there exists
    $s\in\Nl_0$ such that the inequality \eqref{eq:BoundRhoM} is valid
    for all $\qn\in\mathcal Q$ for the same value of $s$.  Using a
    cutoff function $\chi$ as in the first part of this proposition,
    we can then apply the arguments in the proof of
    Lemma~\ref{lemma:InullSymbolContinuity} to $\chi_\eps
    F\in\Cinfty_0(\Rl^{2n},V)$ to conclude that there exist
    coefficients $b_{\mu\nu}$ such that
    \begin{align}\label{int1}
        I_0(\chi_\eps F)
        &=
        \sum_{|\mu|,|\nu|\leq s} b_{\mu\nu}\int_{\mathbb{R}^{2n}}
        dp\, dx\,e^{i\SP{p,x}}
        \frac{\partial^{|\mu|}}{\partial p^\mu}\frac{\partial^{|\nu|}}{\partial
          x^\nu} \frac{\chi(\eps p,\eps
          x)F(p,x)}{\prod_{k=1}^n (i+p_k)^s(i+x_k)^s}
        \,.
    \end{align}
    To control the limit $\eps\to0$, we again use the same arguments
    as in Lemma~\ref{lemma:InullSymbolContinuity}: For any seminorm
    $\qn$, we find an integrable upper bound to
    $(p,x)\mapsto\qn\left(\partial_p^{\mu}\partial_x^{\nu}
        P^{-s}(p)P^{-s}(x)F(p,x)\right)$.  This allows us to carry out
    the limit $\eps\to0$ in \eqref{int1}. Namely, applying Leibniz'
    rule, we see that all terms in \eqref{int1} which contain
    derivatives of $\chi$, and hence factors of $\eps$, converge to
    zero as $\eps\to0$ because the derivatives of $F$ and the damping
    factors are bounded in each seminorm $\qn$. Only the term with no
    derivatives on $\chi$ remains, and as this has an integrable upper
    bound, and $\chi(0,0)=1$, we obtain the claimed formula
    \eqref{I0-as-Riemann} for $I(F)=\lim_{\eps\to0}I_0(\chi_\eps F)$.
\end{proof}

If $V$ is a Banach space and $\mathcal Q$ consists of just its norm,
then \eqref{eq:m-rho-bounded} is clearly satisfied for any order $m$,
and any admissible type $-1<\rho\leq1$. So in this case, the
oscillatory integrals can always be reformulated as improper Riemann
integrals. But if $\mathcal Q$ is infinite, and $\bm$ unbounded, this
is no longer the case.  Nonetheless, also in this general situation,
oscillatory integrals exhibit many of the familiar properties of
Riemann integrals. In particular, they are compatible with continuous
linear maps, and the usual rules of substitution and integration by
parts still apply, as we now show in the following lemmas and
propositions.
\begin{lemma}\label{lemma:AinI}
    Let $V,U$ be sequentially complete locally convex spaces, $A\colon
    V\longrightarrow U$ a continuous linear map, and
    $F\in\underline{\Symbol}(\Rl^{2n},V)$. Then $AF\colon (p,x)\mapsto
    AF(p,x)$ is a symbol in $\underline{\Symbol}(\Rl^{2n},U)$, and
    \begin{align}
        A\, I_V(F) = I_U(A F)
        \,.
    \end{align}
\end{lemma}
\begin{proof}
    First we note that the equation holds for compactly supported
    $F$. But then the usual continuity and approximation argument
    shows that the equation also holds for arbitrary $F$.
\end{proof}
\begin{remark}
    If we consider an {\em antilinear} continuous map $C:V\to U$
    instead, the only difference to the above described situation is
    that the oscillating factor $e^{i\SP{p,x}}$ has to be
    conjugated. This conjugation can be compensated by a variable
    substitution $p\to-p$ in the integrals. So in this case, we have,
    $F\in \IntSymbol(\Rl^{2n},V)$,
    \begin{align}\label{antilinear-map-in-I}
        C I_V(F) = I_U(C\, F_-)
        \quad\textrm{with}\quad
        F_-(p,x):=F(-p,x)\,.
    \end{align}
\end{remark}
\begin{lemma}\label{lemma:affine-substitutions}
    Let $q,y\in\Rl^n$, $A\in\GL(n,\Rl)$, and denote by $A^T$ the
    transpose of $A$ with respect to the chosen inner product on
    $\Rl^n$. Then, for any $F\in \IntSymbol(\Rl^{2n},V)$, the
    functions
    \begin{align}
        F_{q,y,A}(p,x) &:= e^{-i\SP{p,y}}F(Ap+q,x)\,,\\
        F^{q,y,A}(p,x) &:= e^{-i\SP{q,x}}F(p,A^Tx+y)\,,
    \end{align}
    are elements of $\IntSymbol(\Rl^{2n},V)$ as well, and
    \begin{align}
        I(F_{q,y,A}) = I(F^{q,y,A})\,.
    \end{align}
\end{lemma}
\begin{proof}
    For $F\in\IntSymbol(\Rl^{2n},V)$, an application of
    Lemma~\ref{lemma:GlnActsOnSymbols} and
    Lemma~\ref{lemma:TranslationsActOnSymbols} shows that the
    functions $F_{q,0,A}$ and $F^{0,y,A}$ without the oscillating
    factors lie in $\IntSymbol(\Rl^{2n},V)$. But $(p,x)\mapsto
    e^{-i\SP{p,y}}$ and $(p,x)\mapsto e^{-i\SP{q,x}}$ are scalar
    symbols of type $0$ and order $0$, as is easily verified by
    differentiation. Hence Corollary~\ref{corollary:SymbolAlgebra}
    yields $F_{q,y,A},F^{q,y,A}\in\IntSymbol(\Rl^{2n},V)$.  To compare
    the oscillatory integrals of these functions, we pick $\chi\in
    \Cinfty_0(\Rl^{2n},\Rl)$ as in
    Proposition~\ref{proposition:integral-explicit}, and compute
    according to \eqref{integral-explicit}
    \begin{align*}
        (2\pi)^n I(F_{q,y,A})
        &=
        \lim_{\eps\to0}\int_{\mathbb{R}^{2n}}
        dp\,dx\,e^{i\SP{p,x}}e^{-i\SP{p,y}}\,
        \chi(\eps p,\eps x)F(Ap+q,x)
        \\
        &=
        \lim_{\eps\to0}
        |\det A|^{-1}
        \int_{\mathbb{R}^{2n}}
        dp\,dx\,e^{i\SP{A^{-1}(p-q),x}}\,
        \chi\left(\eps A^{-1}p,\eps(x+y)\right)F(p,x+y)
        \\
        &=
        \lim_{\eps\to0}
        \int_{\mathbb{R}^{2n}}
        dp\,dx\,e^{i\SP{p,x}}\,
        \chi\left(\eps A^{-1}p,\eps(A^T x+y)\right)
        \,e^{-i\SP{q,x}}F(p,A^T x+y)
        \,.
    \end{align*}
    Since also $(p,x)\mapsto \chi(A^{-1} p,A^Tx)$ is a smooth,
    compactly supported function which is equal to 1 on an open
    neighborhood of the origin, we can use
    Proposition~\ref{proposition:integral-explicit} again to conclude
    that the last line coincides with $(2\pi)^n I(F^{q,y,A})$.
\end{proof}

For the next statement, we represent the bilinear form used in the
oscillating factor as $\SP{p,x}=(p,Mx)$ with some $M\in\GL(n,\Rl)$,
$|\det M|=1$, and the standard Euclidean inner product
$(\,\cdot\,,\,\cdot\,)$ on $\Rl^n$.  The transpose of $M$ with respect
to this inner product will be denote $M^T$.
\begin{proposition}\label{proposition:integration-by-parts}
    Let $F\in \IntSymbol(\Rl^{2n},V)$ and $\mu\in\Nl_0^n$. Then
    $\partial_p^\mu F$, $\partial_x^\mu F$, $(Mx)^{\mu}F$, $(M^Tp)^\mu
    F$ lie in $\IntSymbol(\Rl^{2n},V)$, and
    \begin{align}
        \label{eq:IntegrationByParts}
        I(\partial_p^\mu F)
        &=
        (-i)^{|\mu|}I((Mx)^\mu F)
        \,,\qquad
        I(\partial_x^\mu F)
        =
        (-i)^{|\mu|}I((M^Tp)^\mu F)\,.
    \end{align}
\end{proposition}
\begin{proof}
    For $F\in\Cinfty_0(\Rl^{2n},V)$, the claimed equations amount to
    an integration by parts against the oscillating factor
    $e^{i\SP{p,x}}=e^{i(p,Mx)}$, since
    \begin{align*}
	\partial_p^\mu e^{i(p,Mx)}
	=
	i^{|\mu|}\,(Mx)^\mu\,e^{i(p,Mx)}
	\,,
	\qquad
	\partial_x^\mu e^{i(p,Mx)}
	=
	i^{|\mu|}\,(M^Tp)^\mu\,e^{i(p,Mx)}
	\,.
    \end{align*}
    Thus we only have to show that the functions on the left and right
    hand side in \eqref{eq:IntegrationByParts} are symbols in
    $\IntSymbol(\Rl^{2n},V)$, and can be approximated by compactly
    supported symbols.  So let $F\in \Symbol^{\bm,\brho}(\Rl^{2n},V)$
    for some order $\bm$ and type $-1<\brho\leq1$ for a defining
    system of seminorms $\mathcal Q$ for $V$, and pick a sequence
    $F_n\in \Cinfty_0(\Rl^{2n},V)$ converging to $F$ in the
    $\Symbol^{\bm',\brho'}$-topology for some $\bm'>\bm$,
    $-1<\brho'\leq\brho$.  Then, according to
    Proposition~\ref{proposition:PartialDerivativesOnSymbols},
    $\partial_p^\mu F_n\to \partial_p^\mu F$ in the
    $\Symbol^{\bm'-\brho'|\mu|,\brho'}$-topology.  But as $\brho'>-1$,
    we can apply Proposition~\ref{proposition:SymbolFirstProperties}
    \refitem{item:SymbolsInSymbols} to see that this sequence also
    converges in the $\Symbol^{\bm'+|\mu|,\brho'}$-topology. Hence, by
    definition of $I$, we have $I(\partial_p^\mu F)=\lim_n
    I_0(\partial_p^\mu F_n)$.  Concerning the right hand side in
    \eqref{eq:IntegrationByParts}, note that $(Mx)^\mu$ is a scalar
    symbol of order $|\mu|$ and type $1$. Hence, using
    Corollary~\ref{corollary:SymbolAlgebra} and $\brho\leq1$, we see
    that $(Mx)^\mu F$ is a symbol of order $\bm+|\mu|$ and type
    $\brho$, and thus an element of
    $\IntSymbol(\Rl^{2n},V)$. Moreover, $(Mx)^\mu F_n \to (Mx)^\mu F$
    in the $\Symbol^{\bm'+|\mu|,\brho'}$-topology. Thus $I((Mx)^\mu
    F)=\lim_n I_0((Mx)^\mu F_n)$, and the first identity in
    \eqref{eq:IntegrationByParts} follows.  The proof of the second
    identity is completely analogous.
\end{proof}

We next compute the oscillatory integrals of symbols $(p,x)\mapsto
F(p,x)$ which are constant in either $x$ or $p$. The following result
also explains our choice of normalization factor $(2\pi)^{-n}$ in the
definition \eqref{eq:Inull} of the oscillatory integral.
\begin{proposition}\label{proposition:normalization}
    Let $F\in \IntSymbol(\Rl^{2n},V)$ satisfy $F(p,x)=F(0,x)$ or
    $F(p,x)=F(p,0)$ for all $p,x\in\Rl^n$. Then
    \begin{align}
        I(F)=F(0)\,.
    \end{align}
    In particular, constant symbols $v:(p,x)\mapsto v$, $v\in V$, have
    oscillatory integral $I(v)=v$.
\end{proposition}
\begin{proof}
    We present only the proof for the claims about a symbol $F\in
    \IntSymbol(\Rl^{2n},V)$ satisfying $F(p,x)=F(0,x)$; the arguments
    for the other case are analogous. To evaluate the oscillatory
    integral of $F$, let $\chi\in\Cinfty_0(\Rl^n,\Rl)$, with
    $\chi(x)=1$ for $|x|\leq 1$, and let
    $\tilde{\chi}(x):=(2\pi)^{-n/2}\int_{\mathbb{R}^{2n}}
    dp\,e^{i\SP{p,x}}\chi(p)$ denote the Fourier transform of $\chi$
    with respect to the chosen inner product.  Then
    Proposition~\ref{proposition:integral-explicit} can be applied
    with the product cutoff function $(p,x)\mapsto\chi(p)\chi(x)$, and
    we obtain
    \begin{align*}
        I(F)
        &=
        \lim_{\eps\to0}(2\pi)^{-n}
        \int_{\mathbb{R}^{n}} dp
        \int_{\mathbb{R}^{n}} dx
        \,e^{i\SP{p,x}}\,\chi(\eps p)\chi(\eps x)F(0,x)
        \\
        &=
        \lim_{\eps\to0}(2\pi)^{-n}\eps^{-n}\int_{\mathbb{R}^{n}} dp
        \int_{\mathbb{R}^{n}} dx
        \,e^{i\SP{p,x/\eps}}\,\chi(p)\chi(\eps x)F(0,x)
        \\
        &=
        \lim_{\eps\to0}(2\pi)^{-n/2}\eps^{-n}
        \int_{\mathbb{R}^{n}} dx\,
        \tilde{\chi}(x/\eps)\chi(\eps x)F(0,x)
        \\
        &=
        \lim_{\eps\to0}(2\pi)^{-n/2}\int_{\mathbb{R}^{n}} dx
        \,\tilde{\chi}(x)\chi(\eps^2 x)F(0,\eps x)
        \,.
    \end{align*}
    To show that this limit coincides with $F(0)$, let $\mathcal Q$ be
    a defining system of seminorms on $V$, with order $\bm$ and type
    $-1<\brho\leq1$, such that
    $F\in\Symbol^{\bm,\brho}(\Rl^{2n},V)$. For any $\qn\in\mathcal Q$,
    we have the estimate
    \begin{align*}
        &\qn\left(
            \int_{\mathbb{R}^{n}} dx
            \,\tilde{\chi}(x)\chi(\eps^2 x)F(0,\eps x)
            -
            \int_{\mathbb{R}^{n}} dx
            \,\tilde{\chi}(x)F(0)
        \right) \\
        &\qquad\leq
        \int_{\mathbb{R}^{n}} dx\,
        |\tilde{\chi}(x)|\,\qn\left(\chi(\eps^2 x)F(0,\eps x)-F(0)\right)
        \,.
    \end{align*}
    As the Fourier transform of a smooth, compactly supported
    function, $\tilde{\chi}$ is an element of $\Ss(\Rl^n,\Cl)$, and
    since $F\in \Symbol^{\bm,\brho}(\Rl^{2n},V)$, and $\chi$ is
    bounded, we easily find a scalar integrable function $g$ such that
    $|\tilde{\chi}(x)|\,\qn\left(\chi(\eps^2 x)F(\eps
        x)-F(0)\right)\leq g(x)$ for all $\eps\leq1$. Hence the limit
    $\eps\to0$ of the right hand side of the above estimate can be
    evaluated by dominated convergence, and since $\chi(0)=1$, this
    limit is zero. As $\qn$ was arbitrary, we have in the topology of
    $V$
    \begin{align*}
        I(F)
        &=
        \lim_{\eps\to0}(2\pi)^{-n/2}\int_{\mathbb{R}^{n}} dx
        \,\tilde{\chi}(x)\chi(\eps^2 x)F(0,\eps x)
        =
        (2\pi)^{-n/2}\int_{\mathbb{R}^{n}} dx
        \,\tilde{\chi}(x) F(0)
        \,.
    \end{align*}
    But in view of our normalization of the inner product
    $\SP{\,\cdot\,,\,\cdot\,}$, the inverse Fourier transform gives
    $(2\pi)^{-n/2}\int dx\,\tilde{\chi}(x)=\chi(0)=1$. So we arrive at
    the claimed identity
    \begin{align*}
        I(F)
        =
        \chi(0) F(0)
        =
        F(0)
        \,.
    \end{align*}
    The statement about constant symbols follows by choosing $F$
    constant.
\end{proof}

As the last property of oscillatory integrals needed in our
applications, we discuss a Fubini type theorem for multiple
oscillatory integrals. To this end, we consider symbols $F\in
\IntSymbol(\Rl^{2n_1}\oplus\Rl^{2n_2},V)$ depending on two pairs of
variables $(p_1,x_1), (p_2,x_2)$, with $p_j,x_j\in\Rl^{n_j}$,
$j=1,2$. Our discussion of multiple oscillatory integrals is greatly
facilitated by Proposition~\ref{proposition:symbol-valued-symbols},
stating that for $F\in\IntSymbol(\Rl^{2n_1}\oplus\Rl^{2n_2},V)$, the
maps
\begin{equation}
    \label{eq:F1F2}
    F_1 \colon \Rl^{2n_1} \longrightarrow
    \IntSymbol(\Rl^{2n_2},V)
    \quad
    \textrm{and}
    \quad
    F_2\colon \Rl^{2n_2} \longrightarrow
    \IntSymbol(\Rl^{2n_1},V),
\end{equation}
given by
\begin{equation*}
    F_1(p_1,x_1) \colon (p_2,x_2) \mapsto
    F(p_1,x_1;p_2,x_2)
    \quad
    \textrm{and}
    \quad
    F_2(p_2,x_2) \colon (p_1,x_1) \mapsto
    F(p_1,x_1;p_2,x_2),
\end{equation*}
respectively, are symbols in their own right.

Subsequently we have to distinguish different kinds of oscillatory
integrals.  On $\Rl^{2n_1}\oplus\Rl^{2n_2}$ we use the induced pairing
\begin{align}\label{sum-of-bilinear-forms}
    \SP{p_1\oplus p_2,\,x_1\oplus x_2} = \SP{p_1,x_1} + \SP{p_2,x_2}\,.
\end{align}
Hence the oscillatory integrals over symbols
$F\in\IntSymbol(\Rl^{2n_1}\oplus\Rl^{2n_2},V)$ will be carried out
with respect to the oscillating factor
$e^{i\SP{p_1,x_1}}e^{i\SP{p_2,x_2}}$, and denoted by $I$ as before. On
the other hand, oscillatory integrals over the symbols \eqref{eq:F1F2}
which are defined on $\Rl^{2n_1}$ (respectively $\Rl^{2n_2}$), and
take values in some other symbol space
$\Symbol^{\bm,\brho}(\Rl^{2n_2},V)$ (respectively
$\Symbol^{\bm,\brho}(\Rl^{2n_1},V)$), will be carried out with respect
to the oscillating factors $e^{i\SP{p_1,x_1}}$ (respectively
$e^{i\SP{p_2,x_2}}$), and denoted $\hat{I}_1$ (respectively
$\hat{I}_2$).
\begin{proposition}
    \label{proposition:Fubini}
    Let $V$ be a sequentially complete locally convex space. For any
    $F\in\IntSymbol(\Rl^{2n_1}\oplus\Rl^{2n_2},V)$, we have
    $\hat{I}_1(F_1)\in\IntSymbol(\Rl^{2n_2},V)$,
    $\hat{I}_2(F_2)\in\IntSymbol(\Rl^{2n_1},V)$, with
    \begin{align}\label{eq:fubini}
        I_2(\hat{I}_1(F_1)) = I_1(\hat{I}_2(F_2)) = I(F) \,.
    \end{align}
    Equivalently,
    \begin{align*}
        &(2\pi)^{-n_2}\int_{\Rl^{2n_2}} dp_2\,dx_2\,e^{i\SP{p_2,x_2}}
        \left(
            (2\pi)^{-n_1}
            \int_{\Rl^{2n_1}} dp_1\,dx_1\,e^{i\SP{p_1,x_1}}
            \,F(p_1,x_1;p_2,x_2)
        \right)
        \\
        =
        &(2\pi)^{-n_1}\int_{\Rl^{2n_1}} dp_1\,dx_1\,e^{i\SP{p_1,x_1}}
        \left(
            (2\pi)^{-n_2}
            \int_{\Rl^{2n_2}} dp_2\,dx_2\,e^{i\SP{p_2,x_2}}
            \,F(p_1,x_1;p_2,x_2)
        \right)
        \\
        =
        &(2\pi)^{-(n_1+n_2)}\int_{\Rl^{2n_1}\oplus\Rl^{2n_2}}
        dp\,dx\,e^{i\SP{p,x}}
        \,F(p,x)
    \end{align*}
\end{proposition}
\begin{proof}
    Let $\mathcal Q$ be a defining system of seminorms for $V$, and
    $\bm$, $-1<\brho\leq1$ an order and a type for $\mathcal Q$ such
    that
    $F\in\Symbol^{\bm,\brho}(\Rl^{2n_1}\oplus\Rl^{2n_2},V)$. According
    to Proposition~\ref{proposition:symbol-valued-symbols},
    $F_1\in\Symbol^{\hat{\bm},\hat{\brho}} (\Rl^{2n_1},
    \Symbol^{\bm,\brho}(\Rl^{2n_2}, V))$,
    $F_2\in\Symbol^{\hat{\bm},\hat{\brho}} (\Rl^{2n_2},
    \Symbol^{\bm,\brho}(\Rl^{2n_1}, V))$ with the order $\hat{\bm}$
    and type $\hat{\brho}$ defined in
    \eqref{order-and-type-for-S}. Note that since
    $-1<\hat{\brho}\leq1$, the oscillatory integrals
    $\hat{I}_1(F_1)\in\Symbol^{\bm,\brho}(\Rl^{2n_2},V)\subset\IntSymbol(\Rl^{2n_2},
    V )$ and
    $\hat{I}_2(F_2)\in\Symbol^{\bm,\brho}(\Rl^{2n_1},V)\subset\IntSymbol(\Rl^{2n_1}
    , V )$ exist. Hence all integrals in \eqref{eq:fubini} are
    well-defined. To show that they coincide, we can argue with the
    usual continuity and approximation techniques: for compactly
    supported symbols the integrals coincide by the Fubini theorem for
    Riemann integrals. Then the continuity statements of
    Proposition~\ref{proposition:symbol-valued-symbols} and
    Theorem~\ref{Theorem:OscillatoryIntegral} give the equality for
    all symbols.
\end{proof}

\section{Rieffel deformations for polynomially bounded $\Rl^n$-actions}
\label{section:rieffel}

We now apply the symbol calculus developed so far to extend Rieffel's
deformation of Fr\'echet algebras with isometric $\Rl^n$-actions
\cite{rieffel:1993a} to a more general setting. As before, we will
consider functions taking values in locally convex sequentially
complete vector spaces $V$. We will in this chapter always assume a
filtrating defining system $\cal Q$ of seminorms for $V$. The symbols
we are interested in will be generated with the help of suitable
$\Rl^n$-actions, and we introduce some standard notation first.

For an $\Rl^n$-action $\alpha\colon \Rl^n\times V\longrightarrow V$,
we consider the functions
\begin{align}
	\alpha(v)\colon \Rl^n\longrightarrow V,\qquad x\mapsto\alpha_x(v)
\end{align}
for $v \in V$.  The action will be called strongly smooth if
$\alpha(v)\in\Cinfty(\Rl^n,V)$ for all $v\in V$. Its derivatives at
$x=0$ are denoted by
\begin{align}
	X^\mu\colon V\longrightarrow V\,,
        \qquad
        X^\mu v := \partial_x^\mu\alpha_x(v)|_{x=0}\,,
\end{align}
where $\mu \in \mathbb{N}_0^n$ as usual.  All actions will be assumed
to act by linear maps $\alpha_x\colon V\longrightarrow V$. If $\alpha$
is strongly smooth and the $\alpha_x$ are continuous for all
$x\in\Rl^n$, then one has $\partial_x^\mu\alpha_x(v)=X^\mu\alpha_x v =
\alpha_x X^\mu v$.
\begin{definition}\label{definition:SmoothSymbolAction}
    Let $V$ be a sequentially complete locally convex space with
    defining system of seminorms $\mathcal Q$, and let $\bm$ be an
    order for $\cal Q$. A smooth polynomially bounded $\Rl^n$-action
    (of order $\bm$) is an action $\alpha\colon \Rl^n\times
    V\longrightarrow V$ such that
    \begin{definitionlist}
    \item \label{item:AlphavIstSymbol} $\alpha(v)\in
        \Symbol^{\bm,0}(\Rl^n,V)$ for each $v\in V$.
    \item\label{item:VtoAlphaVContinuous} $V\ni v\mapsto
        \alpha(v)\in\Symbol^{\bm,0}(\Rl^n,V)$ is continuous, i.e.\ for
        any $\qn\in\mathcal Q$, $\mu\in\Nl_0^n$, there exists
        $\qn'\in\mathcal Q$, such that
        \begin{align}
            \|\alpha(v)\|^{\bm,0}_{\qn,\mu}
            \leq
            \qn'(v)
        \end{align}
        for all $v\in V$.
    \end{definitionlist}
\end{definition}
Smooth polynomially bounded actions can equivalently be characterized
as follows.
\begin{lemma}\label{lemma:SmoothSymbolActionAlternative}
    Let $\alpha$ be a strongly smooth action on $V$. Then $\alpha$ is
    polynomially bounded of order $\order{m}$ in the sense of
    Definition~\ref{definition:SmoothSymbolAction} if and only if the
    following two conditions are satisfied:
    \begin{lemmalist}
    \item \label{item:EquivalentToSmoothActionI} For each $\qn\in\cal
        Q$, there exists $\qn'\in\cal Q$ such that
        \begin{align}\label{eq:BoundAlphaX}
            \qn(\alpha_x(v))
            \leq
            (1+\|x\|^2)^{\frac{1}{2}\order{m}(\qn)}\,\qn'(v)
        \end{align}
        for all $x\in\Rl^n$, $v\in V$.
    \item \label{item:EquivalentToSmoothActionII} The derivatives
        $X^\mu\colon V\longrightarrow V$ are continuous.
    \end{lemmalist}
    If $V$ is a Fr\'echet space and $\mathcal{Q}$ is chosen countable
    then $\alpha$ is polynomially bounded of order $\order{m}$ if and
    only if just the first condition is satisfied.
\end{lemma}
\begin{proof}
    Let $\alpha$ be a smooth polynomially bounded action of order
    $\order{m}$. Then its defining properties imply that for any
    $\qn\in\cal Q$ there exists $\qn'\in\cal Q$ such that for all
    $x\in\Rl^n$, $v\in V$,
    \begin{align*}
        \qn(\alpha_x(v))
        \leq
        (1+\|x\|^2)^{\frac{1}{2}\order{m}(\qn)}\|\alpha(v)\|^{\order{m},0}_{\qn,0}
        \leq
        (1+\|x\|^2)^{\frac{1}{2}\order{m}(\qn)}\qn'(v)\,,
    \end{align*}
    i.e. \eqref{eq:BoundAlphaX} holds. Furthermore, given $\qn\in\cal
    Q$, $\mu\in\Nl_0^n$, there exists $\qn'\in\cal Q$ such that
    \begin{align*}
        \qn(X^\mu v)
        &=
        \qn(\partial_x^\mu\alpha_x(v)|_{x=0})
        \leq
        \sup_{x\in\Rl^n}\frac{\qn(\partial_x^\mu\alpha_x(v))}{(1+\|x\|^2)^{\frac{1}{2}
            \order{m}(\qn)}}
        =
        \|\alpha(v)\|^{\bm,0}_{\qn,\mu}
        \leq
        \qn'(v)
    \end{align*}
    for all $v\in V$. Hence $X^\mu\colon V\longrightarrow V$ is
    continuous.  Now assume that $\alpha$ is a strongly smooth action
    satisfying the two conditions listed in this lemma. Then to any
    $\qn\in\cal Q$, $\mu\in\Nl_0^n$, there exist $\qn',\qn''\in\cal Q$
    such that
    \begin{align*}
        \qn(\partial_x^\mu\alpha_x(v))
        =
        \qn(\alpha_x(X^\mu v))
        \leq
        (1+\|x\|^2)^{\frac{1}{2}\order{m}(\qn)}\qn'(X^\mu v)
        \leq
        (1+\|x\|^2)^{\frac{1}{2}\order{m}(\qn)}\qn''(v)
    \end{align*}
    for all $x\in\Rl^n$, $v\in V$. This shows both,
    $\alpha(v)\in\Symbol^{\order{m},0}(V)$, and the continuity of
    $v\mapsto\alpha(v)$. Hence
    Definition~\ref{definition:SmoothSymbolAction} and the two
    conditions in this lemma are equivalent.  We now consider the
    special but important case that $V$ is a Fr\'echet space. Thus
    assume that $\mathcal{Q}$ is countable. When equipped with the
    family of seminorms $\mathcal{Q}^\infty:=\{\qn_\mu:=\qn\circ
    X^\mu\,|\,\qn\in{\cal Q},\,\mu\in\Nl_0^n\}$, this space will be
    called $V^\infty$. Also $V^\infty$ is a Fr\'echet space.  As
    linear spaces, $V=V^\infty$, and clearly, the identity $\id\colon
    V^\infty\longrightarrow V$ is linear, continuous and
    bijective. Hence we can apply the open mapping theorem for
    Fr\'echet spaces to conclude that $\id\colon V\longrightarrow
    V^\infty$ is continuous as well, i.e.\ $V=V^\infty$ as Fr\'echet
    spaces. But this is equivalent to the derivatives $X^\mu\colon
    V\longrightarrow V$ being continuous, i.e.\ condition
    \refitem{item:EquivalentToSmoothActionII} is automatically
    satisfied.
\end{proof}
\begin{remark}\label{remark:Actions}
    \begin{remarklist}
    \item \label{item:AlphaXContinuous} A polynomially bounded action
        $\alpha$ acts by continuous maps $\alpha_x$, as can be read
        off from \eqref{eq:BoundAlphaX}.
    \item \label{item:OnlyTypeRhoNullNeeded} The arguments in the
        above lemma also explain why we consider only actions of type
        $\type{\rho}=0$ here. For if
        $\alpha(v)\in\Symbol^{\order{m},\type{\rho}}(V)$ and the
        derivatives $X^\mu$ are continuous, we can argue as above to
        show that $\alpha(v)$ is actually of type $0$.
    \end{remarklist}
\end{remark}

In Rieffel's original approach, $V$ is taken to be a Fr\'echet algebra
with a strongly continuous $\Rl^n$-action $\alpha$ by automorphisms
$\alpha_x$ which are isometric for all $\qn\in\mathcal Q$. On the
subspace $V^\infty$ of all smooth vectors for $\alpha$, the
$\alpha(v)$ are then symbols of order $0$ and type $0$, see
\cite{rieffel:1993a}. Using the symbol calculus of the preceding
sections, we will now extend many of the results of Rieffel to the
case where $\alpha(v)\in\Symbol^{\bm,0}(\Rl^{n},V)$ for an arbitrary
order $\bm$. In Section~\ref{section:examples}, we will then provide
various examples of smooth polynomially bounded $\Rl^n$-actions.

In comparison to \cite{rieffel:1993a}, we will take here also a
somewhat more general point of view concerning the algebraic
structure, which involves three sequentially complete locally convex
spaces $V,W,U$ with filtrating defining systems of seminorms ${\cal
  Q}^V, {\cal Q}^W, {\cal Q}^U$. Each of these spaces is equipped with
a smooth polynomially bounded $\Rl^n$-action
$\alpha^V,\alpha^W,\alpha^U$ of order $\order{m}^V$, $\order{m}^W$,
$\order{m}^U$, respectively, and the derivatives with respect to these
actions will be denoted $X_U^\mu$, $X_V^\mu$, $X_W^\mu$.

In this setting, we consider a bilinear map
\begin{align}\label{eq:BilinearMapMu}
    \mu\colon V\times W \longrightarrow U
\end{align}
which is required to be covariant in the sense that
\begin{align}\label{eq:Covariance}
    \alpha_x^U\mu(v,w)=\mu(\alpha^V_x v,\alpha^W_x w)
    \,,\qquad
    v\in V,\,w\in W,\,x\in\Rl^n\,.
\end{align}
In many applications $\mu$ will be jointly continuous, but in some
cases we also need to work with a bilinear map $\mu$ which is only
separately continuous. In the following we will therefore always only
assume that $\mu$ is separately continuous, and explicitly point out
when we consider the special case that $\mu$ is jointly continuous.

This setting includes the case where $\A:=V=W=U$ is an algebra with
(separately) continuous product $\mu$, and
$\alpha:=\alpha^V=\alpha^W=\alpha^U$ acts by automorphisms. But the
more general formulation allows, for example, to also consider
covariant modules, where $\A:=V$ is taken to be an algebra and
$\E:=W=U$ is a left $\A$-module with a smooth $\Rl^n$-action
$\beta:=\alpha^W=\alpha^U$, and (separately) continuous module
structure $\mu\colon \A \times \E \longrightarrow \E$ satisfying
\eqref{eq:Covariance}.  This setup will therefore be suitable for the
discussion of deformations of algebras and their covariant modules. In
the following, we will always assume without further mentioning that
spaces $V,W,U$, actions $\alpha^V,\alpha^W,\alpha^U$, and a bilinear
map $\mu$ with the specified properties are given.

Following Rieffel, we now consider a real $(n\times n)$-matrix $\te$
as our deformation parameter, and introduce the functions, $v\in V$,
$w\in W$,
\begin{align}\label{eq:MuSymbols}
 \mu^\te_{vw}\colon \Rl^n\times\Rl^n \longrightarrow U,
\qquad
 \mu^\te_{vw}(p,x) &:= \mu(\alpha^V_{\te p}v,\alpha^W_x w)
\,.
\end{align}
As $\alpha^V$, $\alpha^W$ are smooth and polynomially bounded, it
follows from Proposition~\ref{proposition:SymbolProducts},
\refitem{item:ProductInTwoVariables}, that these functions are symbols
in $\IntSymbol(\Rl^{2n},U)$ if $\mu$ is jointly continuous. If $\mu$
is however only separately continuous, more work is needed to arrive
at this conclusion. We begin with the following lemma.
\begin{lemma}\label{lemma:Falpha}
    Let $\alpha$ be a smooth polynomially bounded action on
    $V$. Moreover, let $F\in\Cinfty(\mathbb{R}^n, V)$. Then
    \begin{align*}
        F^\alpha:\Rl^n\times\Rl^n \longrightarrow V\,,\qquad
        F^\alpha(p,x) := \alpha_p(F(x))
    \end{align*}
    is smooth.
\end{lemma}
\begin{proof}
    First we show that $F^\alpha$ is continuous. To this end, we make
    use of the differentiability of $\alpha$, which implies that for
    any $\qn\in\cal Q$, there exists $\qn'\in\cal Q$ and a continuous
    function $f\colon \Rl^n\times\Rl^n \longrightarrow \Rl_+$ such
    that for all $p,p'\in\Rl^n$, $v\in V$
    \begin{align*}
        \qn(\alpha_p(v)-\alpha_{p'}(v))
        \leq
        \|p-p'\|\,f(p,p')\,\qn'(v)\,.
    \end{align*}
    The proof of this estimate can be carried out along the same lines
    as in Lemma~\ref{lemma:tauySymbolContinuous}. With this bound we
    find, $p,p',x,x'\in\Rl^n$,
    \begin{align*}
        &\qn\left(F^\alpha(p,x)-F^\alpha(p',x')\right) \\
        &\quad\leq
        \qn(\alpha_p(F(x))-\alpha_{p'}(F(x))) + \qn(\alpha_{p'}(F(x))-\alpha_{p'}(F(x')))
        \\
        &\quad\leq
        \|p-p'\|\,f(p,p')\,\qn'(F(x)) + (1+\|p'\|^2)^{\frac{1}{2}\order{m}(\qn)}\qn''(F(x)-F(x'))
        \,,
    \end{align*}
    with some $\qn',\qn''\in\cal Q$. The continuity of $F^\alpha$ is
    then clear.  Furthermore, $(p,x) \mapsto F^\alpha(p,x)$ is
    separately smooth in $p$ and $x$ (in $x$ because the $\alpha_p$
    are linear and continuous), with partial derivatives
    \begin{align*}
        \partial_p^\nu F^\alpha(p,x)
        &=
        \alpha_p X^\nu F(x)
        \,,\qquad
        \partial_x^\nu F^\alpha(p,x)
        =
        \alpha_p \partial_x^\nu F(x)
        \,.
    \end{align*}
    According to Lemma~\ref{lemma:SmoothSymbolActionAlternative}, the
    derivatives $X^\nu\colon V\longrightarrow V$ are continuous. Thus
    $x\mapsto X^\nu F(x)$ is smooth, and clearly,
    $x\mapsto \partial_x^\nu F(x)$ is smooth as well. Hence the
    partial derivatives of $F^\alpha$ are of the same form as
    $F^\alpha$, and thus in particular continuous. This implies that
    $F^\alpha$ is smooth.
\end{proof}
\begin{lemma}\label{lemma:FvwGvw}
    Let $v\in V$, $w\in W$ and consider
    $F_{vw}, G_{vw}\colon \Rl^n\times\Rl^n \longrightarrow U$ defined by
    \begin{equation}
        F_{vw}(p,x)
        :=
        \alpha^U_p\left(\mu\left(v, \alpha^W_x(w)\right)\right)
        \quad
        \textrm{and}
        \quad
        G_{vw}(p,x)
        :=
        \alpha^U_x\left(\mu\left(\alpha^V_p(v) ,w\right)\right)\,.
    \end{equation}
    Then there exists an order $\order{\hat{m}}$ on ${\cal Q}^U$ such
    that $F_{vw},G_{vw}\in\Symbol^{\order{\hat{m}},0}(\Rl^{2n},U)$,
    and for fixed $v_0\in V$, $w_0\in W$, the mappings
    \begin{align*}
        W \ni w \mapsto F_{v_0w} \in \Symbol^{\order{\hat{m}},0}(\Rl^{2n},U)\\
        V \ni v \mapsto G_{vw_0} \in \Symbol^{\order{\hat{m}},0}(\Rl^{2n},U)
    \end{align*}
    are linear and continuous.
\end{lemma}
\begin{proof}
    We will only prove the statements about $F_{vw}$ as the discussion
    of $G_{vw}$ is completely analogous. In view of the separate
    continuity of $\mu$, the map $w\mapsto\mu(v,w)$ is continuous for
    fixed $v$, and as $\alpha^W$ is a smooth action, we see that
    $F_{vw}$ is of the form considered in the previous lemma and hence
    smooth. Its partial derivatives are, $\nu,\kappa\in\Nl_0^n$,
    \begin{align*}
        \partial_p^\nu\partial_x^\kappa F_{vw}(p,x)
        =
        \alpha_p^U\left(
            X_U^\nu \mu\left(
                v,
                \alpha_x^W (X_W^\kappa w)
            \right)
        \right)
        \,.
    \end{align*}
    To estimate these derivatives, let $\qn\in{\cal Q}^U$. Then there
    exists $\qn'\in{\cal Q}^U$ such that
    \begin{align*}
        \qn(\partial_p^\nu\partial_x^\kappa F_{vw}(p,x))
        &\leq
        (1+\|p\|^2)^{\frac{1}{2}\order{m}^U(\qn)}\qn'(X_U^\nu \mu(v,\alpha_x^W X_W^\kappa w))
        \,.
    \end{align*}
    As $w\mapsto\mu(v,w)$ is continuous and $U=U^\infty$, $W=W^\infty$
    as locally convex spaces, we now find
    $\qn'',\qn''',\qn''''\in{\cal Q}^W$, $\sigma\in\Nl_0^n$, and
    constants $c, C_v>0$ such that
    \begin{align*}
        \qn(\partial_p^\nu\partial_x^\kappa F_{vw}(p,x))
        &\leq
        (1+\|p\|^2)^{\frac{1}{2}\order{m}^U(\qn)}\,C_v\,\qn''(X_W^\sigma X_W^\kappa \alpha_x^W w)
        \\
        &\leq
        C_v\,(1+\|p\|^2)^{\frac{1}{2}\order{m}^U(\qn)}(1+\|x\|^2)^{\frac{1}{2}\order{m}^W(\qn'')}\qn'''(X_W^\sigma X_W^\kappa w)
        \\
        &\leq
        C_v\,c\,(1+\|p\|^2+\|x\|^2)^{\frac{1}{2}(|\order{m}^U(\qn)|+|\order{m}^W(\qn'')|)}\qn''''(w)
    \end{align*}
    for all $p,x\in\Rl^n$, $w\in W$. Here $c$ depends on
    $\kappa,\nu,\qn$, but not on $v,w,p,x$. This estimate shows that
    $F_{vw}$ is a symbol in $\Symbol^{\order{\hat{m}},0}(\Rl^{2n},U)$,
    with
    $\order{\hat{m}}(\qn):=|\order{m}^U(\qn)|+|\order{m}^W(\qn'')|$
    (cf. Proposition~\ref{proposition:SymbolProducts}). Moreover,
    \begin{align*}
        \|F_{vw}\|^{\order{\hat{m}},0}_{\qn,\nu\oplus\kappa}
        \leq
        C_v\,c\,\qn''''(w)\,,
    \end{align*}
    that is, $w\mapsto F_{vw}$ is continuous for fixed $v$.
\end{proof}

After these preparations we can derive the following basic statement
about the deformation of $\mu$.
\begin{proposition}\label{proposition:DeformedMuBasicProperties}
    Let $v\in V$, $w\in W$, and $\te\in\Rl^{n\times n}$.
    \begin{propositionlist}
    \item The functions $\mu^\te_{vw}$ are symbols in
        $\underline{\Symbol}(\Rl^{2n},U)$.
    \item The maps defined by their oscillatory integrals
        \begin{align}\label{eq:MuTheta}
            V\times W \ni (v,w)
            \mapsto
            \mu_\te(v,w)
            :=
            I_U(\mu^\te_{vw})
            \in U
        \end{align}
        are bilinear and (separately) continuous if $\mu$ is
        (separately) continuous.
    \item $\mu_\te$ satisfies the covariance property
        \eqref{eq:Covariance}.
    \end{propositionlist}
\end{proposition}
\begin{proof}
    Let $v\in V$, $w\in W$ and $\te$ be fixed. Thanks to the
    covariance of $\mu$ \eqref{eq:Covariance}, we have
    \begin{align*}
	\mu^\te_{vw}(p,x)
	=
	F_{vw}(\te p,x-\te p)
	=
	G_{vw}(x,\te p-x)
	\,.
    \end{align*}
    As $F_{vw}, G_{vw}\in\underline{\Symbol}(\mathbb{R}^{2n}, U)$, an
    application of Lemma~\ref{lemma:affine-substitutions} shows
    $\mu^\te_{vw}\in\underline\Symbol(\mathbb{R}^{2n}, U)$ and hence
    the first part. By the preceding lemma, we see that
    \begin{align*}
	V\times W \ni (v,w)
        \mapsto
        \mu^\te_{vw} \in \Symbol^{\order{\hat{m}},0}(\Rl^{2n},U)
    \end{align*}
    is separately continuous for some order $\order{\hat{m}}$ on
    ${\cal Q}^U$. Hence the oscillatory integral $I_U(\mu^\te_{vw})$
    exists and depends separately continuous on $v,w$. The bilinearity
    of $\mu^\te$ is clear.  In case $\mu$ is jointly continuous, note
    that by assumption, $\alpha^V_\te(v) \in
    \Symbol^{\bm^V,0}(\Rl^n,V)$ and $\alpha^W(w) \in
    \Symbol^{\bm^W,0}(\Rl^n,W)$ with orders $\bm^V,\bm^W$ for ${\cal
      Q}^V$, ${\cal Q}^W$. The maps $v \mapsto \alpha^V_\te(v)$ and $w
    \mapsto \alpha^W(w)$ are continuous from $V$ (respectively $W$) to
    $\Symbol^{\bm^V,0}(\Rl^n, V)$ (respectively
    $\Symbol^{\bm^W,0}(\Rl^n, W)$) by
    Definition~\ref{definition:SmoothSymbolAction},
    \refitem{item:VtoAlphaVContinuous}.  Furthermore, according to
    Proposition~\ref{proposition:SymbolProducts},
    \refitem{item:ProductInTwoVariables}, $\alpha^V_\te(v),
    \alpha^W(w) \mapsto \mu(\alpha^V_\te(v), \alpha^W(w))$ maps
    $\Symbol^{\bm^V,0}(\Rl^n,V)\times \Symbol^{\bm^W,0}(\Rl^n,W)$
    continuously into $\Symbol^{\bm',0}(\Rl^{2n},U)$, with some order
    $\bm'$ for ${\cal Q}^U$. Finally, the oscillatory integral maps
    $\mu^\te_{vw}\mapsto \mu_\te(v,w)$ continuously from
    $\Symbol^{\bm',0}(\Rl^{2n},U)$ to $U$ by
    Theorem~\ref{Theorem:OscillatoryIntegral}. As a composition of
    these continuous maps, $\mu_\te\colon V\times W\to U$ is therefore
    continuous, too. This completes the second part.  To check the
    covariance property \eqref{eq:Covariance}, note that since
    $\alpha_x^U$ is continuous for each $x\in\Rl^n$ by
    Remark~\ref{remark:Actions} \refitem{item:AlphaXContinuous}, it
    can be pulled inside the oscillatory integral defining $\mu_\te$
    according to Lemma~\ref{lemma:AinI}. Since \eqref{eq:Covariance}
    holds for $\mu$, and $\alpha^V$, $\alpha^W$ are $\Rl^n$-actions,
    it follows that $\mu_\te$ satisfies \eqref{eq:Covariance} as well:
    \begin{align*}
	\alpha_x^U\mu_\te(v,w)
        &=
        (2\pi)^{-n}\int_{\mathbb{R}^{2n}}
        dp\,dy\,e^{i\SP{p,y}}\,
        \alpha^U_x\mu(\alpha^V_{\te q}(v),\,\alpha^W_y(w))
        \\
        &=
        (2\pi)^{-n}\int_{\mathbb{R}^{2n}}
        dp\,dy\,e^{i\SP{p,y}}\,
        \mu(\alpha^V_{\te q+x}(v),\,\alpha^W_{y+x}(w))
        \\
        &=
        \mu_\te(\alpha^V_x(v),\,\alpha^W_x(w))
        \,.
    \end{align*}
\end{proof}

Depending on the context, $\mu_\te$ from \eqref{eq:MuTheta} will be
referred to as the {\em deformed product}, {\em deformed module
  structure}, or just {\em deformed bilinear map}.

In the next proposition we justify these names by demonstrating the
most basic feature of a deformation, namely that it reduces to the
identity for vanishing deformation parameter.
\begin{proposition}\label{proposition:DeformationBasics}
    Let $\te,\te'\in\Rl^{n\times n}$.
    \begin{propositionlist}
    \item \label{item:MuThetaIsMu}For $\te=0$, we have $\mu_0=\mu$.
    \item \label{item:MuAdditiveInTheta}
        $(\mu_\te)_{\te'}=\mu_{\te+\te'}$.
    \end{propositionlist}
\end{proposition}
\begin{proof}
    Let $v\in V, w\in W$.  For $\te=0$, the symbol $(p,x) \mapsto
    \mu(\alpha^V_{\te p}v,\alpha^W_x w)$ \eqref{eq:MuSymbols} is
    independent of $p$. Hence
    Proposition~\ref{proposition:normalization} applies, and we have
    $\mu_0(v,w)=\mu_\te(\alpha^V_{0}(v),\alpha_0^W(w)) = \mu(v,w)$,
    showing the first part. For the second part, by
    Proposition~\ref{proposition:DeformedMuBasicProperties}, $\mu_\te$
    has the same properties as $\mu$, so $(\mu_\te)_{\te'}$ is
    well-defined. Using successively the definition of $\mu_\te$, the
    substitution $x\to x-x'$ according to
    Lemma~\ref{lemma:affine-substitutions} and Fubini's theorem in
    form of Proposition~\ref{proposition:Fubini}, we compute
    \begin{align*}
        &(\mu_\te)_{\te'}(v,w) \\
        &\quad=
        (2\pi)^{-n}\int_{\mathbb{R}^{2n}}
        dp'\,dx'\,e^{i\SP{p',x'}}\,
        \mu_\te(\alpha^V_{\te'p'}(v),\alpha^W_{x'}(w))
        \\
        &\quad=
        (2\pi)^{-n}\int_{\mathbb{R}^{2n}}
        dp'\,dx'\,e^{i\SP{p',x'}}\,\left(
            (2\pi)^{-n}\int_{\mathbb{R}^{2n}}
            dp\,dx\,e^{i\SP{p,x}}\,
            \mu(\alpha^V_{\te p +\te'p'}(v),\alpha^W_{x+x'}(w))
        \right)
        \\
        &\quad=
        (2\pi)^{-n}\int_{\mathbb{R}^{2n}}
        dp'\,dx'\,e^{i\SP{p',x'}}\,\left(
            (2\pi)^{-n}\int_{\mathbb{R}^{2n}}
            dp\,dx\,e^{i\SP{p,x}}e^{-i\SP{p,x'}}\,
            \mu(\alpha^V_{\te p +\te'p'}(v),\alpha^W_{x}(w))
        \right)
        \\
        &\quad=
        (2\pi)^{-n}\int_{\mathbb{R}^{2n}}
        dp\,dx\,e^{i\SP{p,x}}
        \left(
            (2\pi)^{-n}\int_{\mathbb{R}^{2n}}
            dp'\,dx'\,e^{i\SP{p',x'}}
            \,e^{-i\SP{p,x'}}
            \mu(\alpha^V_{\te'p'+\te p}(v), \alpha^W_{x}(w))
        \right)
        .
    \end{align*}
    Lemma~\ref{lemma:affine-substitutions} and
    Proposition~\ref{proposition:normalization} show that the inner
    oscillatory integral has the value
    $\mu(\alpha^V_{(\te+\te')p}(v),\alpha^W_{x}(w))$. Plugging this
    result into the above computation gives the desired answer
    $(\mu_\te)_{\te'}(v,w)=\mu_{\te+\te'}(v,w)$ by definition of
    $\mu_\te$.
\end{proof}

The following lemma shows two further invariance properties of the
deformation which are helpful in many situations.
\begin{lemma}\label{lemma:InvariancesOfMu}
    \begin{propositionlist}
    \item\label{item:MuOnInvariantVectors} Let $v\in V$ and $w\in
        W$. If either $v$ is $\alpha^V$-invariant or $w$ is
        $\alpha^W$-invariant, then $\mu_\te(v,w)=\mu(v,w)$.
    \item\label{item:MuComposedWithInvariantMap} Let $Y$ be another
        sequentially complete locally convex vector space, and
        $T\colon U\longrightarrow Y$ linear and continuous. If
        $\te\in\Rl^{n\times n}$ is skew-symmetric and $T$ is
        $\alpha^U$-invariant, i.e. $T\circ\alpha^U_x=T$ for all
        $x\in\Rl^n$, then
        \begin{align}
            T \mu_\te(v,w) = T \mu(v,w)\,.
        \end{align}
    \end{propositionlist}
\end{lemma}
\begin{proof}
    For part~\refitem{item:MuOnInvariantVectors}, note that under the
    specified circumstances, the symbol $(p,x)\mapsto\mu(\alpha^V_{\te
      p}(v),\alpha^W_x(w))$ depends only on one of its two variables
    $p,x$. Hence Proposition~\ref{proposition:normalization} applies,
    and we have $\mu_\te(v,w)=\mu(\alpha^V_0(v),\alpha^W_0(w)) =
    \mu(v,w)$.  For part~\refitem{item:MuComposedWithInvariantMap},
    let $v\in V$, $w\in W$. Using the continuity and linearity of $T$
    as in Lemma~\ref{lemma:AinI}, as well as the covariance
    \eqref{eq:Covariance} and the invariance of $T$ gives
    \begin{align*}
        (2\pi)^n\,T\mu_\te(v,w)
        &=
        \int_{\mathbb{R}^{2n}}
        dp\,dx\,e^{i\SP{p,x}}
        \,T\mu(\alpha^V_{\te p}(v),\alpha^W_x(w)) \\
        &=
        \int_{\mathbb{R}^{2n}}
        dp\,dx\,e^{i\SP{p,x}} \,T\mu(\alpha^V_{\te p-x}(v), w)
        \,.
    \end{align*}
    Now we use Lemma~\ref{lemma:affine-substitutions} to carry out the
    substitution $x\to x+\te p$. As $\te$ is skew-symmetric,
    $\SP{p,\te p}=0$, and we get
    \begin{align*}
        T\mu_\te(v,w)
        &=
        (2\pi)^{-n}\int_{\mathbb{R}^{2n}}
        dp\,dx\,e^{i\SP{p,x}} \,T\mu(\alpha^V_{\te p}(v), w)
        \,.
    \end{align*}
    This is again an oscillatory integral over a symbol which is
    constant in one variable, and by
    Proposition~\ref{proposition:normalization}, we arrive at
    $T\mu_\te(v,w) = T\mu(\alpha^V_0(v),w) = T\mu(v,w)$.
\end{proof}

We now consider the deformation of algebras and modules. Let $\A:=V$
be an algebra with separately continuous product $\mu\colon
\A\times\A\longrightarrow \A$, and assume that the (smooth,
polynomially bounded) $\Rl^n$-action $\alpha$ acts by
automorphisms. Furthermore, let $\E:=W=U$ be a left $\A$-module with
separately continuous module map $\tilde{\mu}\colon
\A\times\E\longrightarrow \E$ and a (smooth, polynomially bounded)
$\Rl^n$-action $\beta$ satisfying \eqref{eq:Covariance} with $V=\A$,
$W=U=\E$, $\alpha^V=\alpha$, and $\alpha^W=\alpha^U=\beta$.  In this
situation, we can deform the product $\mu$ according to
\begin{align}\label{eq:DeformedProduct}
    \mu_\te(a,b)
    :=
    a\times_\te b
    :=
    (2\pi)^{-n}\int_{\mathbb{R}^{2n}}
    dp\,dx\,e^{i\SP{p,x}}\,\tilde{\mu}(\alpha_{\te
      p}(a),\alpha_x(b))
    \,,\qquad
    a,b\in\A,
\end{align}
and the module structure $\tilde{\mu}$ according to
\begin{align}\label{eq:DeformedModuleMap}
    \tilde{\mu}_\te(a,\psi)
    :=
    a_\te \psi
    :=
    (2\pi)^{-n}\int_{\mathbb{R}^{2n}}
    dp\,dx\,e^{i\SP{p,x}}\,\tilde{\mu}(\alpha_{\te
      p}(a),\beta_x(\psi))
    \,,\qquad
    a\in\A,\;\psi\in\E\,,
\end{align}
with the same deformation parameter $\te\in\Rl^{n\times n}$. We will
write $\A_\te$ for the algebra given by the linear space $\A$ and the
product $\times_\te$.
\begin{theorem}\label{theorem:AssociativityOfDeformedMu}
    Let $\A$ be a sequentially complete locally convex algebra with
    separately continuous product $\mu$, and let $\E$ be a
    sequentially complete locally convex left $\A$-module with
    separately continuous module structure $\tilde{\mu}$. Let $\alpha$
    be a smooth polynomially bounded $\Rl^n$-action by automorphisms
    on $\A$, and $\beta$ a smooth polynomially bounded $\Rl^n$-action
    on $\E$ such that
    \begin{align}\label{eq:CovariantModule}
        \beta_x(\tilde{\mu}(a,\psi))
        &=
        \tilde{\mu}(\alpha_x(a),\,\beta_x(\psi))\,,\qquad
        a\in\A,\,\psi\in\E,\,x\in\Rl^n\,.
    \end{align}
    \begin{theoremlist}
    \item In this case $(\E,\tilde{\mu}_\te)$ is a left
        $\A_\te$-module, i.e.
        \begin{align}\label{eq:DeformedModule}
            (a\times_\te b)_\te\psi
            &=
            a_\te b_\te \psi
            \,,\qquad a,b\in\A,\psi\in\E\,.
        \end{align}
    \item If the product $\mu$ in $\A$ is associative, then so is the
        deformed product $\mu_\te$ \eqref{eq:DeformedProduct}.
    \end{theoremlist}
\end{theorem}
\begin{proof}
    Let $a,b\in\A$, $\psi\in\E$, and $\te\in\Rl^{n\times n}$. By
    applying repeatedly the arguments from Lemma~\ref{lemma:FvwGvw},
    one sees that
    \begin{align*}
        \Rl^{4n}\ni(p,x,p',x')
        \longmapsto
        \tilde{\mu}\left(
            \alpha_{\te p'+\te p}(a),
            \tilde{\mu}(\alpha_{\te p + x'}(b),\beta_x(\psi))
        \right)
    \end{align*}
    is a symbol in $\underline{\Symbol}(\Rl^{4n},\E)$. Its oscillatory
    integral can be written with the help of Fubini's theorem, the
    module property
    $\tilde{\mu}(\mu(a,b),\psi)=\tilde{\mu}(a,\tilde{\mu}(b,\psi))$,
    the separate continuity of $\tilde{\mu}$, and
    \eqref{eq:Covariance} as
    \begin{align*}
        &(2\pi)^{-2n}\int_{\mathbb{R}^{4n}}
        dp\,dx\,dp'\,dx'\,e^{i\SP{p,x}+i\SP{p',x'}}\,
        \tilde{\mu}\left(
            \alpha_{\te p'+\te p}(a),
            \tilde{\mu}(\alpha_{\te p + x'}(b),\beta_x (\psi))
        \right)
        \\
        &\quad=
        (2\pi)^{-2n}\int_{\mathbb{R}^{2n}}
        dp\,dx\,e^{i\SP{p,x}}\,
        \left(
            \int_{\mathbb{R}^{2n}}
            dp'\,dx'\,e^{i\SP{p',x'}}\,
            \tilde{\mu}\left(
                \alpha_{\te p'+\te p}(a),
                \tilde{\mu}(\alpha_{\te p + x'}(b),
                \beta_x(\psi))
            \right)
        \right)
        \\
        &\quad=
        (2\pi)^{-2n}\int_{\mathbb{R}^{2n}}
        dp\,dx\,e^{i\SP{p,x}}\,
        \left(
            \int_{\mathbb{R}^{2n}}
            dp'\,dx'\,e^{i\SP{p',x'}}\,
            \tilde{\mu}\left(
                \mu(\alpha_{\te p'+\te p}(a),
                \alpha_{\te p + x'}(b)),
                \beta_x(\psi)
            \right)
        \right)
        \\
        &\quad=
        (2\pi)^{-2n}\int_{\mathbb{R}^{2n}}
        dp\,dx\,e^{i\SP{p,x}}\,\tilde{\mu}
        \left(
            \alpha_{\te p} \int_{\mathbb{R}^{2n}}
            dp'\,dx'\,e^{i\SP{p',x'}}\, 
            \mu(\alpha_{\te p'}(a),
            \alpha_{x'}(b)),
            \beta_x(\psi)
        \right)
        \\
        &\quad=
        (2\pi)^{-n}\int_{\mathbb{R}^{2n}}
        dp\,dx\,e^{i\SP{p,x}}\,
        \tilde{\mu}\left(
            \alpha_{\te p}(\mu_\te(a,b)),\beta_x(\psi)
        \right)
        \\
        &\quad=
        \tilde{\mu}_\te(\mu_\te(a,b),\psi)\,.
    \end{align*}
    On the other hand, we can use
    Lemma~\ref{lemma:affine-substitutions} to carry out the
    substitutions $p'\to p'-p$ and $x\to x+x'$ in the first
    oscillatory integral. This gives
    \begin{align*}
        &\tilde{\mu}_\te(\mu_\te(a,b),\psi) \\
        &\quad=
        (2\pi)^{-2n}\int_{\mathbb{R}^{4n}}
        dp\,dx\,dp'\,dx'\,e^{i\SP{p,x}+i\SP{p',x'}}\,
        \tilde{\mu}\left(
            \alpha_{\te p'+\te p}(a),
            \tilde{\mu}(\alpha_{\te p + x'}(b),
            \beta_x(\psi))
        \right)
        \\
        &\quad=
        (2\pi)^{-2n}\int_{\mathbb{R}^{4n}}
        dp\,dx\,dp'\,dx'\,e^{i\SP{p,x+x'}+i\SP{p'-p,x'}}\,
        \tilde{\mu}\left(
            \alpha_{\te p'}(a),
            \tilde{\mu}(\alpha_{\te p + x'}(b),
            \beta_{x+x'}(\psi))
        \right)
        \,.
    \end{align*}
    Notice that the exponential appearing here equals
    $e^{i\SP{p,x}+i\SP{p',x'}}$ because the term $\SP{p,x'}$ drops
    out. So we can again use the covariance and separate continuity of
    $\tilde{\mu}$, and split the double oscillatory integral into two
    single oscillatory integrals, to arrive at the desired result,
    \begin{align*}
        &\tilde{\mu}_\te(\mu_\te(a,b),\psi) \\
        &\quad=
        (2\pi)^{-2n}\int_{\mathbb{R}^{4n}}
        dp\,dx\,dp'\,dx'\,e^{i\SP{p,x}+i\SP{p',x'}}\,
        \tilde{\mu}\left(
            \alpha_{\te p'}(a),
            \beta_{x'}\left(
                \tilde{\mu}(\alpha_{\te p}(b), 
                \beta_{x}(\psi))
            \right)
        \right)
        \\
        &\quad=
        (2\pi)^{-2n}\int_{\mathbb{R}^{2n}}
        dp'\,dx'\,e^{i\SP{p',x'}}\,
        \tilde{\mu}\left(
            \alpha_{\te p'}(a),
            \beta_{x'}\left(
                \int_{\mathbb{R}^{2n}}
                dp\,dx\,e^{i\SP{p,x}}
                \tilde{\mu}(\alpha_{\te p}(b),\beta_{x}(\psi))
            \right)
        \right)
        \\
        &\quad=
        (2\pi)^{-n}\int_{\mathbb{R}^{2n}}
        dp'\,dx'\,e^{i\SP{p',x'}}\, 
        \tilde{\mu}\left(
            \alpha_{\te p'}(a),
            \beta_{x'}\left(
                \tilde{\mu}_\te(b,\psi)
            \right)
        \right)
        \\
        &\quad=
        \tilde{\mu}_\te(a,\tilde{\mu}_\te(b,\psi))
        \,.
    \end{align*}
    Rewriting $\mu_\te$ and $\tilde{\mu}_\te$ according to
    \eqref{eq:DeformedProduct} and \eqref{eq:DeformedModuleMap} yields
    \eqref{eq:DeformedModule}. The second part follows by considering
    the special case $\E=\A$, $\tilde{\mu}=\mu$, $\beta=\alpha$.
\end{proof}

For isometric actions on Fr\'echet algebras, the associativity of the
deformed product is known from Rieffel's work
\cite{rieffel:1993a}. The deformation of the module structure can also
be viewed as an alternative deformation of an algebra $\A$ represented
on $\E$, which changes the elements $a\in\A$ according to $a \mapsto
a_\te$, but keeps the product unchanged. This deformation has been
introduced under the name of {\em warped convolution} in the context
of $C^*$-algebras \cite{BuchholzSummers:2008,
  BuchholzLechnerSummers:2010}, it is equivalent to the deformation of
the product according to \eqref{eq:DeformedModule}.

Sticking to the setting of an algebra $\A$ with product $\mu$ and a
left $\A$-module with module structure $\tilde{\mu}$, and actions
$\alpha,\beta$ satisfying the assumptions of
Theorem~\ref{theorem:AssociativityOfDeformedMu}, we next show how
identities and star involutions behave under the deformation.
\begin{proposition}\label{Proposition:IdentityAndStarVsDeformation}
    Let $\A$ be a locally convex sequentially complete algebra with
    separately continuous associative product, and $\alpha\colon
    \Rl^n\times \A\longrightarrow \A$ a smooth, polynomially bounded
    $\Rl^n$-action by automorphisms.
    \begin{propositionlist}
    \item \label{item:IdentityUndeformed} If $\A$ has an identity $1$,
        this is also an identity for the deformed product
        \eqref{eq:DeformedProduct}.
    \item \label{item:StarInvolutionUndeformed} If $\A$ is a
        ${}^*$-algebra with continuous $^*$-involution and $\te$ is
        skew-symmetric with respect to the inner product used in the
        oscillatory integrals defining the deformed product, then
        $a\mapsto a^*$ is also a star involution for the deformed
        product,
        i.e.,
        \begin{align}\label{eq:InvolutionForDeformedProduct}
            (a \times_\te b)^*
            =
            b^*\times_\te a^*\,,\qquad a,b\in \A\,.
        \end{align}
    \end{propositionlist}
\end{proposition}
\begin{proof}
    The first part is clear: since $\alpha$ acts by automorphisms, we
    have $\alpha_x(1)=1$ for all $x\in\Rl^n$. Hence, by
    Lemma~\ref{lemma:InvariancesOfMu},
    \refitem{item:MuOnInvariantVectors},
    \begin{align*}
        a\times_\te 1 = a1 = a\,,\qquad 1\times_\te a=1a=a
    \end{align*}
    for any $a\in\A$.  For the second part, we note that as the
    involution $a\mapsto a^*$ is antilinear and continuous, we can use
    \eqref{antilinear-map-in-I} and
    Lemma~\ref{lemma:affine-substitutions} to compute for $a,b\in\A$
    \begin{align*}
        (a\times_\te b)^*
        &=
        (2\pi)^{-n}\int_{\mathbb{R}^{2n}}
        dp\,dx\,e^{i\SP{p,x}}\,(\alpha_{-\te p}a\,\alpha_x b)^*
        \\
        &=
        (2\pi)^{-n}\int_{\mathbb{R}^{2n}}
        dp\,dx\,e^{i\SP{p,x}}\,\alpha_x b^*\, \alpha_{-\te p}a^*
        \\
        &=
        (2\pi)^{-n}\int_{\mathbb{R}^{2n}}
        dp\,dx\,e^{i\SP{p,x}}\,\alpha_{-\te^T x} b^*\, \alpha_{p}a^*
        \\
        &=
        b^* \times_{-\te^T} a^*
        \,.
    \end{align*}
    In case $\te$ is skew-symmetric, {\em i.e.}, $\te^T=-\te$,
    \eqref{eq:InvolutionForDeformedProduct} follows.
\end{proof}

Again, these statements are well-known in Rieffel's setting
\cite{rieffel:1993a}. Analogous to the preceding proposition, there
exist two closely related properties in the module deformation setting
of Theorem~\ref{theorem:AssociativityOfDeformedMu}: First, if a vector
$\psi\in\E$ is $\beta$-invariant, then we have $a_\te\psi =
a\psi$. This is again a straightforward consequence of
Lemma~\ref{lemma:InvariancesOfMu},
\refitem{item:MuOnInvariantVectors}.

Second, in the case of a covariant Hilbert space representation of a
${}^*$-algebra $\A$, we have a compatability between the
${}^*$-operation and the deformation similar to Proposition
\ref{Proposition:IdentityAndStarVsDeformation}
\refitem{item:StarInvolutionUndeformed}. To describe this, consider a
locally convex sequentially complete ${}^*$-algebra $\A$ with a smooth
polynomially bounded $\mathbb{R}^n$-action $\alpha^\A$ by
${}^*$-automorphisms.  Let furthermore $\Hil$ be a Hilbert space
carrying a strongly continuous unitary representation $u$ of $\Rl^n$,
and let $\E\subset\Hil$ denote the subspace of smooth vectors for $u$.
We consider a covariant representation of $\A$, i.e. a
${}^*$-representation $\pi$ of $\A$ by (closable) operators defined on
$\E$ such that $\pi(\alpha^\A_x(a))\psi=u(x)\pi(a)u(x)^{-1}\psi$ for
all $a\in\A$, $x\in\Rl^n$, $\psi\in\E$.

Then we can apply our deformation formula to the module map
$\mu(a,\psi):=\pi(a)\psi$. In case of a skew-symmetric deformation
parameter $\te$, the map $\pi_\te$ defined by the deformed module map,
$\pi_\te(a)\psi:=\mu_\te(a,\psi)$, then gives a ${}^*$-representation
of $\A_\te$ on $\E$, i.e.
\begin{align}
    \pi_\te(a^*)\psi=\pi(a)^*\psi\,,\qquad a\in\A,\,\psi\in\E\,.
\end{align}
In a $C^*$-framework with order 0 actions, this fact has been
established in \cite[Lemma~2.2]{BuchholzLechnerSummers:2010}. Since
the proof is essentially the same in the present situation, we refrain
from giving the details here.

\section{Examples}\label{section:examples}

In this section we present a number of explicit examples of
polynomially bounded $\Rl^n$-actions complying with the conditions in
Definition~\ref{definition:SmoothSymbolAction}. In particular, we show
how target spaces with unbounded orders appear naturally when studying
compactly supported $\Rl^n$-actions.

\subsection{The canonical $\Rl^n$-action on symbol spaces}

The first example is the action studied in
Section~\ref{subsec:FurtherPropertiesSymbols} on the symbol spaces.
\begin{proposition}\label{proposition:PullbackAction}
    Let $V$ be a sequentially complete locally convex space with a
    defining system of seminorms $\mathcal Q$, and let $\bm$, $\brho$
    be an order and a type for $\mathcal Q$, with $\brho\geq0$. Then
    $(\alpha_x F)(y):=F(x+y)$ is a smooth polynomially bounded
    $\Rl^n$-action on $\Symbol^{\bm,\brho}(\Rl^n,V)$ of order
    $\hat{\bm}(\|\cdot\|^{\bm,\brho}_{\qn,\mu}):=|\bm(\qn)-\brho(\qn)|\mu||$
    and type $\hat{\brho}(\|\cdot\|^{\bm,\brho}_{\qn,\mu}):=0$ where
    $\qn\in\mathcal Q,\;\mu\in\Nl_0^n$, as defined in
    Definition~\ref{definition:SmoothSymbolAction}. More precisely,
    for any $\qn\in\mathcal Q$, $\mu\in\Nl_0^n$, there exists
    $C_{\qn,\mu}>0$ such that
    \begin{align}\label{eq:FtoTauFBound}
	\|\alpha(F)\|^{\hat{\bm},0}_{\|\cdot\|^{\bm,\brho}_{\qn,\mu},\nu}
        \leq
        C_{\qn,\mu}
        \|F\|^{\bm,\brho}_{\qn,\mu+\nu}
    \end{align}
    for all $F\in\Symbol^{\bm,\brho}(\Rl^n,V)$, $\nu\in\Nl_0^n$.
\end{proposition}
\begin{proof}
    It has been shown in
    Proposition~\ref{proposition:AffineInvarianceOfSymbols} that
    $\alpha$ is an $\Rl^n$-action on $\Symbol^{\bm,\brho}(\Rl^n,V)$,
    and in Proposition~\ref{proposition:TranslationActSmoothOnSymbols}
    that $\Rl^n\ni x\mapsto \alpha_x(F)
    \in\Symbol^{\bm,\brho}(\Rl^n,V)$ is smooth for each
    $F\in\Symbol^{\bm,\brho}(\Rl^n,V)$ if $\brho\geq0$.  To derive the
    statements about the polynomial bounds, let $\qn\in\mathcal Q$,
    $\mu,\nu\in\Nl_0^n$, and $F\in\Symbol^{\bm,\brho}(\Rl^n,V)$ with
    $\brho\geq0$.  The derivatives $\partial_x^\nu \alpha_x(F) =
    \alpha_x(\partial_x^\nu F)$, see \eqref{eq:DerivativesTaufSymbol},
    satisfy according to Lemma~\ref{lemma:TranslationsActOnSymbols}
    \begin{align*}
	\|\partial_x^\nu \alpha_x(F)\|^{\bm,\brho}_{\qn,\mu}
        &=
	\|\alpha_x(\partial_x^\nu F)\|^{\bm,\brho}_{\qn,\mu}
        \leq
        c(x) \,\|\partial_x^\nu F\|^{\bm,\brho}_{\qn,\mu}
    \end{align*}
    with a positive scalar symbol
    $c\in\Symbol^{|\bm(\qn)-\brho(\qn)|\mu||,1}(\Rl^n,\Rl)$.
    Furthermore, we have by application of
    Proposition~\ref{proposition:PartialDerivativesOnSymbols} and
    Proposition~\ref{proposition:SymbolFirstProperties},
    \refitem{item:SymbolsInSymbols}
    \begin{align*}
        \|\partial_x^\nu F\|^{\bm,\brho}_{\qn,\mu}
        =
        \|F\|^{\bm+\brho|\nu|,\brho}_{\qn,\mu+\nu}
        \leq
        \|F\|^{\bm,\brho}_{\qn,\mu+\nu}
    \end{align*}
    since $\brho\geq0$. With these two bounds, we arrive at
    \begin{align*}
	\|\alpha(F)\|^{\hat{\bm},0}_{\|\cdot\|^{\bm,\brho}_{\qn,\mu},\nu}
        &=
        \sup_{x\in\Rl^n}
        \frac{
          \|\partial_x^\nu \alpha_x(F)\|^{\bm,\brho}_{\qn,\mu}
        }
        {(1+\|x\|^2)^{\frac{1}{2}\hat{\bm}}}
        \leq
        \sup_{x\in\Rl^n}
        \frac{c(x)}{(1+\|x\|^2)^{\frac{1}{2}|\bm(\qn)-\brho(\qn)|\mu||}
        } \,
        \|F\|^{\bm,\brho}_{\qn,\mu+\nu}
        \,,
    \end{align*}
    which establishes \eqref{eq:FtoTauFBound} with the constant
    $C_{\qn,\mu}:=\|c\|^{|\bm(\qn)-\brho(\qn)|\mu||, 1 } _ {0 }
    <\infty$. Hence $\alpha(F) \in
    \Symbol^{\hat{\bm},0}(\Rl^n,\Symbol^{\bm,\brho}(\Rl^n,V))$ and
    $F\mapsto \alpha(F)$ is continuous, as required in
    Definition~\ref{definition:SmoothSymbolAction}.
\end{proof}

By the same arguments, one checks that $(\alpha_x(F))(y)=F(x+y)$ gives
a smooth polynomially bounded $\Rl^n$-action on the vector valued
Schwartz space $\Schwartz(\Rl^n,V)$
(Definition~\ref{definition:VectorValuedSchwartzSpace}), topologized
by the seminorms $\qn_{m,\mu}(\cdot):=\|\cdot \|^{-m,0}_{\qn,\mu}$,
with $\qn\in\mathcal Q$, $m\in\Nl_0$, $\mu\in\Nl_0^n$
\eqref{eq:SchwartzSeminormqmmu}. Here the order is
$\hat{\bm}(\qn_{m,\mu})=m$, and again $\hat{\brho}(\qn_{m,\mu})=0$.
\begin{remark}
    \label{remark:StuffIsReallyUnbounded}%
    Remarkably, in both examples the orders $\hat{\order{m}}$ of the
    induced action is necessarily unbounded, even if we started with
    symbols of bounded order. Only in the particular case where
    $\order{m}(\qn) = 0 = \type{\rho}(\qn)$ we get again a bounded
    order $\hat{\order{m}} = 0$. This was the particular case of an
    \emph{isometric} action as discussed by Rieffel in
    \cite{rieffel:1993a}.
\end{remark}

If $V = \mathcal{A}$ is an algebra with continuous product, we can use
the action $\alpha$ to deform the pointwise product in
$\Symbol^{\bm,\brho}(\Rl^n, \mathcal{A})$ as in
\eqref{eq:DeformedProduct}, with some deformation parameter
$\te\in\Rl^{n\times n}$. As the evaluation maps
$\Symbol^{\bm\brho}(\Rl^n, \mathcal{A})\ni F\mapsto F(x)\in
\mathcal{A}$ are continuous, we have the explicit formula
\begin{align}\label{eq:MoyalProduct}
    (F\times_\te^\alpha G)(x)
    =
    (2\pi)^{-n}\int_{\mathbb{R}^{2n}}
    dp\,dy\,e^{i\SP{p,y}}\,F(y+\te p)G(y+x)
\end{align}
as a $\mathcal{A}$-valued oscillatory integral.

In addition to $\alpha$, we have on the Schwartz space also smooth
polynomially bounded $\Rl^n$-actions of the form
\begin{align}\label{eq:ActionByMultiplicativePhase}
	(\beta_xF)(y) := e^{i(x,y)} F(y)\,,
\end{align}
where $(\cdot,\cdot)$ denotes a bilinear form on $\Rl^n$. Considered
on a symbol space $\Symbol^{\bm,\brho}(\Rl^n,V)$ of fixed order, these
actions are not smooth, but on $\Schwartz(\Rl^n,V)$, they comply with
Definition~\ref{definition:SmoothSymbolAction}, with order
$\hat{\bm}(\qn_{m,\mu})=|\mu|$ and type
$\hat{\brho}(\qn_{m,\mu})=0$. Taking $V = \mathcal{A}$ to be an
algebra, the deformation of the pointwise product in
$\Ss(\Rl^n,\mathcal{A})$ with the action
\eqref{eq:ActionByMultiplicativePhase} is however almost trivial; one
has $(F\times^\beta_\te G)(x)=e^{i(\te Ax,x)}F(x)G(x)$ with a matrix
$A$ depending on the choice of inner product on $\Rl^n$.

We now explain how some deformations of algebras of scalar-valued
functions discussed in the literature fit into our framework. The
first and best-known example is clearly the scalar Schwartz space
$\Schwartz(\Rl^n,\Cl)$ with pointwise product. Here the deformed
product \eqref{eq:MoyalProduct} even exists pointwise as a Riemann
integral because of the decay of the integrand. It is usually referred
to as Moyal product or twisted product, see e.g.
\cite{Gracia-BondiaVarilly:1988}.

Another version of this is to consider $\Schwartz(\Rl^n,\Cl)$ as an
algebra with convolution $(f*g)(x)=\int_{\mathbb{R}^n} dy\,f(y)g(x-y)$
as product, and the multiplicative action
\eqref{eq:ActionByMultiplicativePhase}. Taking all inner products of
$\Rl^n$ as the usual Euclidean inner product, and $\te$ to be
antisymmetric, we find
\begin{align}\label{eq:TwistedConvolution}
    (f *_\te^\beta g)(x)
    =
    \int_{\mathbb{R}^n} dy\,e^{i\SP{x,\te y}}\,f(y)g(x-y)\,.
\end{align}
This deformed product is usually referred to as a {\em twisted
  convolution} according to \cite{Gracia-BondiaVarilly:1988}. Since
the Fourier transform $\F\colon \Schwartz\longrightarrow \Schwartz$
intertwines the pointwise product and convolution as well as the
actions $\alpha$ and $\beta$, the twisted convolution product is
equivalent to the product $\times_\te^\alpha$.

The deformed products $\times_\te^\alpha$ and $*_\te^\beta$ can be
extended from $\Ss(\Rl^n,\Cl)$ to spaces of distributions
\cite{Gracia-BondiaVarilly:1988,
  Dubois-VioletteKrieglMaedaMichor:2001}. In particular, in
\cite{Dubois-VioletteKrieglMaedaMichor:2001} it is explained how
$\times^\alpha_\te$ can be defined on the distribution space
$\mathcal{O}_M'(\Rl^n)$, the dual of the space $\mathcal{O}_M(\Rl^n)$
of tempered smooth functions. Recall that
$\mathcal{O}_M(\mathbb{R}^n)$ is defined as the set of all smooth
$f\colon \Rl^n \longrightarrow \Cl$ such that for each multiindex
$\mu$, there exists some $k\in\Zl$ such that
$x\mapsto(1+\|x\|^2)^k|(\partial^\mu_xf)(x)|$ is bounded. In our
notation, that is $\mathcal{O}_M =
\cup_{m,\rho}\Symbol^{m,\rho}(\Rl^n,\Cl)$, where the union runs over
all orders and types $m, \rho\in\Rl$. Similarly, the classical
function space $\mathcal{O}_C(\mathbb{R}^n)$ is in our notation
$\mathcal{O}_C(\Rl^n) = \cup_{m}\Symbol^{m,0}(\Rl^n,\Cl) =
\Symbol^{\infty,0}(\Rl^n, \Cl)$. Clearly $\Schwartz(\mathbb{R}^n)
\subset\mathcal{O}_C(\mathbb{R}^n) \subset \mathcal{O}_M(\mathbb{R}^n)
\subset \Schwartz'(\mathbb{R}^n)$ and $\Schwartz(\mathbb{R}^n) \subset
\mathcal{O}_M'(\mathbb{R}^n) \subset \mathcal{O}_C'(\mathbb{R}^n)
\subset \Schwartz'(\mathbb{R}^n)$, and the Fourier transform $\F$ on
$\Schwartz'(\mathbb{R}^n)$ restricts to isomorphisms
$\mathcal{O}_C(\mathbb{R}^n) \longrightarrow
\mathcal{O}_M'(\mathbb{R}^n)$ and $\mathcal{O}_M(\mathbb{R}^n)
\longrightarrow \mathcal{O}_C'(\mathbb{R}^n)$.  Since
$\mathcal{O}_C(\mathbb{R}^n)$ contains only symbols of type $\rho=0$,
we can form the deformed products $f\times_\te^\alpha g$
\eqref{eq:MoyalProduct} for $f,g\in\mathcal{O}_C(\mathbb{R}^n)$. (For
$f,g\in\mathcal{O}_M(\mathbb{R}^n)$, this is not possible because we
need restrictions on the type for $\alpha$ to be smooth and the
oscillatory integrals to exist.)  Making use of the Fourier transform
$\F\colon \mathcal{O}_C(\mathbb{R}^n) \longrightarrow
\mathcal{O}_M'(\mathbb{R}^n)$, this also gives us a product on
$\mathcal{O}_M'(\mathbb{R}^n)$,
\begin{align}\label{eq:TwistedConvolutionForDistributions}
    T \times S := \F(\F^{-1}T\times_\te^\alpha \F^{-1}S)
    \,.
\end{align}
As $\F$ intertwines the actions $\alpha$ and $\beta$, it is easy to
see that \eqref{eq:TwistedConvolutionForDistributions} coincides with
the ``other twisted convolution'' constructed in
\cite{Dubois-VioletteKrieglMaedaMichor:2001}.

\subsection{Compactly supported $\Rl^n$-actions}

In this subsection we construct and study a different class of smooth
polynomially bounded $\Rl^n$-actions on function spaces. The actions
we are interested in here are given by pullbacks of $\Rl^n$-actions
$\tau$ on $\Rl^n$ which act non-trivially only in a compact set $K$,
i.e.\ satisfy $\tau_x(y)=y$ for all $y\notin K$, $x\in\Rl^n$. We want
to construct $\tau$ in such a way that $\alpha_x^K(f) := f\circ
\tau_x$ is smooth and polynomially bounded in the sense of
Definition~\ref{definition:SmoothSymbolAction}, say on
$\Cinfty(\Rl^n,\Cl)$. For simplicity, we restrict to scalar-valued
functions here. It is clear that we cannot hope for an isometric
action as required by Rieffel's original construction as soon as we
leave the $C^0$-framework: controlling also derivatives as needed in
the $C^\infty$-topology will necessarily lead to a non-isometric
action.

To check what kind of condition on $\tau$ is necessary for this,
consider a function $f_j$ which coincides with a coordinate $x\mapsto
x_j$, $j=1,\ldots,n$, on $K$. If $\alpha^K(f_j) \in
\Symbol^{\hat{\bm},0}(\Rl^n,\Cinfty(\Rl^n,\Cl))$ for some appropriate
$\hat{\order{m}}$, then the supremum
\begin{align}
    \|
    \alpha^K(f_j)
    \|^{\hat{\bm}(\halbnorm{p}_{K,l}),0}_{\halbnorm{p}_{K,l},\mu}
    &=
    \sup_{x\in\Rl^n}\frac{\halbnorm{p}_{K,l}(\partial_x^\mu \alpha^K_x
      f_j)}{(1+\|x\|^2)^{\frac{1}{2}\hat{\bm}(\halbnorm{p}_{K,l})}}
    =
    \sup_{x\in\Rl^n\atop{y\in K,
        |\nu|\leq l}}\frac{|\partial_x^\mu\partial_y^\nu
      \tau_x(y)_j|}{(1+\|x\|^2)^{\frac{1}{2}\hat{\bm}(\halbnorm{p}_{K,l})}}
\end{align}
must be finite. Hence we need bounds of the form
$|\partial_y^\nu \partial_x^\mu \tau_x(y)_j| \leq c_{\mu
  l}(1+\|x\|^2)^{\frac{1}{2}b_l}$ for all $\nu\in\Nl_0^n$ with
$|\nu|\leq l$. Taking into account that $\tau$ satisfies $\tau_x(K)=K$
for all $x\in\Rl^n$ by its support property, that $K$ is compact, and
that $\tau$ is an $\Rl^n$-action, it follows that we can choose
$b_0=0$. These observations motivate the following definition.
\begin{definition}\label{definition:ActionOfCompactSupport}
    Let $K\subset\Rl^n$ be compact, and
    $\boldsymbol{b}:=\{b_l\}_{l\in\Nl_0}\subset\Rl_+$ a sequence
    starting with $b_0=0$. A smooth $\Rl^n$-action with support in $K$
    and order $\boldsymbol{b}$ is a smooth function
    $\tau\colon\Rl^n\times\Rl^n\longrightarrow\Rl^n$ such that
    \begin{definitionlist}
    \item $\tau_x(\tau_{x'}(y))=\tau_{x+x'}(y)$ for all
        $x,x',y\in\Rl^n$.
    \item $\tau_x(y)=y$ for all $x\in\Rl^n$ and all
        $y\in\Rl^n\backslash K$.
    \item For each $\mu\in\Nl_0^n$, $l\in\Nl_0$, there exists a
        constant $c_{l\mu}>0$ such that
        \begin{align}\label{eq:BoundOnDMuNuTau}
            \sup_{y\in K,|\nu|\leq l\atop j\in\{1,\ldots,n\}}
            |\partial_y^\nu \partial_x^\mu \tau_x(y)_j|
            &\leq
            c_{l\mu} (1+\|x\|^2)^{\frac{1}{2}b_l}
        \end{align}
        holds for all $x\in\Rl^n$.
    \end{definitionlist}
\end{definition}

We will later construct explicit examples of actions satisfying these
assumptions. Postponing this construction for a moment, we first show
that such $\tau$ do indeed define smooth polynomially bounded
$\Rl^n$-actions by pullback. To begin with, we note the following
elementary lemma.
\begin{lemma}\label{lemma:CompactlySupportedAction}
    Let $\tau$ be a smooth $\Rl^n$-action with support in a compact
    set $K\subset\Rl^n$, and order $\boldsymbol{b}$. Then for each
    $\mu\in\Nl_0^n$, $l\in\Nl_0$, there exists a constant $C_{l\mu}>0$
    such that
    \begin{align}\label{eq:BoundOnDMuNuFTau}
        \sup_{y\in K, |\nu|\leq l}
        |\partial_y^\nu\partial_x^\mu f(\tau_x(y))|
        &\leq
        C_{l\mu}\,(1+\|x\|^2)^{\frac{1}{2}(b_1+\cdots+b_l)}
        \cdot
        \halbnorm{p}_{K,l+|\mu|}(f)
    \end{align}
    for all $f\in\Cinfty(\Rl^n,\Cl)$, $x\in\Rl^n$. Here
    $\halbnorm{p}_{K,l+|\mu|}$ denotes the usual $\Cinfty$-seminorms
    \eqref{eq:CinftyNorms} with $\qn=|\cdot|$.
\end{lemma}
\begin{proof}
    We first consider the case $l=0$ without derivatives with respect
    to $y$. By the chain rule, we have
    \begin{align*}
	\partial_x^\mu f(\tau_x(y)) = \sum_{\la\leq\mu} (\partial^\la
        f)(\tau_x(y))\cdot g_\la(x,y)\,,
    \end{align*}
    where the $g_\la(x,y)$ are polynomials in partial derivatives of
    the $\tau_x(y)_j$ with respect to the components of $x$. According
    to \eqref{eq:BoundOnDMuNuTau} with $l=0$ (and $b_0=0$), these
    functions are uniformly bounded in $x\in\Rl^n$ and $y\in
    K$. Furthermore, we have $\tau_x(y)\in K$ for all $x\in\Rl^n$,
    since $y\in K$. Hence \eqref{eq:BoundOnDMuNuFTau} follows by
    straightforward estimate.  We now proceed by induction and assume
    that \eqref{eq:BoundOnDMuNuFTau} holds for some $l\in\Nl_0^n$, and
    all $\mu\in\Nl_0^n$, $f\in\Cinfty(\Rl^n,\Cl)$. Then,
    $j\in\{1,\ldots,n\}$, $|\nu|\leq l$,
    \begin{align*}
        \partial_y^{\nu+e_j}\partial_x^\mu f(\tau_x(y))
        &=
        \partial_y^\nu\partial_x^\mu\sum_{j'=1}^n
        \partial_y^{e_j}\tau_x(y)_{j'}\cdot (\partial^{e_{j'}}f)(\tau_x(y))
        \\
        &=
        \sum_{j'=1}^n\sum_{\nu'\leq\nu\atop\mu'\leq\mu}
        \left(\nu \atop \nu' \right)\left(\mu \atop \mu' \right)
        \left(
            \partial_y^{\nu-\nu'+e_j}\partial_x^{\mu-\mu'}\tau_x(y)_{j'}\right)
        \left(\partial_y^{\nu'}\partial_x^{\mu'}(\partial^{e_{j'}}f)(\tau_x(y))\right)
        .
    \end{align*}
    In this sum, the derivatives of $\tau_x(y)_{j'}$ can be estimated
    directly with \eqref{eq:BoundOnDMuNuTau}, taking into account
    $|\nu-\nu'+e_j|\leq l+1$. For the derivatives of $f$, we can use
    \eqref{eq:BoundOnDMuNuFTau} by our induction hypothesis, since
    $|\nu'|\leq|\nu|\leq l$. This yields constants $C_{j'\nu'\mu'}>0$
    such that
    \begin{align*}
        |\partial_y^{\nu+e_j}\partial_x^\mu f(\tau_x(y))|
        &\leq
        \sum_{j',\nu',\mu'}
        C_{j'\nu'\mu'}
        (1+\|x\|^2)^{\frac{1}{2}b_{l+1}}
        (1+\|x\|^2)^{\frac{1}{2}(b_1+\cdots+b_l)}
        \halbnorm{p}_{K,|\mu'|+|\nu'|}(\partial^{e_{j'}}f)
        \\
        &\leq
        C'_{j\mu\nu}(1+\|x\|^2)^{\frac{1}{2}(b_1+\cdots+b_l+b_{l+1})}
        \cdot
        \halbnorm{p}_{K,l+1+|\mu|}(f)\,.
    \end{align*}
    Since $j$ was arbitrary, \eqref{eq:BoundOnDMuNuFTau} follows by
    induction in $l$.
\end{proof}
\begin{proposition}
    Let $\tau$ be a smooth $\Rl^n$-action on $\Rl^n$, with order
    $\boldsymbol{b}$ and support in some compact set
    $K\subset\Rl^n$. Then its pullback
    $(\alpha^K_xf)(y):=f(\tau_x(y))$ is a smooth polynomially bounded
    $\Rl^n$-action on $\Cinfty(\Rl^n,\Cl)$, on each symbol space
    $\Symbol^{m,\rho}(\Rl^n,\Cl)$, $m,\rho\in\Rl$, and on the Schwartz
    space $\Schwartz(\Rl^n,\Cl)$.
\end{proposition}
\begin{proof}
    Let us first consider $\alpha^K$ acting on
    $\Cinfty(\Rl^n,\Cl)$. It is clear that it is an $\Rl^n$-action on
    this space because $\tau$ is an action and smooth. To estimate its
    seminorms, let $J\subset\Rl^n$ be compact, $\mu,\nu\in\Nl_0^n$,
    and $f\in\Cinfty(\mathbb{R}^n, \mathbb{C})$. Taking into account
    that $\tau$ acts trivially outside $K$, and using the bounds of
    Lemma~\ref{lemma:CompactlySupportedAction}, we find
    \begin{align*}
        &\sup_{x\in\Rl^n \atop y\in J,
          |\nu|\leq l}\frac{|\partial_y^\nu\partial_x^\mu
          f(\tau_x(y))|}{(1+\|x\|^2)^{\frac{1}{2}(b_1+ \cdots +b_l)}} \\
        &\quad\leq
        \sup_{x\in\Rl^n \atop y\in J\backslash K,
          |\nu|\leq l}\frac{|\partial_y^\nu\partial_x^\mu
          f(y)|}{(1+\|x\|^2)^{\frac{1}{2}(b_1+\cdots+b_l)}}
        +
        \sup_{x\in\Rl^n \atop y\in J\cap K,
          |\nu|\leq l}\frac{|\partial_y^\nu\partial_x^\mu
          f(\tau_x(y))|}{(1+\|x\|^2)^{\frac{1}{2}(b_1+\cdots+b_l)}}
        \\
        &\quad\leq
        \delta_{\mu,0}\sup_{x\in\Rl^n \atop y\notin
          J\backslash K, |\nu|\leq l}\frac{ |\partial_y^\nu
          f(y)|}{(1+\|x\|^2)^{\frac{1}{2}(b_1+\cdots+b_l)}}
        +
        C_{l\mu}\halbnorm{p}_{K,l+|\mu|}(f)
        \\
        &\quad=
        \delta_{\mu,0}\,\halbnorm{p}_{J\backslash
          K,l}(f) + 	C_{l\mu}\,\halbnorm{p}_{K,l+|\mu|}(f)
        \,,
    \end{align*}
    where the last step relies on $b_l\geq0$. This estimate shows in
    particular that for any $x\in\Rl^n$, the map
    $\alpha^K_x\colon\Cinfty(\mathbb{R}^n, \mathbb{C}) \longrightarrow
    \Cinfty(\mathbb{R}^n, \mathbb{C})$ is continuous, a fact which is
    known to be true for \emph{any} diffeomorphism. Moreover, once we
    have checked that $x\mapsto\alpha^K_x f$ is smooth in the topology
    of $\Cinfty(\mathbb{R}^{n}, \mathbb{C})$, the estimate also shows
    that $\alpha^K(f)$ is a symbol in
    $\Symbol^{\hat{\bm},0}(\Rl^n,\Cinfty(\mathbb{R}^{n},
    \mathbb{C}))$, of order $\hat{\bm}(\halbnorm{p}_{J,l}):=b_1+
    \cdots +b_l$, and that $f\mapsto \alpha^K(f)$ is continuous. So in
    order to verify all conditions of
    Definition~\ref{definition:SmoothSymbolAction}, it only remains to
    establish the smoothness of $\alpha^K$: but this is true for
    arbitrary smooth Lie group actions on smooth manifolds.  We now
    consider $\alpha^K$ on the symbol and Schwartz subspaces of
    $\Cinfty(\mathbb{R}^n, \mathbb{C})$. Since $\tau$ acts
    non-trivially only in a compact set, it is clear that these
    subspaces are invariant under $\alpha^K$, and $\alpha^K$ restricts
    to $\Rl^n$-actions on all these spaces. Concerning smoothness,
    note that the functions $\alpha_x^K(f)-f$ and
    $\eps^{-1}(\alpha_{\eps
      e_j}^K(f)-f)-\partial_{t}\alpha^K_{te_j}(f)|_{t=0}$ have compact
    support in $K$ for all $f\in\Symbol^{m,\rho}$, $\eps>0$, $j\in\{1,
    \ldots, n\}$, $x\in\Rl^n$. So their symbol seminorms
    $\|\cdot\|^{m,\rho}_\nu$ can be estimated against some
    $\halbnorm{p}_{K,l}(\cdot)$. But the latter seminorms converge to
    zero for $x\to0$ respectively $\eps\to0$, by the preceding results
    about $\alpha^K$ on $\Cinfty(\mathbb{R}^n, \mathbb{C})$. Thus we
    conclude that $\alpha^K$ is also smooth on the symbol spaces
    $\Symbol^{m,\rho}(\mathbb{R}^n, \mathbb{C})$, $m,\rho\in\Rl$, and
    the Schwartz space $\Schwartz(\mathbb{R}^n, \mathbb{C})$. The
    symbol property of $\alpha^K(f)$, and the continuity of
    $f\mapsto\alpha^K(f)$ for these spaces can now be estimated as
    before, by splitting $f = f_0+f_1\in\Symbol^{m,\rho}(\mathbb{R}^n,
    \mathbb{C})$ into a compactly supported symbol $f_0$, and a symbol
    $f_1$ with support disjoint from $K$ which is fixed by $\alpha^K$.
\end{proof}

We now turn to the construction of examples of smooth compactly
supported $\Rl^n$-actions, and consider the one-dimensional case $n=1$
first. As a starting point, we use the same idea as in
\cite{heller.neumaier.waldmann:2007a} and take a diffeomorphism
$\gamma\colon(-1,1)\longrightarrow\Rl$ to define
\begin{align}\label{def:tux}
    \tau_x(y)
    :=
    \left\{
        \begin{array}{rcl}
            \gamma^{-1}(\gamma(y)+x) &;& |y|<1\\
            y &;& |y|\geq 1
        \end{array}
    \right.
    \,,\qquad x\in\Rl\,.
\end{align}
It is clear that $\tau$ is an $\Rl$-action, {\em i.e.},
$\tau_x(\tau_{x'}(y))=\tau_{x+x'}(y)$ for all $x,x',y\in\Rl$, and
$\tau$ acts non-trivially only inside the interval $K:=[-1,1]$. But we
have to choose $\gamma$ in such a way that also the smoothness and
boundedness assumptions of
Definition~\ref{definition:ActionOfCompactSupport} are satisfied. For
the discussion of these two properties, it is instructive to view
$\tau$ as the flow of an autonomous ordinary differential equation
$d\phi/dx=L(\phi(x))$ with initial condition $\phi(0)=y$.
Differentiation of $\phi(x)=\tau_x(y)$ \eqref{def:tux} with respect to
$x$ then shows that
\begin{align}\label{def:L}
    L(x)
    &=
    \left\{
        \begin{array}{rcl}
            \frac{1}{\gamma'(x)} &;& |x|<1\\
            0 &;& |x|\geq 1
        \end{array}
    \right.
    \,.
\end{align}
It is a well-known fact that the solutions $x\mapsto
\phi(x)=\tau_x(y)$ will depend smoothly on $x$ and the initial
condition $y$ if $L$ is smooth. Thus smoothness of $\tau$ is
guaranteed if $\gamma'(x)$ diverges fast enough as $x\to\pm1$, such
that \eqref{def:L} is smooth. On the other hand, the bounds on
$\partial_x^k\partial_y^l \tau_x(y)$ that can be obtained by
exploiting that $\tau$ is the flow of a differential equation with
compactly supported $L$ are only of \emph{exponential}
type. Therefore, we show in the following lemma that by a careful
adjustment of the diffeomorphism $\gamma$, one can achieve
\emph{polynomial} bounds on $\partial_x^k\partial_y^l \tau_x(y)$.

\begin{lemma}\label{lemma:1DAction}
    There exist smooth $\Rl$-actions on $\Rl$ with support in $[-1,1]$
    and order $b_l=2l+1$, which act transitively on $(-1,1)$.
\end{lemma}
\begin{proof}
    The action will be constructed in the form \eqref{def:tux} with
    appropriately chosen $\gamma$. It is already clear from
    \eqref{def:tux} that $\tau$ is an action with support in $[-1,1]$,
    acting transitively on $(-1,1)$. To verify the crucial bounds
    \eqref{eq:BoundOnDMuNuTau},
    \begin{align}\label{eq:BoundsOn1DAction}
	\sup_{|y|\leq 1}|\partial_x^k\partial_y^l \tau_x(y)|
        \leq
        c_{lk}(1+x^2)^{\frac{1}{2}(2l+1)}
        \,,\qquad
        x\in\Rl\,,
    \end{align}
    we first derive a formula for the derivatives of
    $\tau_x(y)=\gamma^{-1}(\gamma(y)+x)$, $|y|<1$, for generic
    diffeomorphisms $\gamma$. This is done with the help of two
    differentiation identities, the first of which states that
    multiple derivatives of $\gamma^{-1}$ have the form
    \begin{align}
        \partial_y^l\gamma^{-1}(y)
        &=
        \sum_n c_n\, \frac{\gamma'(\gamma^{-1}(y))^{n_1}\cdots
          \gamma^{(l)}(\gamma^{-1}(y))^{n_l}}{\gamma'(\gamma^{-1}(y))^{2l-1}}
        \,,
    \end{align}
    where the sum runs over finitely many terms with numerical
    coefficients $c_n$.  In the above formula, the powers $n_j$
    satisfy $n_1+\cdots +n_l=l-1$ in each term, a fact that can easily be
    proven by induction in $l$.

    The second identity is an iterated chain rule for smooth functions
    $f,g\colon \Rl\longrightarrow \Rl$,
    \begin{align}\label{chainrule}
        \partial_y^l(f(g(y))
        &=
        \sum_{r=1}^l 
        c_r'\,g'(y)^{s_1}\cdots g^{(l)}(y)^{s_l}\cdot f^{(r)}(g(y))
        \,,
    \end{align}
    where the $c_r'$ are numerical coefficients and the powers $s_j$
    satisfy $s_1 + \cdots + s_l = r$ in each term. Again, the proof by
    induction is straightforward.

    Application of these two differentiation rules to
    $\tau_x(y)=\gamma^{-1}(\gamma(y)+x)$, $|y|<1$, yields
    \begin{align}
        \partial_x^k  \partial_y^l \tau_x(y)
        &=
        \partial_x^k
        \sum_{r=1}^l c_r'  \gamma'(y)^{s_1}\cdots \gamma^{(l)}(y)^{s_l}\cdot
        (\gamma^{-1})^{(r)}(\gamma(y)+x)
        \nonumber
        \\
        &=
        \sum_{r=1}^l  c_r' \gamma'(y)^{s_1}\cdots \gamma^{(l)}(y)^{s_l}\cdot
        (\gamma^{-1})^{(r+k)}(\gamma(y)+x)
        \nonumber
        \\
        &=
        \sum_{r=1}^l  c_r' \sum_n c_n
        \frac{\gamma'(y)^{s_1}\cdots
          \gamma^{(l)}(y)^{s_l}\cdot\gamma'(\tau_x(y))^{n_1}\cdots
          \gamma^{(r+k)}(\tau_x(y))^{n_{r+k}} }
        {\gamma'(\tau_x(y))^{2(r+k)-1}}\,.
        \label{monster}
    \end{align}
    To obtain useful bounds on this expression, we have to estimate
    the higher derivatives $|\gamma^{(m)}(y)|$ for $m=1, \ldots, l$,
    and $|\gamma^{(j)}(\tau_x(y))|$ for $j=1, \ldots, r+k$, in terms
    of the first derivatives $|\gamma'(\tau_x(y))|$. Simple estimates
    of this form do apparently not exist for generic $\gamma$. We
    therefore choose $\gamma$ of a special form, which will allow for
    convenient computations.

    So let $\gamma$ be antisymmetric, {\em i.e.},
    $\gamma(-y)=-\gamma(y)$, and choose it to be equal to
    $e(y):=\exp\frac{1}{1-y}$ for $y\geq\frac{1}{2}$. Note that this
    choice already implies that $\tau\colon \Rl^2\longrightarrow \Rl$
    is smooth, since all derivatives of
    $L(y):=1/\gamma'(y)=(1-y)^2e^{-1/(1-y)}$ converge to zero for
    $y\to1$. Moreover, a short calculation shows that the tangent of
    $\gamma$ in $y=\frac{1}{2}$ has its zero at $y=\frac{1}{4}$ and
    consequently, we can choose $\gamma$ on
    $[-\frac{1}{2},\frac{1}{2}]$ in such a way that $\gamma'(y)$
    increases monotonically as $|y|$ increases. In particular, we then
    have the lower bound $\gamma'(y)\geq\gamma'(0)>0$, $|y|<1$.

    We now turn to estimating the derivatives $|\gamma^{(j)}(y)|$ for
    a diffeomorphism $\gamma$ with the specified properties. Since
    $\gamma'$ is bounded from below and $\gamma^{(j)}$ is continuous,
    there exist constants $C_j<\infty$ such that
    $|\gamma^{(j)}(y)/\gamma'(y)|\leq C_j$ for all
    $y\in[-\frac{1}{2},\frac{1}{2}]$. For $y>\frac{1}{2}$, we can use
    the explicit form $\gamma(y)=\exp\frac{1}{1-y}$, which implies
    that $\gamma^{(j)}(y)$ is the product of $\gamma(y)$ and a
    polynomial of order $2j$ in $\frac{1}{1-y}$. Hence
    $\gamma^{(j)}(y)/\gamma'(y)$ coincides for $y>\frac{1}{2}$ with a
    rational function of $y$ which diverges polynomially as
    $y\to1$. Because of the symmetry properties of $\gamma'$ and
    $\gamma^{(j)}$, the same is true for the region $y<-\frac{1}{2}$
    and the limit $y\to-1$. But as $\gamma'(y)$ has no zeros and
    diverges exponentially for $|y|\to\pm1$, we find for any $\eps>0$
    a constant $C_{j,\eps}$ such that
    \begin{align}\label{bound-to-gamma-1}
        \frac{|\gamma^{(j)}(y)|}{\gamma'(y)}
        \leq
        C_{j,\eps}\,\gamma'(y)^\eps
        \,,\qquad
        y\in(-1,1)\,.
    \end{align}
    Since $\tau_x$ leaves the interval $(-1,1)$ invariant for any
    $x\in\Rl$, we also have
    \begin{align}\label{bound-to-gamma}
        |\gamma^{(j)}(\tau_x(y))|
        \leq
        C_{j,\eps}\,\gamma'(\tau_x(y))^{1+\eps}
        \,,\qquad
        x\in\Rl,\,y\in(-1,1)\,.
    \end{align}
    
    Next we derive a bound on the ratios
    $\frac{\gamma'(y)}{\gamma'(\tau_x(y))}$.  For this it is
    sufficient to consider $y\in[0,1)$ because of the symmetry of
    $\gamma'$, and for fixed $x\in\Rl$, we split this interval at
    \begin{align}
        \xi(x)
        :=
        \tau_{2|x|}(\tfrac{1}{2})
        =
        e^{-1}(e(\tfrac{1}{2})+2|x|)
        \geq
        \tfrac{1}{2}
        \,,
    \end{align}
    and estimate in the two regions $y\in[0,\xi(x)]$ and
    $y\in(\xi(x),1)$ separately. Note that $e^{-1}(y)=1-1/\log y$ and
    $e'(y)=(1-y)^{-2}e(y)$.

    In the inner region $[0,\xi(x))$, the monotonicity of $\gamma'$
    and the explicit form of our diffeomorphism around
    $\xi(x)\geq\frac{1}{2}$ yield
    \begin{align}
        \frac{\gamma'(y)}{\gamma'(\tau_x(y))}
        &\leq
        \frac{\gamma'(\xi(x))}{\gamma'(0)}
        =
        \frac{e(\xi(x))}{\gamma'(0)\,(1-\xi(x))^2}
        \,,\qquad\qquad
        |y|\leq\xi(x)\,,
        \nonumber
        \\
        &=
        \gamma'(0)^{-1}\,
        \left(e(\tfrac{1}{2})+2|x|\right)\cdot\log(e(\tfrac{1}{2})+2|x|)
        \nonumber\\
        &\leq c(1+x^2)^\delta
        \,,
        \label{inner-bound}
    \end{align}
    where the power $\delta>\frac{1}{2}$ can be chosen arbitrarily
    close to $\frac{1}{2}$ (it will be fixed at the end of the proof),
    and the numerical constant $c$ depends on $\delta$.

    For the estimate in the outer region $(\xi(x),1)$, we use the
    inequalities
    \begin{align}\label{ineq1}
        e(y)+x
        \geq
        e(\xi(x))-|x|
        =
        e(\tfrac{1}{2})+|x|
        \geq
        e(\tfrac{1}{2})\,,
        \qquad
        y\in[\xi(x),1)\,,
    \end{align}
    and $e(y)\geq e(\xi(x))>2|x|$, implying $e(y)-|x|\geq
    \frac{1}{2}e(y)>1$ for $y>\xi(x)$. It follows from \eqref{ineq1}
    that $\tau_x(y)=e^{-1}(e(y)+x)$ in this region. The explicit form
    of $\gamma$ and these inequalities lead to a uniform bound on
    $\frac{\gamma'(y)}{\gamma'(\tau_x(y))}$,
    \begin{align}
        \frac{\gamma'(y)}{\gamma'(\tau_x(y))}
        &=
        \frac{1}{(1-y)^2\log(e(y)+x)^2}\frac{e(y)}{e(y)+x}
        \,,\qquad\qquad
        y\in(\xi(x),1)\,,
        \nonumber
        \\
        &\leq
        \frac{1}{(1-y)^2\log(e(y)-|x|)^2}\frac{e(y)}{e(y)-|x|}
        \,.
        \nonumber
        \\
        &\leq
        \frac{2}{(1-y)^2(\frac{1}{1-y}-\log 2)^2}
        \leq
        \frac{2}{(1-\frac{1}{2}\log 2)^2}
        \,.
        \label{outer-bound}
    \end{align}
    After a possible readjustment of the constant $c$ in
    \eqref{inner-bound} we therefore obtain
    \begin{align}\label{gamma'-bound}
        \gamma'(y)
        &\leq
        c(1+x^2)^\delta
        \,
        \gamma'(\tau_x(y))
        \,,\qquad x\in\Rl\,,y\in(-1,1)\,,
    \end{align}
    where $\delta>\frac{1}{2}$ can still be chosen. Combining this
    bound with \eqref{bound-to-gamma-1}, we also have
    \begin{align}\label{bound-2}
        \left|\gamma^{(m)}(y)\right|
        =
        \frac{|\gamma^{(m)}(y)|}{\gamma'(y)}
        \cdot
        \gamma'(y)
        \leq
        C_{m,\eps}'\,(1+x^2)^{\delta(1+\eps)}
        \cdot
        \gamma'(\tau_x(y))^{1+\eps}
        \,.
    \end{align}
    We can now apply \eqref{bound-to-gamma} and \eqref{bound-2} to
    estimate \eqref{monster}. Taking into account that in each term in
    that sum, we have $s_1+\cdots +s_l=r$ and $n_1+\cdots
    +n_{r+k}=r+k-1$, we get after collecting all factors
    \begin{align}\label{bound-sum}
        |\partial_x^k\partial_y^l \tau_x(y)|
        &\leq
        \sum_n \sum_{r=1}^l
        C(1+x^2)^{\delta(1+\eps)r}\,
        \gamma'(\tau_x(y))^{\eps(2r+k-1)-k}
        \,,
    \end{align}
    where $C$ represents all the numerical constants appearing in the
    various bounds, depending on $l,k,m,n$ and some arbitrary
    $\eps>0$, $\delta>\frac{1}{2}$. For $k\geq 1$, the exponent of
    $\gamma'(\tau_x(y))$ is negative for sufficiently small $\eps$,
    and then
    \begin{align*}
        |\partial_x^k\partial_y^l \tau_x(y)|
        &\leq
        \sum_n \sum_{r=1}^l
        C (1+x^2)^{\delta(1+\eps)r}\,\gamma'(0)^{\eps(2r+k-1)-k}
        \leq
        C' (1+x^2)^{\delta(1+\eps)l}
        \,,\qquad
        |y|<1\,.
    \end{align*}
    For the case $k=0$, note that since $\tau_0(y)=y$ and $l\geq1$, we
    have $\partial_y^l \tau_0(y)=0$ and can estimate via
    \begin{align}
        |\partial_y^l \tau_x(y)|
        =
        \left|\int_0^x dx'\, \partial_{x'} \partial_y^l \tau_{x'}(y) \right|
        \leq
        |x|\cdot C'(1+x^2)^{\delta(1+\eps)l}
        \leq
        C''(1+x^2)^{\delta(1+\eps)l+\frac{1}{2}}\,.
    \end{align}
    Choosing $\delta=\frac{2}{3}$ and $\eps=\frac{1}{2}$ now gives the
    claimed bounds \eqref{eq:BoundsOn1DAction}.
\end{proof}

Next we show how the above constructed $\Rl$-action on $\Rl$ can be
promoted to suitable $\Rl^n$-action on $\Rl^n$ based on the ideas from
\cite[Ex.~4.5]{waldmann:2007a}.
\begin{lemma}\label{lemma:nDAction}
    Let $\tau^1$ be a smooth polynomially bounded $\Rl$-action on
    $\Rl$ with support in $[-1,1]$, as constructed in
    Lemma~\ref{lemma:1DAction}. Furthermore, let $\eps>0$ and
    $\chi:\Rl\to\Rl$ be a smooth function which is equal to $1$ on
    $[-1,1]$, and has support in $[-1-\eps,1+\eps]$. Then
    \begin{align}
	\tau^n_x(y) :=
        (\tau^1_{x_1\cdot\chi(y_1)\cdots\chi(y_n)}(y_1),\ldots,\tau^1_{
          x_n\cdot\chi(y_1)\cdots\chi(y_n)}(y_n))
    \end{align}
    is a smooth polynomially bounded $\Rl^n$-action on $\Rl^n$, with
    support in $[-1-\eps,1+\eps]^{\times n}$.
\end{lemma}
\begin{proof}
    Let $I:=[-1,1]$ and $I_\eps:=[-1-\eps,1+\eps]$. If $y\notin
    I_\eps^{\times n}$, there exists $j\in\{1,\ldots,n\}$ such that
    $y_j\notin I_\eps$, and hence $\chi(y_1)\cdots\chi(y_n)=0$. Thus
    $\tau_x^n(y)=(\tau^1_0(y_1),\ldots,\tau^1_0(y_n))=y$ in this case,
    which shows that $\tau^n$ has support in $B$. As a composition of
    smooth functions $\tau^n$ is smooth. To show that it is an action,
    let $x,x',y\in\Rl^n$, $j\in\{1,\ldots,n\}$, and compute
    \begin{align*}
        \tau^n_x(\tau^n_{x'}(y))_j
        &=
        \tau^1_{x_j\cdot\chi(\tau^n_{x'}(y)_1)\cdots\chi(\tau^n_{x'}(y)_n)}(\tau^n_{x'}
        (y)_j)
        \\
        &=
        \tau^1_{x_j\cdot\chi(\tau^n_{x'}(y)_1)\cdots\chi(\tau^n_{x'}(y)_n)}(\tau^1_{
          x'_j\cdot\chi(y_1)\cdots\chi(y_n)}
        (y_j))
        \\
        &=
        \tau^1_{x_j\cdot\chi\big(\tau^1_{x'_1
            \chi(y_1)\cdots\chi(y_n)}(y_1)\big)\cdots\chi\big(\tau^1_{x'_n
            \chi(y_1)\cdots\chi(y_n)}(y_1)\big)
          +
          x'_j\cdot\chi(y_1)\cdots\chi(y_n)}
        (y_j)\,.
    \end{align*}
    This coincides with
    $\tau^n_{x+x'}(y)_j=\tau^1_{(x_j+x_j')\chi(y_1)\cdots\chi(y_n)}(y_j)$
    if
    \begin{align}\label{eq:chichi}
        \chi(y_1)\cdots\chi(y_n)
        =
	\chi\big(\tau^1_{x'_1
          \chi(y_1)\cdots\chi(y_n)}(y_1)\big)\cdots\chi\big(\tau^1_{x'_n
          \chi(y_1)\cdots\chi(y_n)}(y_1)\big)
        \,.
    \end{align}
    Assume some component $y_k$ does not lie in $I$. Then
    $\tau^1_{x'_k\chi(y_1)\cdots\chi(y_n)}(y_k)=y_k$ by the support
    properties of $\tau^1$. If, on the other hand, $y_k\in I$, then
    $\tau^1_{x'_k\chi(y_1)\cdots\chi(y_n)}(y_k)\in I$ as well, and
    since $\chi=1$ on $I$, we find also in this case
    $\chi(y_k)=\chi(\tau^1_{x'_k\chi(y_1)\cdots\chi(y_n)}(y_k))$. Hence
    \eqref{eq:chichi} holds for all $x,y,y'$, and it follows that
    $\tau^n$ is an $\Rl^n$-action.

    It remains to verify the bounds \eqref{eq:BoundOnDMuNuTau} on the
    derivatives of $\tau^n$, i.e.\ we have to estimate
    $\partial_x^\mu\partial_y^\nu
    \tau^1_{x_j\chi(y_1)\cdots\chi(y_n)}(y_j)$. In comparison to
    $\partial_x^\mu\partial_y^\nu \tau^1_{x_j}(y_j)$, the
    $y$-derivatives produce finitely many extra factors of $x_j$ and
    derivatives of $\chi(y_1)\cdots\chi(y_n)$, and the $x$-derivatives
    produce extra factors of $\chi(y_1)\cdots\chi(y_n)$. All
    $y$-dependence can be uniformly estimated because of the compact
    support of (the derivatives of) $\chi$. So we arrive at a finite
    sum of the form
    \begin{align*}
	|\partial_x^\mu\partial_y^\nu\tau^n_x(y)_j|
        &\leq
        \sum_{\nu'\leq\nu,\mu'\leq\mu} c_{\nu'\mu'} |x_j|^{s(\nu',\mu')}\,
        |(\partial_x^{\mu'}\partial_y^{\nu'}\tau^1)_{x_j\chi(y_1)\cdots\chi(y_n)}(y_j)|
        \,,
    \end{align*}
    with $s(\nu',\mu')\leq|\nu|$. For $y_j\in I$, the derivatives of
    $\tau^1$ can now be estimated with \eqref{eq:BoundsOn1DAction},
    and $(1+x_j^2)\leq(1+\|x\|^2)$. For $y_j\in I_\eps\backslash I$,
    we can use the invariance $\tau_t(y_j)=y_j$, $t\in\Rl$, and
    $|y_j|\leq 1+\eps$, to estimate the derivatives of $\tau^1$. This
    shows that if $\tau^1$ was of order $\boldsymbol{b}$, then
    $\tau^n$ is of order at most $b_l+l<\infty$.
\end{proof}

After these constructions, it is now easy to show the existence of
smooth polynomially bounded $\Rl^n$-actions supported in arbitrarily
small regions.
\begin{theorem}
    Let $K\subset\Rl^n$ be open. Then there exist non-trivial smooth
    polynomially bounded $\Rl^n$-actions on $\Rl^n$, with support in
    $K$.
\end{theorem}
\begin{proof}
    In Lemma~\ref{lemma:nDAction}, we have constructed a non-trivial
    smooth polynomially bounded $\Rl^n$-action $\tau^n$ with support
    in a cube $[-r,r]^{\times n}$ centered at the origin. Clearly, the
    polynomial estimates are at most rescaled by affine
    transformations of $\mathbb{R}^n$ which allows to squeeze and move
    the support into any given compact subset.
\end{proof}


\bibliographystyle{plain}

\end{document}